\documentclass[11pt]{article}

\usepackage[T1]{fontenc}
\usepackage[utf8]{inputenc}
\usepackage{xcolor}
\usepackage{calc}
\usepackage{amsmath}
\usepackage{amsthm}
\usepackage{amssymb}
\usepackage[all]{xy}
\usepackage{tikz}
\usetikzlibrary{matrix, arrows, decorations.pathmorphing}
\usetikzlibrary{cd} % per diagrammi commutativi semplici

\usepackage{geometry}
\makeatletter
\theoremstyle{remark}
\newtheorem*{notation*}{\protect\notationname}
\theoremstyle{plain}
\newtheorem{thm}{\protect\theoremname}[section]
\theoremstyle{definition}
\newtheorem{defn}[thm]{\protect\definitionname}
\theoremstyle{definition}
\newtheorem{example}[thm]{\protect\examplename}
\theoremstyle{remark}
\newtheorem{rem}[thm]{\protect\remarkname}
\theoremstyle{plain}
\newtheorem{lem}[thm]{\protect\lemmaname}
\theoremstyle{plain}
\newtheorem{prop}[thm]{\protect\propositionname}
\theoremstyle{definition}

\theoremstyle{remark}
\newtheorem*{claim*}{\protect\claimname}
\theoremstyle{plain}
\newtheorem{cor}[thm]{\protect\corollaryname}
\theoremstyle{definition}
\newtheorem*{xca*}{\protect\exercisename}
\theoremstyle{remark}

\newtheorem{construction}[thm]{Construction}
\newtheorem{fact}[thm]{Fact}
\newtheorem{notation}[thm]{Notation}

\usepackage[bookmarksnumbered]{hyperref}

 \theoremstyle{definition}

\makeatother

\usepackage{babel}
\providecommand{\claimname}{Claim}
\providecommand{\corollaryname}{Corollary}
\providecommand{\definitionname}{Definition}
\providecommand{\examplename}{Example}
\providecommand{\exercisename}{Exercise}
\providecommand{\lemmaname}{Lemma}
\providecommand{\notationname}{Notation}
\providecommand{\notename}{Note}
\providecommand{\propositionname}{Proposition}
\providecommand{\remarkname}{Remark}
\providecommand{\theoremname}{Theorem}

\global\long\def\S{\mathcal{S}}%
\global\long\def\Pn{\mathbb{P}^{n}}%
\global\long\def\PN{\mathbb{P}^{N}}%
\global\long\def\Pu{\mathbb{P}^{1}}%
\global\long\def\Pd{\mathbb{P}^{2}}%
\global\long\def\ot{\otimes}%
\global\long\def\A{\mathbb{A}}%
\global\long\def\cc{\mathbb{C}}%
\global\long\def\M{\mathcal{M}}%
\global\long\def\g{\Gamma}%
\global\long\def\op{\oplus}%
\global\long\def\os{\mathcal{O}_{\S}}%
\global\long\def\F{\mathcal{F}}%
\global\long\def\L{\mathcal{L}}%
\global\long\def\C{\mathcal{C}}%
\global\long\def\oc{\mathcal{O}_{\C}}%
\global\long\def\iso{\simeq}%
\global\long\def\N{\mathbb{N}}%
\global\long\def\Z{\mathbb{Z}}%
\global\long\def\Q{\mathbb{Q}}%
\global\long\def\D{\mathrm{Div}}%
\global\long\def\im{\mathrm{Im}}%
\global\long\def\mm{\mathfrak{m}}%
\global\long\def\I{\mathcal{I}}%
\global\long\def\rg{\mathrm{rk}}%
\global\long\def\ns{\mathrm{NS}}%
\global\long\def\bl{\mathrm{Bl}}%
\global\long\def\x{\chi}%
\global\long\def\li{\equiv}%
\global\long\def\supp{\mathrm{Supp}}%
\global\long\def\K{\mathcal{K}}%
\global\long\def\bs{\mathrm{Bs}}%
\global\long\def\ks{\K_{\S}}%
\global\long\def\num{\mathrm{Num}}%
\global\long\def\gl{\mathrm{GL}}%
\global\long\def\pgl{\mathrm{PGL}}%
\global\long\def\pc{\mathbb{P}_{\C}}%
\global\long\def\xt{\x_{\mathrm{top}}}%
\global\long\def\ff{\mathbb{F}_{n}}%
\global\long\def\pic{\mathrm{Pic}}%
\global\long\def\deg{\mathrm{deg}}%
\global\long\def\su{\subseteq}%
\global\long\def\ra{\dashrightarrow}%

\begin{document}
\title{Ruled and rational surfaces and their models}
\author{Giacomo Graziani}
\date{}
\maketitle
\tableofcontents{}

\begin{abstract}
    
\end{abstract}
One of the most powerful ideas in the study and classification
of algebraic varieties is the notion of a \emph{model}—that is, to
single out an object, in the appropriate isomorphism class, with nice properties. This survey aims to define and study suitable models of rational and, more generally, ruled surfaces in the smooth complex case, and to use them to study such surfaces.

A key consequence of Castelnuovo's contractibility Theorem is that
every surface $\S$ admits a \emph{relatively minimal model}: a birational
morphism $\S\to\S'$ where $\S'$ has no $\left(-1\right)$-curves.
We will see that the relatively minimal models of surfaces ruled over
an irrational curve are exactly the $\Pu$-bundles, i.e., the projectivization
of rank 2 vector bundles over that curve. It is also possible
to introduce a numerical invariant $e$ for such bundles, and this
invariant will be shown to completely classify the indecomposable relatively minimal models over an elliptic curve.

Nevertheless, the notion of a relatively minimal model turns out to
be not completely satisfactory since, in general, such a model is not unique. To solve this problem, we define \emph{minimal models}
as surfaces with a nef canonical system; it is straightforward to see that such a model, when it exists, is unique. Dealing
with minimal models requires the introduction of a new biregular invariant
of a smooth surface, namely the notion of \emph{type}. The existence
of this invariant provides an analogue of Kodaira's classification
for smooth complex surfaces. From these results, we can deduce
many important statements, such as the existence of a minimal model
for every non-ruled surface and an extended description of relatively
minimal models in the general ruled case. This leads to a broad
characterization of ruled and rational surfaces, which contains as
particular cases both \emph{Castelnuovo's criterion for rationality}
and an affirmative answer to \emph{Lüroth's problem for complex surfaces}.

In the final part of this survey, we use the theory we have developed to study the embeddings of ruled surfaces into projective spaces.
Namely, we will show that for every pair of integers $k,n$ with $k>n\ge0$
we can embed $\ff$ as a rational scroll of degree $2k-n$ in $\mathbb{P}^{2k-n+1}$. We also show that for every $n\ge e+3\ge2$ there exists an
elliptic scroll in $\mathbb{P}^{2n-e-1}$ of degree $2n-e$ and invariant $e$. These results, in analogy with the case of curves, suggest
that the possible degrees of an embedding are related to the rationality
of the surface. After showing that the degree of an irreducible non-degenerate
surface in $\mathbb{P}^{n}$ is bounded below by $n-1$, we give a
complete classification of smooth surfaces that can be embedded in
$\mathbb{P}^{n}$ with degree $n-1$ or $n$. As it turns out, all such surfaces are rational.

The material presented in this survey follows the classical treatments
of \cite{Beau} and \cite{HA}, while we refer to \cite{Ke} for background material.

\section{Technical Tools}

\subsection{Ampleness and very ampleness}
\begin{lem}
\label{lem:extend1}Let $X$ be an affine variety and $f\in\cc\left[X\right]$.
Let $U_{f}$ be the open set defined by $f\neq0$ and $\F$ a quasi-coherent
sheaf on $X$. Then

\begin{enumerate}
\item if $s\in\g\left(X,\F\right)$ is a global section such that $s_{|U_{f}}=0$,
then there exists an integer $n>0$ with $f^{n}s=0$;
\item if $t\in\g\left(U_{f},\F\right)$, there exists an integer $n>0$
such that $f^{n}t$ extends to a global section of $\F$.
\end{enumerate}
\end{lem}

\begin{proof}
Let us note that, since $\F$ is quasi-coherent on $X,$ we can find
open sets $\left(U_{g_{i}}\right)$ with $g_{i}\in\cc\left[X\right]$
such that they form a $\emph{finite}$ open cover of $X$ and $\F_{|U_{g_{i}}}=\tilde{M_{i}}$,
where $M_{i}$ is a $\cc\left[X\right]_{g_{i}}$-module. Moreover
$U_{g_{i}}\cap U_{g_{j}}=U_{g_{i}g_{j}}$. Let $s$ be a global section
with $s_{|U_{f}}=0$, we denote with $s_{i}$ the restriction $s_{|U_{g_{i}}}$.
For every $i$, since $\F_{|U_{f}\cap U_{g_{i}}}=\tilde{\left(M_{i}\right)_{f}}$,
the image of $s_{i}$ in $\left(M_{i}\right)_{f}$ is zero, thus there
exists $n_{i}\in\N$ such that $f^{n_{i}}s_{i}=0$. Since the open
cover is finite, we can find $n$ such that $f^{n}s_{i}=0$ for every
$i$, and this proves $1)$. To prove $2)$, we note that, for every
$i$, there exists $t_{i}\in M_{i}=\g\left(U_{g_{i}},\F\right)$ such
that the image of $t_{i}$ on $U_{f}\cap U_{g_{i}}=U_{fg_{i}}$ is
$f^{n_{i}}t$ for some $n_{i}$. As in the preceding point, we can
take $n$ working for every $i$. On the intersection $U_{g_{i}g_{j}}$
we have two sections $t_{i}$ and $t_{j}$ which agree on $U_{fg_{i}g_{j}}$,
from point $1)$ there exists $m$ such that $f^{m}\left(t_{i}-t_{j}\right)=0$
on $U_{g_{i}g_{j}}$ and we can take $m$ big enough to work for every
$i,j$. Finally, local sections $f^{m}t_{i}\in\g\left(U_{g_{i}},\F\right)$
glue to give a global section whose restriction to $U_{f}$ is $f^{n+m}t$.
\end{proof}
\begin{lem}
\label{lem:extend2}Let $X$ be an algebraic variety, $\L\in\pic\left(X\right)$
and $f\in\g\left(X,\L\right)$. We set $X_{f}=\left\{ x\in X\,\,|\,\,f_{x}\notin\mm_{x}\cdot\L_{x}\right\} .$
For every quasi-coherent sheaf $\F$ on $X$:

\begin{enumerate}
\item if $X$ is quasi-compact and $s\in\g\left(X,\F\right)$ is a section
such that $s_{|X_{f}}=0$, then there exists $n>0$ with $f^{n}s=0$
as a global section of $\F\ot\L^{n}$;
\item moreover, if $\L$ is locally free over an open cover $\left(U_{i}\right)$
such that the intersections  $U_{i}\cap U_{j}$ are quasi-compact,
for every $t\in\g\left(X_{f},\F\right)$, there exists an integer
$n>0$ such that $f^{n}t\in\g\left(X_{f},\F\ot\L^{n}\right)$ extend
to a global section.
\end{enumerate}
\end{lem}

\begin{proof}
It is a straightforward consequence of Lemma \ref{lem:extend1}. We
simply remark that the hypothesis $U_{i}\cap U_{j}$ quasi-compact
is used to apply point $1)$.
\end{proof}
\begin{prop}[{\cite[Exercise 5.5]{HA}}]
\label{prop:pushforwardcoherent}If $f:X\to Y$ is a finite morphism
between projective varieties and $\F$ is a $\emph{coherent}$ sheaf
on $X$, then $f_{\ast}F$ is a coherent sheaf on $Y$. In particular,
if $i:X\hookrightarrow Y$ is a closed embedding, then $i_{\ast}F$
is coherent for every coherent $\mathcal{O}_{X}$-module $\F$.
\end{prop}

\begin{proof}
Let $\left(U_{i}\right)$ be an affine cover of $Y$, then $\left(f^{-1}\left(U_{i}\right)\right)$
is an affine cover of $X$ and thus, for every $i$, $\F_{|f^{-1}\left(U_{i}\right)}\iso\tilde{M_{i}}$
for a finitely generated $\cc\left[f^{-1}\left(U_{i}\right)\right]$-module
$M_{i}$. Then, in view of \cite[Lemma 5.6.1]{Ke}, we have
\[
f_{\ast}F_{|U_{i}}\iso\left(M_{i}\ot_{\cc\left[U_{i}\right]}\cc\left[f^{-1}\left(U_{i}\right)\right]\right)\tilde{}.
\]
Thus $f_{\ast}F$ is coherent since $\cc\left[f^{-1}\left(U_{i}\right)\right]$
is a finitely generated $\cc\left[U_{i}\right]$-module for every
$i$. Moreover it is clear that a closed embedding is a finite morphism.
\end{proof}
We have all the tools to prove
\begin{thm}[Serre]
\label{thm:veryampleimplicaample}Let $X\subseteq\PN$ be a projective
variety and let $\F$ be a coherent $\mathcal{O}_{X}$-module. Then, for $n\gg0$,
the sheaf $\F\left(n\right)$ is generated by global sections.
\end{thm}

\begin{proof}
Let $j$ be the inclusion $X\hookrightarrow\PN$, we have that
\[
j^{\ast}\left(\mathcal{O}\left(1\right)\right)=\mathcal{O}_{X}\left(1\right).
\]
Moreover, in view of Proposition \ref{prop:pushforwardcoherent},
$j_{\ast}F$ is coherent and in view of projection formula we deduce
\begin{eqnarray*}
j_{\ast}\left(\F\left(1\right)\right) & = & j_{\ast}\left(\F\ot\mathcal{O}_{X}\left(1\right)\right)=j_{\ast}\left(\F\ot j^{\ast}\left(\mathcal{O}\left(1\right)\right)\right)=\\
 & = & j_{\ast}\F\ot\mathcal{O}\left(1\right)=\left(j_{\ast}\F\right)\left(1\right)
\end{eqnarray*}
from which we get that $j_{\ast}\left(\F\left(n\right)\right)=\left(j_{\ast}\F\right)\left(n\right)$
for every $n$. Clearly $\F\left(n\right)$ is generated by global
sections if and only if $\left(j_{\ast}\F\right)\left(n\right)$ is
generated by global sections, thus we reduce to the case when $X=\PN$.
Let $\left(U_{i}\right)$ be the standard affine open cover of $\PN$,
then $\F_{|U_{i}}=\tilde{M_{i}}$ where $M_{i}$ is a finitely generated
$A_{i}$-module, where $A_{i}=\cc\left[\frac{x_{0}}{x_{i}},\dots,\frac{x_{N}}{x_{i}}\right]$.
For every $i$, let $s_{ij}$, for $j=1,\dots,k_{i}$ be generators
for $M_{i}$ over $A_{i}$. In view of Lemma \ref{lem:extend2} there
exists $n$ such that $x_{i}^{n}s_{ij}$ extends to a global section
$t_{ij}$ of $\F\left(n\right)$ for every $i$. Thus $\F\left(n\right)$
corresponds to an $A_{i}$-module $M_{i}'$ over $U_{i}$ and the
product with $x_{i}^{n}$ gives an epimorphism $\F_{|U_{i}}\to\F\left(n\right)_{|U_{i}}$
which, in view of Nakayama's lemma, yields an isomorphism $M_{i}\iso M'_{i}$.
Then the sections $t_{ij}$ generates $\g\left(X,\F\left(n\right)\right)$
and the theorem is proved.
\end{proof}
\begin{defn}
Let $X$ be a projective variety, then we say that a sheaf $\L\in\pic\left(X\right)$
is $\emph{very ample }$ if it is generated by global section and
the induced morphism is a closed immersion.
\end{defn}

\begin{rem}
\label{rem:sumveryample}Given a projective variety $X$ and $\L,\M\in\pic\left(X\right)$
with $\L$ very ample and $\M$ generated by global sections, then
$\L\ot\M$ is still very ample. To see this, let us note that the
map associated to $\L\ot\M$ is the composition
\[
X\overset{\Delta}{\longrightarrow}X\times X\overset{\varphi_{\L}\times\varphi_{\M}}{\longrightarrow}\mathbb{P}^{n}\times\mathbb{P}^{m}\overset{s_{n,m}}{\longrightarrow}\PN
\]
where $s_{n,m}$ denotes the Segre embedding.
\end{rem}

\begin{prop}
\label{prop:characterizationample}Let $X$ be a projective variety
and let $\L\in\pic\left(X\right)$. Then the following are equivalent:

\begin{enumerate}
\item there exists $n>0$ such that $\L^{n}$ is very ample;
\item for every coherent $\mathcal{O}_{X}$-module $\F$ there exists $n_{0}\in\N$
such that $\F\ot\L^{n}$ is generated by global sections for every
$n\ge n_{0}$.
\end{enumerate}
\end{prop}

\begin{proof}
$\,$

\begin{description}
\item [{$1)\Longrightarrow2)$}] Let $\M=\L^{n}$ be very ample and let
$i:X\hookrightarrow\PN$ be the associated closed embedding, then
$i^{\ast}\left(\mathcal{O}\left(1\right)\right)=\M$. Let $\F$ be a coherent
sheaf on $X$ then, in view of Theorem \ref{thm:veryampleimplicaample},
there exists $m$ such that $\F\left(m\right)$ is generated by global
sections.
\item [{$2)\Longrightarrow1)$}] Let $x\in X$ and $U$ an affine neighborhood
of $x$ such that $\L_{|U}$ is free, and let $Y=U^{C}$. Under our
hypothesis, if $\I$ is the ideal sheaf associated to $Y$, there
exists $n$ such that $\I\ot\L^{n}$ is generated by global sections.
In particular there exists $s\in\g\left(X,\I\ot\L^{n}\right)$ such
that $s_{x}\notin\mm_{x}\cdot\left(\I\ot\L^{n}\right)_{x}$. Regarding
$\I\ot\L^{n}$ as a subsheaf of $\L^{n}$, we consider $s$ as a section
of $\L^{n}$. Let
\[
U_{s}=\left\{ p\in X\,\,|\,\,s_{p}\notin\mm_{p}\cdot\L_{p}^{n}\right\} ,
\]
thus $x\in U_{s}\su U$. For our choice of $U$, we can identify $\L_{|U}$
with $\mathcal{O}_{U}$ and $s$ will correspond to $f\in\g\left(U,\mathcal{O}_{U}\right)$.
We can repeat this construction to produce a finite affine open cover
$U_{s_{i}}$, corresponding to sections $s_{i}\in\g\left(X,\L^{n_{i}}\right)$.
We can replace the sections $s_{i}$ with suitable powers and get
$s_{1},\dots,s_{k}\in\g\left(X,\L^{n}\right)$ without changing the
open cover $\left(U_{s_{i}}\right)$. For every $i=1,\dots,k$, let
$\left\{ b_{ij}\,\,|\,\,j=1,\dots,k_{i}\right\} $ be the generators
of $\cc\left[U_{i}\right]$ as a $\cc$-algebra. In view of Lemma
\ref{lem:extend2}, there exists $m$ such that $s_{i}^{m}b_{ij}$
extends to a global section $c_{ij}\in\g\left(X,\L^{nm}\right)$.
The sheaf $\L^{nm}$ is generated by the sections $s_{i}^{m}$, thus
it is generated by global sections. Let us consider the morphism $\phi:X\to\PN$
defined by $\left\{ s_{i}^{m}\,\,|\,\,i=1,\dots,k\right\} $ and $\left\{ c_{ij}\,\,|\,\,i=1,\dots,k,\,j=1,\dots,k_{i}\right\} $.
Let $\left\{ x_{i}\,\,|\,\,i=1,\dots,k\right\} $ and $\left\{ x_{ij}\,\,|\,\,i=1,\dots,k,\,j=1,\dots,k_{i}\right\} $
be coordinates on $\PN$ corresponding to our sections. For every
$i=1,\dots,k$, let $V_{i}\su\PN$ be the open subset given by $\left(x_{i}\neq0\right)$.
Then $\phi^{-1}\left(V_{i}\right)=U_{i}$ and we have a surjective
$\cc$-algebra homomorphism
\begin{eqnarray*}
\cc\left[\left\{ y_{i}\right\} ,\left\{ y_{ij}\right\} \right] & \to & \cc\left[U_{i}\right]\\
y_{ij} & \mapsto & b_{ij}=\frac{c_{ij}}{s_{i}^{m}}.
\end{eqnarray*}
Then $X$ is mapped isomorphically onto a closed subvariety of $\PN$
by $\phi$ and so $\L^{nm}$ is very ample.
\end{description}
\end{proof}
\begin{defn}
Let $X$ be a projective variety. We say that $\L\in\pic\left(X\right)$
is $\emph{ample}$ if one of the equivalent conditions of Proposition
\ref{prop:characterizationample} holds.
\end{defn}

\begin{rem}
\label{rem:genglobsections}Let $X$ be a projective variety and $\L\in\pic\left(X\right)$.
Let us suppose that, for every $x\in X$, we have $H^{1}\left(\mm_{x}\cdot\L_{x}\right)=0$.
Then, from the exact sequence
\[
0\to\mm_{x}\cdot\L_{x}\to\L\to\L\ot\cc\left(x\right)\to0,
\]
we deduce that $\L$ is generated by global sections in a neighborhood
of $x$.
\end{rem}

\begin{thm}
\label{prop:amplemultiple}Let $X$ be a projective variety and $\L\in\pic\left(X\right)$.
The following are equivalent:

\begin{enumerate}
\item $\L$ is ample;
\item for every coherent $\mathcal{O}_{X}$-module $\F$ there exists $n_{0}\in\N$
such that, for every $n\ge n_{0}$ and for every $i>0$, we have
\[
H^{i}\left(X,\F\ot\L^{n}\right)=0.
\]
\end{enumerate}
\end{thm}

\begin{proof}
$\,$

\begin{description}
\item [{$1)\Longrightarrow2)$}] If $\L$ is ample, let $\M=\L^{n}$ be
a very ample multiple. Then there exists a closed embedding $i:X\hookrightarrow\PN$
such that $i^{\ast}\left(\mathcal{O}\left(1\right)\right)=\M$. In view of
\cite[Theorem 9.3.1]{Ke} we conclude that
\[
H^{i}\left(X,\F\ot\M^{k}\right)=0
\]
definitively for $k\gg0$ and for every $i>0$.
\item [{$2)\Longrightarrow1)$}] Let $\F$ be a coherent $\mathcal{O}_{X}$-module,
we want to see that there exists $n\in\N$ such that $\F\ot\L^{n}$
is generated by global sections. In view of Remark \ref{rem:genglobsections}
we only need to see that for every $x\in X$ we have
\begin{equation}
H^{1}\left(\mm_{x}\cdot\left(\F\ot\L^{n}\right)\right)=0\label{eq:gensezglobx}
\end{equation}
for some $n$ depending on $x$, but this follows from the hypothesis.
Moreover, in view of Nakayama's lemma, formula ($\ref{eq:gensezglobx}$)
holds in an open neighborhood of $x$, and thus, by quasi-compactness,
we can find $n$ big enough such that $\left(\ref{eq:gensezglobx}\right)$
holds for every $x\in X$.
\end{description}
\end{proof}
\begin{cor}
\label{cor:BertiniII}Let $X$ be a smooth projective variety with
$\dim X\ge2$ and $H\in\D\left(X\right)$ an ample effective divisor.
Then $\supp H$ is connected.
\end{cor}

\begin{proof}
We can suppose the divisor $H$ is very ample with associated invertible
sheaf $\mathcal{O}\left(1\right)$. Let $Y=\supp H$ and, for every $n>0$,
let $Y^{n}$ be the variety with $\supp Y^{n}=\supp Y$ and defined
by the ideal sheaf $\mathcal{O}_{X}\left(-n\right)$. Then, by definition,
we have an exact sequence
\[
0\to\mathcal{O}_{X}\left(-n\right)\to\mathcal{O}_{X}\to\mathcal{O}_{Y^{n}}\to0.
\]
In view of Theorem \ref{prop:amplemultiple} and Serre duality, for
$n\gg0$ we have
\[
H^{1}\left(X,\mathcal{O}_{X}\left(-n\right)\right)=0
\]
and hence a surjection
\[
\cc\to\cc^{\#\pi_{0}\left(Y\right)}\to0.
\]
Therefore $\#\pi_{0}\left(Y\right)=1$ and $Y$ is connected.
\end{proof}
\begin{cor}
\label{lem:finiteample}Let $f:X\to Y$ be a finite morphism and let
$\L\in\pic\left(Y\right)$ be an ample sheaf. Then $f^{\ast}\L$ is
ample.
\end{cor}

\begin{proof}
We'll make use of Theorem \ref{prop:amplemultiple}. Let $\F$ be
a coherent $\mathcal{O}_{X}$-module, then, in view of the finitess of $f$
and Proposition \ref{prop:pushforwardcoherent}, there exists $n_{0}$
such that, for every $n\ge n_{0}$ and for every $i>0$, we have
\[
H^{i}\left(Y,f_{\ast}\F\ot\L^{n}\right)=0.
\]
In view of \cite[Corollary 8.2.3]{Ke} and projection formula,
we conclude.
\end{proof}
\begin{lem}
\label{lem:trucchettocoomologia}Let $\S$ be a smooth surface and
$D\in\D\left(\S\right)$ an effective divisor such that $\bs|D|_{|D}=\emptyset$
and the restriction sequence
\[
H^{0}\left(\S,\os\left(D\right)\right)\to H^{0}\left(D,\os\left(D\right)_{|D}\right)\to0
\]
 is exact. Then $\bs|D|=\emptyset$.
\end{lem}

\begin{proof}
Clearly $\bs|D|\su\supp\left(D\right)$. Let $x\in\supp\left(D\right)$,
then, under our assumptions, there exists $D_{0}\in|D|_{|D}$ such
that $x\notin\supp\left(D_{0}\right)$. In view of the surjectivity
of the restriction we conclude that there exists $\tilde{D}_{0}\in|D|$
with $x\notin\supp\left(\tilde{D}_{0}\right)$.
\end{proof}

\subsection{Intersection theory and Riemann-Roch for surfaces}

The aim of this section is to give a cohomological treatment of intersection
theory on smooth surfaces.
\begin{defn}
\label{def:intprod}Let $\C_{1},\C_{2}$ be distinct prime divisors
on a surface $\S$, and let $x\in\C_{1}\cap\C_{2}$. We define the
intersection multiplicity of $x$ in $\C_{1}\cap\C_{2}$ as
\[
\mu\left(x,\C_{1}\cap\C_{2}\right)=\dim_{\cc}\frac{\mathcal{O}_{\S,x}}{\left(f,g\right)}
\]
where $f,g$ are local equations in $x$ for $\C_{1}$ and $\C_{2}$
respectively. We set
\[
\C_{1}\cdot\C_{2}=\sum_{x\in C_{1}\cap C_{2}}\mu\left(x,\C_{1}\cap\C_{2}\right),
\]
called the $\emph{intersection product}$ of $\C_{1}$ and $\C_{2}$.
\end{defn}

\begin{lem}
\label{lem:disuguaglianzamolteplicita}Let $\C_{1},\C_{2}\su\S$ be
distinct curves on a surface, then
\[
\mu\left(x,\C_{1}\cap\C_{2}\right)\ge\mu\left(x,\C_{1}\right)\mu\left(x,\C_{2}\right)
\]
for every $x\in\S$, in particular
\[
\C_{1}\cdot\C_{2}\ge\sum_{x\in\C_{1}\cap\C_{2}}\mu\left(x,\C_{1}\right)\mu\left(x,\C_{2}\right).
\]
\end{lem}

\begin{proof}
It is enough to check that
\[
\mu\left(x,\C_{1}\cap\C_{2}\right)\ge\mu\left(x,\C_{1}\right)\mu\left(x,\C_{2}\right)
\]
for every $x\in\C_{1}\cap\C_{2}$. Let $x\in\C_{1}\cap\C_{2}$ and
$\sigma:\S'=\bl_{x}\left(\S\right)\to\S$ be the blow-up in $x$ with
exceptional divisor $E$. We can write
\begin{eqnarray*}
\C_{1}\cdot\C_{2} & = & \sigma^{\ast}\C_{1}\cdot\sigma^{\ast}\C_{2}\\
 & = & \left(\tilde{\C_{1}}+\mu\left(x,\C_{1}\right)E\right)\cdot\left(\tilde{\C_{2}}+\mu\left(x,\C_{2}\right)E\right)\\
 & = & \tilde{\C_{1}}\cdot\tilde{\C_{2}}+\mu\left(x,\C_{1}\right)\mu\left(x,\C_{2}\right)\\
 & = & \sum_{y\in E}\mu\left(y,\tilde{\C_{1}}\cap\tilde{\C_{2}}\right)+\sum_{y\not\in E}\mu\left(y,\tilde{\C_{1}}\cap\tilde{\C_{2}}\right)+\mu\left(x,\C_{1}\right)\mu\left(x,\C_{2}\right)\\
 & \ge & \sum_{y\not\in E}\mu\left(y,\tilde{\C_{1}}\cap\tilde{\C_{2}}\right)+\mu\left(x,\C_{1}\right)\mu\left(x,\C_{2}\right)\\
 & = & \sum_{y\neq x}\mu\left(y,\C_{1}\cap\C_{2}\right)+\mu\left(x,\C_{1}\right)\mu\left(x,\C_{2}\right).
\end{eqnarray*}
Since
\begin{eqnarray*}
\C_{1}\cdot\C_{2} & = & \sum_{z\in C_{1}\cap C_{2}}\mu\left(z,\C_{1}\cap\C_{2}\right)\\
 & = & \mu\left(x,\C_{1}\cap\C_{2}\right)+\sum_{y\neq x}\mu\left(y,\C_{1}\cap\C_{2}\right),
\end{eqnarray*}
we conclude that $\mu\left(x,\C_{1}\cap\C_{2}\right)\ge\mu\left(x,\C_{1}\right)\mu\left(x,\C_{2}\right)$.
\end{proof}
\begin{rem}
In the setting of Definition \ref{def:intprod}, let us consider the
skycraper sheaf
\[
\mathcal{O}_{C_{1}\cap C_{2}}=\frac{\os}{\os\left(-C_{1}\right)+\os\left(-C_{2}\right)},
\]
then it is clear that $C_{1}\cdot C_{2}=h^{0}\left(\S,\mathcal{O}_{C_{1}\cap C_{2}}\right)$.
\end{rem}

\begin{defn}
\label{def:intprodpic}Let $\L,\M\in\pic\left(\S\right)$. We define
\[
\L\cdot\M=\x\left(\os\right)-\x\left(\L^{-1}\right)-\x\left(\M^{-1}\right)+\x\left(\L\ot\M\right),
\]
called the $\emph{intersection product}$ on $\pic\left(\S\right)$.
\end{defn}

\begin{prop}
\label{prop:prodottogrado}Let $\C\su\S$ be a smooth irreducible
curve. Then, for every $\L\in\pic\left(\S\right)$ we have
\[
\os\left(D\right)\cdot\L=\deg\left(\L_{|\C}\right).
\]
\end{prop}

\begin{proof}
Let us tensorize the exact sequence
\[
0\to\os\left(-\C\right)\to\os\to\oc\to0
\]
with $\L^{-1}$ and get the exact sequence
\[
0\to\L^{-1}\left(-\C\right)\to\L^{-1}\to\L^{-1}\ot\oc\to0.
\]
Thus we have
\[
\begin{array}{c}
\x\left(\os\right)=\x\left(\oc\right)+\x\left(\os\left(-\C\right)\right)\\
\x\left(\L^{-1}\right)=\x\left(\L^{-1}\left(-\C\right)\right)+\x\left(\L^{-1}\ot\oc\right)
\end{array},
\]
therefore, in view of Riemann-Roch for curves and Definition \ref{def:intprodpic},
\begin{eqnarray*}
\os\left(\C\right)\cdot\L & = & \x\left(\oc\right)-\x\left(\L^{-1}\ot\oc\right)=\\
 & = & -\deg\left(\L^{-1}\ot\oc\right)=\deg\left(\L_{|\C}\right).
\end{eqnarray*}
\end{proof}
\begin{lem}
\label{lem:divisordifferencesmooth}Let $D\in\D\left(\S\right)$ and
$H$ a very ample divisor. Then there exists $n\ge0$ such that $D+nH$
is very ample. In particular we have $D\li A-B$, where $A,B$ are
smooth curves with $A\li D+nH$ and $B\li nH$.
\end{lem}

\begin{proof}
Let $d$ such that $\os\left(D+dH\right)$ is generated by global
sections. Therefore, in view of Remark \ref{rem:sumveryample}, the
divisors $D+\left(d+1\right)H$ and $\left(d+1\right)H$ are very
ample, thus we can apply \cite[Theorem 6.6.1]{Ke} to see that
there exist $A,B$ smooth curves with $A\li D+\left(d+1\right)H$
and $B\li\left(d+1\right)H$, therefore $D\li A-B$.
\end{proof}
\begin{thm}
\label{thm:bilinearity}The intersection product defines a bilinear
form on $\pic\left(\S\right)$ and, if $D,D'$ are two distinct irreducible
curves on $\S,$ then
\[
\os\left(D\right)\cdot\os\left(D'\right)=D\cdot D'.
\]
\end{thm}

\begin{proof}
Let us set
\[
f\left(\L_{1},\L_{2},\L_{3}\right)=\L_{1}\cdot\left(\L_{2}\ot\L_{3}\right)-\L_{1}\cdot\L_{2}-\L_{1}\cdot\L_{3}.
\]
A simple computation shows that $f$ is symmetric in the three variables
and, moreover, Proposition \ref{prop:prodottogrado} shows that, if
there exists a smooth curve $\C\su\S$ such that $\L_{i}=\os\left(\C\right)$
for some $i=1,2,3$, we have $f\left(\L_{1},\L_{2},\L_{3}\right)=0$.
In view of Lemma \ref{lem:divisordifferencesmooth}, we can write
$\L_{2}=\os\left(A-B\right)$ where $A,B$ are smooth curves, then
we get $\L_{1}\cdot\L_{2}=\L_{1}\cdot\os\left(A\right)-\L_{1}\cdot\os\left(B\right)$
and hence the linearity in the first argument. Since the intersection
product is obviously symmetric, the bilinearity follows.

Let now $f,g\in\mathcal{O}_{x}$ be local equations for $D,D'$ in a neighborhood
of the point $x$. Let us consider the sequence
\begin{equation}
0\to\mathcal{O}_{x}\overset{\left(g,-f\right)}{\longrightarrow}\mathcal{O}_{x}^{2}\overset{\left(f,g\right)}{\longrightarrow}\mathcal{O}_{x}\to\frac{\mathcal{O}_{x}}{\left(f,g\right)}\to0.\label{eq:exactsequenceintproduct}
\end{equation}
Since $f,g$ are coprime in the factorial ring $\mathcal{O}_{x}$, if, for
some $a,b\in\mathcal{O}_{x}$ we have $af=bg$, then there exists $t\in\mathcal{O}_{x}$
such that $a=tg$ and $b=tf$. This shows that the sequence (\ref{eq:exactsequenceintproduct})
is exact. Moreover, if $s\in H^{0}\left(\S,\os\left(D\right)\right)$
and $s'\in H^{0}\left(\S,\os\left(D'\right)\right)$ are non-zero
sections vanishing on $D$ and $D'$ respectively, then we have shown
that
\[
0\to\os\left(-D-D'\right)\overset{\left(s',-s\right)}{\longrightarrow}\os\left(-D\right)\op\os\left(-D'\right)\overset{\left(s,s'\right)}{\longrightarrow}\os\to\mathcal{O}_{D\cap D'}\to0
\]
is exact, then
\begin{eqnarray*}
D\cdot D' & = & h^{0}\left(\S,\mathcal{O}_{D\cap D'}\right)=\x\left(\mathcal{O}_{D\cap D'}\right)=\\
 & = & \x\left(\os\right)+\x\left(\os\left(-D-D'\right)\right)-\x\left(\os\left(-D\right)\right)-\x\left(\os\left(-D'\right)\right)\\
 & = & \os\left(D\right)\cdot\os\left(D'\right).
\end{eqnarray*}
\end{proof}
Let us see some immediate application:
\begin{lem}
\label{lem:beauvilleusefulremark}Let $D\in\D\left(\S\right)$ be
an effective divisor and $\C$ an irreducible curve with $\C^{2}\ge0$.
Then $D\cdot\C\ge0$.
\end{lem}

\begin{proof}
Write $D=D'+n\C$ where $n\ge0$ and $\C\not\su\supp\left(D'\right)$.
Then
\[
D\cdot\C=D\cdot D'+n\C^{2}\ge0.
\]
\end{proof}
\begin{example}
\label{ex:fibresquaredzero}Let $\C$ be a smooth curve, $f:\S\to\C$
a surjective morphism and let $F=f^{-1}\left[x\right]$ be a fibre
of $f$. Then there exists a divisor $D\in\pic\left(\C\right)$ such
that $\left[x\right]\li D$ but $x\notin\supp\left(D\right)$. Clearly
$F\li f^{\ast}D$, but this is a union of fibres disjoint from $F$,
thus $F^{2}=0$.
\end{example}

$\,$
\begin{example}[Intersection theory on $\Pd$]
\label{ex:inttheoryonplane}Every prime divisor of degree $d$ in
$\Pd$ is linearly equivalent to $dL$, where $L$ is a line. Thus,
since, clearly, for every line $L$ we have $L^{2}=1$, given two
prime divisors $C,C'$ of degrees $d$ and $d'$ respectively, we
get
\[
C\cdot C'=dL\cdot d'L'=dd',
\]
therefore we obtain Bezout's theorem.
\end{example}

$\,$
\begin{example}
\label{ex:inttheoryonquadric}Let $\S=\Pu\times\Pu$ and let $p_{1},p_{2}:\S\to\Pu$
be the projections. For a point $P\in\Pu$, we set $h_{i}=p_{i}^{-1}\left(P\right)\su\S$.
Since
\[
\S-\left(h_{1}\cup h_{2}\right)\iso\A^{2},
\]
every divisor on $S-\left(h_{1}\cup h_{2}\right)$ is principal and
thus $\pic\left(\S\right)=\Z\cdot h_{1}\op\Z\cdot h_{2}$ and, in
view of Example \ref{ex:fibresquaredzero}, we see that $h_{1}^{2}=h_{2}^{2}=0$
and $h_{1}\cdot h_{2}=1$. In particular the group $\pic\left(\S\right)$
is free.
\end{example}

\begin{thm}[Riemann-Roch for surfaces]
Let $\L\in\pic\left(\S\right)$, then
\[
\x\left(\L\right)=\x\left(\os\right)+\frac{1}{2}\left(\L^{2}-\L\cdot\K_{\S}\right).
\]
\end{thm}

\begin{proof}
In view of Serre duality, we compute
\begin{eqnarray*}
\L^{-1}\cdot\left(\L\ot\K_{\S}^{-1}\right) & = & \x\left(\os\right)-\x\left(\L\right)-\x\left(\L^{-1}\ot\K_{\S}\right)+\x\left(\K_{\S}\right)=\\
 & = & 2\x\left(\os\right)-2\x\left(\L\right),
\end{eqnarray*}
by the bilinearity of the intersection product we get
\[
2\x\left(\L\right)=2\x\left(\os\right)-\L^{2}+\L\cdot\K_{\S}.
\]
\end{proof}
\begin{cor}[Genus formula]
\label{cor:genusformula}Let $\C\su\S$ be an irreducible curve,
then
\[
h^{1}\left(\C,\oc\right)=1+\frac{1}{2}\left(\C^{2}+\C\cdot\K_{\S}\right).
\]
\end{cor}

\begin{proof}
Let us consider the exact sequence
\[
0\to\os\left(-\C\right)\to\os\to\oc\to0,
\]
then $\x\left(\os\right)=\x\left(\oc\right)+\x\left(\os\left(-\C\right)\right)$.
In view of Riemann-Roch, we deduce
\[
1-h^{1}\left(\C,\oc\right)=\x\left(\os\right)-\x\left(\os\left(-\C\right)\right)=-\frac{1}{2}\left(\C^{2}+\C\cdot\K_{\S}\right)
\]
and hence the corollary is proved.
\end{proof}
\begin{rem}
\label{rem:geometricarithmeticgenus}Let $\C$ be an irreducible curve
and $f:\tilde{\C}\to\C$ its normalization. Let us recall that the
$\emph{arithmetic genus}$ $p_{a}\left(\C\right)$ of $\C$ is defined
as $h^{1}\left(\C,\oc\right)$, while the $\emph{geometric genus}$
$g\left(\C\right)$ is defined as $g\left(\C\right)=p_{a}\left(\tilde{C}\right)$.
In this way, Corollary \ref{cor:genusformula} can be restated saying
that, for every irreducible curve $\C\su\S$, one has
\[
p_{a}\left(\C\right)=1+\frac{1}{2}\left(\C^{2}+\C\cdot\K_{\S}\right).
\]
The natural inclusion arising from the integral closure gives a monomorphism
$\oc\to f_{\ast}\mathcal{O}_{\tilde{\C}}$, taking the cokernel we get an exact
sequence
\[
0\to\oc\to f_{\ast}\mathcal{O}_{\tilde{\C}}\to\Delta\to0,
\]
where $\Delta$ is supported on singular points of $\C$, and hence
$H^{1}\left(\C,\Delta\right)=0$. Taking cohomology, we obtain the
exact sequence
\[
0\to H^{0}\left(\C,\Delta\right)\to H^{1}\left(\C,\oc\right)\to H^{1}\left(\C,f_{\ast}\mathcal{O}_{\tilde{C}}\right)\to0
\]
from which we deduce
\[
p_{a}\left(\C\right)=g\left(\C\right)+\sum_{x\in\C}\dim_{\cc}\Delta_{x}.
\]
Then, in general, $p_{a}\left(\C\right)\ge g\left(\C\right)$ with
equality if and only if $\C$ is smooth. In particular $p_{a}\left(\C\right)=0$
if and only if $\C$ is smooth with $\C\iso\Pu$.
\end{rem}

\begin{example}
\label{Ex:genusplanecurve}Let $\C\in\Pd$ be an irreducible curve
of degree $d$. Then, since $\K_{\Pd}\li-3L$ and in view of Corollary
\ref{cor:genusformula} and Example \ref{ex:inttheoryonplane}, we
have
\[
p_{a}\left(\C\right)=1+\frac{1}{2}\left(d^{2}-3d\right)=\frac{\left(d-1\right)\left(d-2\right)}{2}.
\]
Thus, in view of Remark \ref{rem:geometricarithmeticgenus}, we have
deduced the fact that a plane smooth curve of degree $d$ has geometric
genus given by
\[
g\left(\C\right)=\frac{1}{2}\left(d-1\right)\left(d-2\right).
\]
\end{example}

\subsection{Numerical properties of divisors on a surface}
\begin{defn}
Let $D$ be a divisor on a surface $\S$. We say that $D$ is $\emph{numerically trivial}$
if $D\cdot E=0$ for every $E\in\D\left(\S\right)$, if this is the
case we will write $D\sim0$. We will say that two divisors $D$ and
$D'$ are $\emph{numerically equivalent}$ and we will write $D\sim D'$
if $D-D'\sim0$. Since numerically trivial classes obviously form
a subgroup $\pic^{n}\left(\S\right)$ of $\pic\left(\S\right)$, it
is well defined the quotient group
\[
\num\left(\S\right):=\frac{\pic\left(\S\right)}{\pic^{n}\left(\S\right)}.
\]
\end{defn}

\begin{rem}
\label{Rem:usefulremhodgeindex}Let $X$ be an algebraic variety and
let $D$ and $D'$ be divisors on $X$. Let us suppose that there
exists $A\in|D|$. Then clearly the map
\begin{eqnarray*}
|D'| & \to & |D+D'|\\
E & \mapsto & E+A
\end{eqnarray*}
is injective, and thus $\dim|D'|\le\dim|D+D'|$.
\end{rem}

\begin{lem}
\label{lem:hodgeindexpossibilita}Let $D$ be a divisor on $\S$ with
$D^{2}>0$ and let $H$ be a very ample divisor. Then only one of
the followings holds true:

\begin{enumerate}
\item $D\cdot H>0$ and
\[
\lim_{n\to\infty}h^{0}\left(\S,\os\left(nD\right)\right)=\infty;
\]
\item $D\cdot H<0$ and
\[
\lim_{n\to-\infty}h^{0}\left(\S,\os\left(nD\right)\right)=\infty.
\]
\end{enumerate}
\end{lem}

\begin{proof}
In view of Riemann-Roch and Serre duality we have
\[
h^{0}\left(\S,\os\left(nD\right)\right)+h^{0}\left(\S,\os\left(\K_{\S}-nD\right)\right)\ge\frac{1}{2}n^{2}D^{2}-\frac{1}{2}nD\cdot\K_{\S}+\x\left(\os\right),
\]
this shows that
\[
\lim_{n\to\pm\infty}\left[h^{0}\left(\S,\os\left(nD\right)\right)+h^{0}\left(\S,\os\left(\K_{\S}-nD\right)\right)\right]=\infty.
\]
Let us see that $h^{0}\left(\S,\os\left(nD\right)\right)$ and $h^{0}\left(\S,\os\left(\K_{\S}-nD\right)\right)$
cannot both diverge for $n\to\pm\infty$: if, for $n\gg0$ (resp.
for $n\ll0$) there exists $E\in|nD|$, in view of Remark \ref{Rem:usefulremhodgeindex}
we must have $\dim|\K_{\S}-nD|\le\dim|\ks|$. Therefore it suffices
to show that $\dim|nD|$ and $\dim|\ks-nD|$ cannot diverge both for
$n\to\infty$ and $n\to-\infty$. If $|nD|\neq\emptyset$ for $n>0$,
then $D\cdot H>0$ and hence, if $m<0$, we must have $|mD|=\emptyset$
(otherwise $D\cdot H<0$). While, if $|\ks-nD|\neq\emptyset$ for
$n\gg0$, in view of Remark \ref{Rem:usefulremhodgeindex} we have
\[
\dim|2\ks|=\dim|\left(\ks+nD\right)+\left(\ks-nD\right)|\ge\dim|\ks+nD|.
\]
\end{proof}
\begin{thm}[Hodge Index Theorem]
\label{thm:Hodge}Let $H$ be an ample divisor on $\S$ and let $D$
be a divisor. If $D\cdot H=0$, then $D^{2}\le0$, with $D^{2}=0$
if and only if $D\sim0$.
\end{thm}

\begin{proof}
Let us suppose that $D^{2}>0$ and let us set $H'=D+nH$, which, in
view of Proposition \ref{prop:characterizationample} and Remark \ref{rem:sumveryample},
is ample for $n\gg0$, moreover $D\cdot H'=D^{2}>0$. In view of Lemma
\ref{lem:hodgeindexpossibilita}, for large $m$ the divisor $mD$
is linearly equivalent to an effective divisor, then $mD\cdot H>0$.
This is not possible since $D\cdot H=0$ .

If $D^{2}=0$ and $D\not\sim0$, then there exists $E\in\D\left(\S\right)$
with $D\cdot E\neq0$. Let $E'=H^{2}E-\left(H\cdot E\right)H$. Then
\[
\begin{array}{c}
D\cdot E'=H^{2}\left(D\cdot E\right)\neq0\\
H\cdot E'=0.
\end{array}
\]
Now let $n\in\Z$ and set $\tilde{D}=E'+nD$, then $\tilde{D}\cdot H=0$
and $\tilde{D}^{2}=E^{2}+2nD\cdot E$. Choosing $n\gg0$ if $D\cdot E>0$
or $n\ll0$ if $D\cdot E<0$, we can suppose $\tilde{D}^{2}>0$, but
then we get a contradiction with Lemma \ref{lem:hodgeindexpossibilita}
.
\end{proof}
\begin{cor}
\label{cor:fibersnumerically}Let $\pi:\S\to\C$ be a surjective morphism
from a surface to a curve. Then, for every $x,y\in\C$, we have $\pi^{\ast}\left[x\right]\sim\pi^{\ast}\left[y\right]$.
\end{cor}

\begin{proof}
Let $F_{1}=\pi^{\ast}\left[x\right]$, $F_{2}=\pi^{\ast}\left[y\right]$
and set $D=F_{1}-F_{2}$. In view of Example \ref{ex:fibresquaredzero}
we have $D^{2}=0$, moreover, let $H$ be an ample divisor. Taking
multiples, we can suppose $H$ to be very ample and hence, in view
of \cite[Theorem 6.6.1]{Ke} and Corollary \ref{cor:BertiniII}
we can take $H$ to be a smooth irreducible curve. The restriction
$\tilde{\pi}=\pi_{|H}:H\to\C$ is a morphism from a smooth irreducible
curve, and hence all the fibres have the same degree. This shows that
$H\cdot D=0$. In view of Theorem \ref{thm:Hodge} we conclude that
$F_{1}\sim F_{2}$.
\end{proof}
\begin{defn}
Let $D$ be a divisor on a surface $\S$. We say that $D$ is $\emph{nef}$
if $D\cdot\C\ge0$ for every curve $\C\su\S$, while we say that $D$
is $\emph{pef}$ if $D\cdot H\ge0$ for every ample divisor $H$ on
$\S$.
\end{defn}

\begin{rem}
It is clear that every ample divisor is nef and that every nef divisor
is pef.
\end{rem}

We see a useful way to check nefness
\begin{lem}
\label{lem:tricktochecknefeness}Let $\S$ be a surface and let $D$
be a divisor such that $|D|$ has only a finite numer of fixed points.
Then $D$ is nef.
\end{lem}

\begin{proof}
If there exists a curve $\C$ such that $D\cdot\C<0$, then $\C$
should be contained in the fixed part of $|D|$. This is not possible
since, under our assumptions, we have $\dim\bs|D|=0$ .
\end{proof}
\begin{prop}
\label{prop:nefpositivesquare}Let $D$ be a nef divisor on a surface
$\S$. Then $D^{2}\ge0$.
\end{prop}

\begin{proof}
Let us suppose $D^{2}<0$ and let $H$ be an ample divisor. Then elementary
algebra shows that there exists $x\in\Q$ such that $x>0$ and
\[
\left(D+xH\right)^{2}>0\,\,\,\,\mbox{and}\,\,\,\,\left(D+\frac{x}{2}H\right)^{2}<0.
\]
Since $\left(D+xH\right)\cdot H=D\cdot H+xH^{2}>0$, in view of Lemma
\ref{lem:hodgeindexpossibilita}, for $m\gg0$ we have that $h^{0}\left(\S,\os\left(m\left(D+xH\right)\right)\right)>0$.
We compute
\[
0>\left(D+\frac{x}{2}H\right)^{2}=D^{2}+xD\cdot H+\frac{x^{2}}{4}H^{2}=\frac{1}{m}D\cdot\left[m\left(D+xH\right)\right]+\frac{x^{2}}{4}H^{2},
\]
which is a contradiction .
\end{proof}
\begin{defn}
Let $D$ be a divisor on a surface $\S$. We say that $D$ is $\emph{big}$
if there exists a constant $C>0$ such that, for $m\gg0$ we have
\[
h^{0}\left(\S,\os\left(mD\right)\right)\iso Cm^{2}.
\]
\end{defn}

\begin{rem}
We see that every ample divisor is big: if $D$ is an ample divisor
on a surface $\S$, then, in view of Theorem \ref{prop:amplemultiple},
there exists $n_{0}$ such that, for every $n\ge n_{0}$, we have
\[
h^{i}\left(\S,\os\left(\left(n+1\right)D\right)\right)=0
\]
for every $i>0$. It follows from Riemann-Roch that, for every $n\ge n_{0}$,
\[
h^{0}\left(\S,\os\left(\left(n+1\right)D\right)\right)=\x\left(\os\right)+\frac{1}{2}\left(\left(n+1\right)^{2}D^{2}-\left(n+1\right)D\cdot\ks\right).
\]
Therefore, for $m\gg0$, we have $h^{0}\left(\S,\os\left(mD\right)\right)\iso\frac{1}{2}m^{2}D^{2}$.
\end{rem}

We state now a theorem which will be useful several times. A full
proof can be found in \cite{La}.
\begin{thm}[Ramanujam-Kawamata-Viehweg]
\label{thm:RKW}Let $H$ be a nef and big divisor on a surface $\S$.
Then
\[
h^{1}\left(\S,\os\left(\ks+H\right)\right)=h^{2}\left(\S,\os\left(\ks+H\right)\right)=0.
\]
\end{thm}

\begin{lem}
\label{lem:pernakaiuno}Let $\S$ be a surface and $D\in\D\left(\S\right)$
be an effective divisor. Let $D_{1}$ be a prime divisor such that
$\supp\left(D_{1}\right)\su\supp\left(D\right)$ and set $D_{2}=D-D_{1}$.
Then there is an exact sequence of $\os$-modules
\[
0\to\mathcal{O}_{D_{1}}\left(-D_{2}\right)\to\mathcal{O}_{D}\to\mathcal{O}_{D_{2}}\to0.
\]
\end{lem}

\begin{proof}
Let $x\in\S$ be a point and, for $j=1,2$, let $I_{j}$ be the local
ideal of $D_{j}$ in $x$. Then the local ideal of $D$ in $x$ is
given by $I=I_{1}\cdot I_{2}$, thus we have an exact sequence of
$\mathcal{O}_{\S,x}$-modules
\[
0\to\frac{I_{2}}{I}\to\frac{\mathcal{O}_{\S,x}}{I}\to\frac{\mathcal{O}_{\S,x}}{I_{1}}\to0.
\]
We conclude since
\[
\frac{I_{2}}{I}\iso\frac{\mathcal{O}_{\S,x}}{I_{1}}\ot_{\mathcal{O}_{\S,x}}I_{2}.
\]
\end{proof}
\begin{prop}
\label{prop:pernakaidue}Let $X$ be projective variety, $\mathcal{O}_{X}\left(1\right)$
a very ample sheaf and let $\F$ be a coherent $\mathcal{O}_{X}$-module. Then
there exists a numerical polynomial $P\left(x\right)\in\Q_{n}\left[x\right]$
such that $\x\left(\F\left(n\right)\right)=P\left(n\right)$ for every
$n\in\Z$ with
\[
\deg P\le\dim\left[\supp\left(F\right)\right].
\]
\end{prop}

\begin{proof}
Let $j:X\hookrightarrow\mathbb{P}^{N}$ be the embedding given by $\mathcal{O}_{X}\left(1\right)$.
In view of Corollary \cite[Corollary 8.2.3]{Ke}, the Proposition
is true for $\F$ if and only if it is for $j_{\ast}F$, thus we can
suppose $X=\mathbb{P}^{N}$. We proceed by induction on $\dim\left[\supp\left(\F\right)\right]$:
if $\F$ is the zero sheaf, then clearly $P=0$ works. Let $\left(f=0\right)$
be an hyperplane section not containing any component of $\supp\left(\F\right)$.
Then the morphism $\F\left(-1\right)\overset{\cdot f}{\longrightarrow}\F$
fits into an exact sequence
\[
0\to\K\to\F\left(-1\right)\overset{\cdot f}{\longrightarrow}\F\to\C\to0.
\]
In view of our choice of $f$, we see that
\[
\dim\left[\supp\left(\K\right)\right],\dim\left[\supp\left(\C\right)\right]<\dim\left[\supp\left(\F\right)\right].
\]
Twisting with $\mathcal{O}_{X}\left(n\right)$ and using induction hypothesis
we see that there exists two numerical polynomials $Q_{1}\left(x\right),Q_{2}\left(X\right)\in\Q_{n}\left[x\right]$
with
\[
\deg\left(Q_{1}\right),\deg\left(Q_{2}\right)<\dim\left[\supp\left(\F\right)\right]
\]
such that
\[
\x\left(\F\left(n\right)\right)-\x\left(\F\left(n-1\right)\right)=Q_{1}\left(n\right)-Q_{2}\left(n\right).
\]
In view of basic properties of numerical polynomials, if
\[
Q_{1}\left(n\right)-Q_{2}\left(n\right)=\sum_{i=0}^{n-1}q_{i}\binom{n}{i},
\]
then
\[
\x\left(\F\left(n\right)\right)=\sum_{i=0}^{n-1}q_{i}\binom{n+1}{i+1}.
\]
\end{proof}
\begin{lem}
\label{lem:sufficienttobebig}Let $\S$ be a surface and let $D$
be a divisor with $D^{2}>0$ and $D\cdot H>0$ for some ample divisor
$H$. Then $D$ is big.
\end{lem}

\begin{proof}
In view of Lemma \ref{lem:hodgeindexpossibilita}, for $m\gg0$ the
linear system $|mD|$ defines a rational map
\[
\varphi_{|mD|}:\S\ra\S'\su\Pn
\]
and hence we have
\[
h^{0}\left(\S,\os\left(mD\right)\right)=h^{0}\left(\S',\mathcal{O}_{\S'}\left(mH'\right)\right)
\]
for an hyperplane section $H'$. In view of Theorem \ref{prop:amplemultiple},
we have $h^{0}\left(\S',\mathcal{O}_{\S'}\left(mH'\right)\right)=\x\left(\mathcal{O}_{\S'}\left(mH'\right)\right)$
for $m\gg0$. Let us consider the exact sequence
\[
0\to\mathcal{O}_{\S'}\left(mH'\right)\to\mathcal{O}_{\S'}\left(\left(m+1\right)H'\right)\to\mathcal{O}_{H'}\left(\left(m+1\right)H'\right)\to0
\]
in which, in view of \cite[Theorem 6.6.1]{Ke} and Corollary \ref{cor:BertiniII}
we can take $H'$ to be a smooth irreducible curve. Then
\[
\x\left(\mathcal{O}_{\S'}\left(\left(m+1\right)H'\right)\right)-\x\left(\mathcal{O}_{\S'}\left(mH'\right)\right)=\x\left(\mathcal{O}_{H'}\left(\left(m+1\right)H'\right)\right).
\]
In view of Proposition \ref{prop:pernakaidue} and by elementary algebra
we conclude that, for $m\gg0$, there exists $C>0$ such that $h^{0}\left(\S,\os\left(mD\right)\right)\le C_{2}m^{2}$.
On the other hand, for $m\gg0$ we have $\left(\ks-mD\right)\cdot H<0$,
by Serre duality we see that
\[
h^{2}\left(\S,\mathcal{O}_{\S}\left(mD\right)\right)=h^{0}\left(\S,\os\left(\ks-mD\right)\right)=0.
\]
By Riemann-Roch we conclude that
\[
h^{0}\left(\S,\os\left(mD\right)\right)\ge\x\left(\os\left(mD\right)\right)=\frac{1}{2}m^{2}D^{2}+\mbox{lower degree terms}.
\]
\end{proof}
\begin{thm}[Nakai-Moishezon Criterion for Ampleness]
\label{thm:NakaiMoishezon}Let $\S$ be a surface and $D\in\D\left(\S\right)$.
Then $D$ is ample if and only if $D^{2}>0$ and $D\cdot\C>0$ for
every curve $\C$.
\end{thm}

\begin{proof}
If $D$ is an ample divisor, then clearly $D^{2}>0$ and $D\cdot\C>0$
for every curve $\C$. To see the converse, let us consider
\begin{equation}
0\to\mathcal{O}_{\S}\left(mD\right)\to\mathcal{O}_{\S}\left(\left(m+1\right)D\right)\to\mathcal{O}_{T}\left(\left(m+1\right)D\right)\to0\label{eq:cohomologynakai}
\end{equation}
where $T=\sum t_{i}T_{i}$ is an effective divisor. We see by induction
on $n=\sum t_{i}$ that, for $m\gg0$, the linear system $|\left(m+1\right)D|$
is base-point-free. If $n=1$, then $T$ is an irreducible curve.
Since $D\cdot T>0$, it follows from Riemann-Roch for curves that
$H^{1}\left(T,\mathcal{O}_{T}\left(\left(m+1\right)D\right)\right)=0$ for
$m\gg0$. Thus, taking the exact cohomology sequence associated to
(\ref{eq:cohomologynakai}) we see that the map
\[
H^{1}\left(\S,\os\left(mD\right)\right)\to H^{1}\left(\S,\os\left(\left(m+1\right)D\right)\right)
\]
is definitevely an isomorphism. In view of Lemma \ref{lem:trucchettocoomologia}
we conclude that, in the case $n=1$, we have $\bs|mD|=\emptyset$
for $m\gg0$. If $n=\sum t_{i}\ge2$, using notations as in Lemma
\ref{lem:trucchettocoomologia}, we have an exact sequence of $\os$-modules
\[
0\to\mathcal{O}_{T_{1}}\left(mD-T_{2}\right)\to\mathcal{O}_{T}\left(mD\right)\to\mathcal{O}_{T_{2}}\left(mD\right)\to0.
\]
By induction hypothesis and in view of our assumptions on $D$, we
conclude that, for $m\gg0$,
\[
h^{1}\left(T_{1},\mathcal{O}_{T_{1}}\left(mD-T_{2}\right)\right)=0
\]
and hence $h^{1}\left(T,\mathcal{O}_{T}\left(mD\right)\right)=0$. As before,
we conclude that $\bs|mD|=\emptyset$ for $m\gg0$. Then, since $D^{2}>0$,
we have a well-defined morphism
\[
\varphi_{|mD|}:\S\to\S'\su\Pn
\]
where $\S'=\im\left(\varphi_{|mD|}\right)$ is a surface. Since $D\cdot\C>0$
for every curve $\C\su\S$, the morphis $\varphi_{|mD|}$ has finite
fibres. In view of \cite[Lemma 10.1.1]{Ke} it follows that $\varphi_{|mD|}$
is finite, thus $mD=\varphi_{|mD|}^{\ast}\mathcal{O}_{\S'}\left(1\right)$,
and hence $D$, is ample in view of Corollary \ref{lem:finiteample}.
\end{proof}
\begin{rem}
\label{rem:sommanefampio}If $D$ is nef and $H$ is ample, then,
in view of Proposition \ref{prop:nefpositivesquare},
\[
\left(D+H\right)^{2}=D^{2}+H^{2}+2D\cdot H>0,
\]
moreover we clearly have $\left(D+H\right)\cdot\C>0$ for every curve
$\C$. We conclude from Theorem \ref{thm:NakaiMoishezon} that $D+H$
is ample.
\end{rem}

\begin{cor}
\label{cor:amplenessnumericalproperty}Let $D\in\D\left(\S\right)$
be an ample divisor and let $D'$ be a divisor with $D\sim D'$. Then
$D'$ is ample.
\end{cor}

\subsection{Rank 2 locally free sheaves on curves}
\begin{prop}
\label{prop:splittingcriterio}Let $X$ be an irreducible and projective
algebraic variety and let
\begin{equation}
0\to\L\to\mathcal{E}\to\M\to0\label{eq:sequenzachespezza}
\end{equation}
be an exact sequence of sheaves on $X$ with $\L,\M\in\pic\left(\C X\right)$
and $\mathcal{E}$ locally free of rank 2. Let
\[
\delta:H^{0}\left(X,\mathcal{O}_{X}\right)=\cc\to H^{1}\left(X,\L\ot\M^{-1}\right)
\]
be the natural connecting homomorphism. Then the sequence (\ref{eq:sequenzachespezza})
splits if and only if $\delta\left(1\right)=0$.
\end{prop}

\begin{proof}
If the sequence (\ref{eq:sequenzachespezza}) splits, then $\mathcal{E}\iso\L\op\M$
and we can write
\[
0\to\L\ot\M^{-1}\to\left(\L\ot\M^{-1}\right)\op\mathcal{O}_{X}\to\mathcal{O}_{X}\to0.
\]
Therefore the map
\[
H^{0}\left(X,\left(\L\ot\M^{-1}\right)\op\mathcal{O}_{X}\right)=H^{0}\left(X,\L\ot\M^{-1}\right)\op\cc\to H^{0}\left(X,\mathcal{O}_{X}\right)=\cc
\]
is surjective and hence $\delta\left(1\right)=0$. To prove the converse,
let us consider the sequence
\[
0\to H^{0}\left(X,\L\ot\M^{-1}\right)\to H^{0}\left(X,\mathcal{E}\ot\M^{-1}\right)\overset{\varphi}{\longrightarrow}\cc\to0
\]
which is exact by our hypothesis. Let $f\in H^{0}\left(X,\mathcal{E}\ot\M^{-1}\right)$
be a section such that $\varphi\left(f\right)=1$, we can define a
$\mathcal{O}_{X}$-module morphism $\psi:\mathcal{O}_{X}\to\mathcal{E}\ot\M^{-1}$ given on stalks
by $\psi_{x}\left(1_{x}\right)=f_{x}$ and extended by linearity.
Clearly $\psi$ gives a section and thus the sequence (\ref{eq:sequenzachespezza})
splits.
\end{proof}
\begin{prop}
\label{prop:Classificationcech}The set of isomorphism classes of
rank $r$ vector bundles on an algebraic variety $X$ is in natural
bijection with $\check{H}^{1}\left(X,\gl_{r}\left(\mathcal{O}_{X}\right)\right)$.
\end{prop}

\begin{proof}
If $\pi:\mathcal{E}\to X$ is a rank $r$ vector bundle on $X$, let $\left(U_{i}\right)$
be an affine cover of $X$ such that there exist trivializations $\varphi_{i}:\pi^{-1}\left(U_{i}\right)\to U_{i}\times\Bbbk^{r}$
for every $i$. Then clearly the transition functions $\left(t_{ij}=\varphi_{i}\circ\varphi_{j|U_{i}\cap U_{j}}^{-1}\right)$
define an element of $\check{H}^{1}\left(X,\gl_{r}\left(\mathcal{O}_{X}\right)\right)$.
Viceversa, given a \v{C}ech 1-cocycle $\left(g_{ij}:U_{i}\cap U_{j}\to\gl_{r}\left(\cc\right)\right)$,
we can consider
\[
\bigcup_{i}U_{i}\times\cc^{r}
\]
and we can identify each fibre $\left\{ z\right\} \times\cc^{r}$
in $U_{i}\times\cc^{r}$ and $U_{j}\times\cc^{r}$ using $g_{ij}\left(z\right)$.
\end{proof}
\begin{rem}
\label{rem:cohomologypic}Let us note that Proposition \ref{prop:Classificationcech},
in view of \cite[Proposition 9.2.2]{Ke}, gives a natural identification
between $\pic\left(X\right)$ and $H^{1}\left(X,\mathcal{O}_{X}^{\ast}\right)$.

From now on we will identify any vector bundle with its sheaf of global
sections.
\end{rem}

\begin{lem}
\label{lem:hasezioneglobale}Let $\mathcal{E}$ be a rank 2 locally free sheaf
over a smooth curve $\C$. Then there exists $\L\in\pic\left(\C\right)$
such that
\[
H^{0}\left(\C,\mathcal{E}\ot\L\right)\neq0.
\]
\end{lem}

\begin{proof}
Let $\oc\left(1\right)\in\pic\left(\C\right)$ be very ample. It follows
from Theorem \ref{thm:veryampleimplicaample} that, for $n\gg0$,
the sheaf $\mathcal{E}\left(n\right)$ is generated by global sections.
\end{proof}
\begin{defn}
Let $\mathcal{E}$ be a rank 2 locally free sheaf over a smooth curve $\C$,
we define $\deg\left(\mathcal{E}\right)$ as $\deg\left(\det\mathcal{E}\right)$. Then,
for every $\L\in\pic\left(C\right)$, we have
\[
\deg\left(\mathcal{E}\ot\L\right)=\deg\left(\mathcal{E}\right)+2\deg\left(\L\right).
\]
\end{defn}

\begin{lem}
\label{lem:rango2sequenzaesatta}Let $\mathcal{E}$ be a rank 2 locally free
sheaf over a smooth curve $\C$. Then there exist $\L,\M\in\pic\left(\C\right)$
and an exact sequence
\[
0\to\L\to\mathcal{E}\to\M\to0.
\]
Moreover, if $h^{0}\left(\C,\mathcal{E}\right)>0$, we can take $\L=\oc\left(D\right)$
with $D\ge0$ and, if $h^{0}\left(\C,\mathcal{E}\right)\ge2$ and $\deg\left(\mathcal{E}\right)>0$,
we can take $D>0$.
\end{lem}

\begin{proof}
In view of Lemma \ref{lem:hasezioneglobale}, replacing $\mathcal{E}$ with
$\mathcal{E}\ot\mathcal{N}$ for some $\mathcal{N}\in\pic\left(C\right)$,
we can assume that $h^{0}\left(\C,\mathcal{E}\right)\ge1.$ Let $s\in H^{0}\left(\C,\mathcal{E}\right)$
be a non-zero section, thus $s$ defines a morphism $\overline{s}:\mathcal{E}^{\ast}\to\oc$
whose image is a subsheaf of ideals of $\oc$, i.e. there exists an
effective divisor $D\in\D\left(C\right)$ such that $\im\left(\overline{s}\right)\iso\oc\left(-D\right).$
Thus, taking the kernel and dualizing we have an exact sequence
\[
0\to\oc\left(D\right)\to\mathcal{E}\to\M\to0.
\]
If $h^{0}\left(\C,\mathcal{E}\right)\ge2$, to prove the last statement it
is enough to show that there is a non-zero section of $\mathcal{E}$ which
vanishes at some point $p\in\C$. Let $s,t\in H^{0}\left(\C,\mathcal{E}\right)$
be two linearly independent global sections. Since $\deg\left(\mathcal{E}\right)>0$,
there exists $p\in\C$ such that $\left(s\wedge t\right)\left(p\right)=0$.
Therefore the exist $\lambda,\mu\in\cc$ not both zero such that $\lambda s\left(p\right)+\mu t\left(p\right)=0$.
\end{proof}
\begin{rem}
\label{rem:secondclass}Let $\C$ be a smooth curve and let us consider
an exact sequence
\[
0\to\L\to\mathcal{E}\to\M\to0
\]
where $\L,\M\in\pic\left(\C\right)$ and $\mathcal{E}$ is a locally free sheaf
of rank 2. Let us compute
\begin{eqnarray*}
\L\cdot\M & = & \L^{-1}\cdot\M^{-1}\\
 & = & \x\left(\mathcal{O}_{\C}\right)-\x\left(\L\right)-\x\left(M\right)+\x\left(\L\ot\M\right)\\
 & = & \x\left(\mathcal{O}_{\C}\right)-\x\left(\mathcal{E}\right)+\x\left(\det\mathcal{E}\right),
\end{eqnarray*}
hence the integer $\L\cdot\M$ depends only on $\mathcal{E}$. We set $\L\cdot M=c_{2}\left(\mathcal{E}\right)$.
\end{rem}

\begin{thm}
\label{thm:fibratiPuno}Let $\mathcal{E}$ be a rank 2 locally free sheaf over
$\C=\Pu$. Then there exist $a,b\in\Z$ such that
\[
\mathcal{E}\iso\oc\left(a\right)\op\oc\left(b\right).
\]
\end{thm}

\begin{proof}
Replacing $\mathcal{E}$ with $\mathcal{E}\ot\L$, we can assume that $d=\deg\left(\mathcal{E}\right)=0,-1.$
In view of \cite[Corollary 8.4.4]{Ke} we have $h^{0}\left(\C,\mathcal{E}\right)\ge d+2\ge1,$
thus, in view of Lemma \ref{lem:rango2sequenzaesatta}, there exist
$a\in\mathbb{N}$ and an exact sequence
\begin{equation}
0\to\oc\left(a\right)\to\mathcal{E}\to\oc\left(d-a\right)\to0.\label{eq:spezzaP}
\end{equation}
The class of this extension lies in
\[
H^{1}\left(\C,\oc\left(2a-d\right)\right)=0.
\]
In view of Proposition \ref{prop:splittingcriterio} the sequence
(\ref{eq:spezzaP}) splits and we conclude.
\end{proof}

\subsection{$\protect\Pu$-bundles over curves}
\begin{lem}
\label{lem:genuszero}Let $\pi:\S\to\C$ be a surjetive morphism from
a surface to a smooth curve. If there exists $x\in\C$ such that $\pi^{-1}\left(x\right)\iso\Pu$,
then $H^{2}\left(\S,\os\right)=0$.
\end{lem}

\begin{proof}
In view of Serre duality, we have $h^{2}\left(\S,\os\right)=h^{0}\left(\S,\omega_{\S}\right)$.
Let $F$ be a fibre of $\pi$, then, in view of Corollaries \ref{cor:fibersnumerically},
\ref{cor:genusformula} and Example \ref{ex:fibresquaredzero} we
get $F\cdot\ks=-2$ and $F^{2}=0$. If there exists $D\in|\ks|$,
in view of Lemma \ref{lem:beauvilleusefulremark} we have
\[
\ks\cdot F=D\cdot F\ge0,
\]
which is a contradiction .
\end{proof}
\begin{prop}
\label{prop:existssection}Let $\pi:\S\to\C$ be a surjetive morphism
from a surface to a smooth curve. If $\pi^{-1}\left(x\right)\iso\Pu$
for every $x\in\C$, then $\pi$ has a section.
\end{prop}

\begin{proof}
Let $F$ be a fibre of $\pi$, we show that there exists a divisor
$H$ on $\S$ such that $H\cdot F=1$. In view of Lemma \ref{lem:genuszero},
the exponential sequence on $\S$ and Remark \ref{rem:cohomologypic}
we deduce that the natural map
\begin{equation}
H^{1}\left(\S,\os^{\ast}\right)=\pic\left(\S\right)\to H^{2}\left(\S,\Z\right)\label{eq:mapsurjective}
\end{equation}
is surjective. Let $f,k$ be the classes of $F$ and $\ks$ in $H^{2}\left(\S,\Z\right)$.
The set $\left\{ a\cdot f\,\,|\,\,a\in H^{2}\left(\S,\Z\right)\right\} \su\Z$
is a subgroup, and hence it is of the form $d\Z$. Since $f\cdot k=-2$,
we have $d\neq0$ and it is well-defined a $\Z$-linear form
\begin{eqnarray*}
H^{2}\left(\S,\Z\right) & \to & \Z\\
a & \mapsto & \frac{1}{d}\left(a\cdot f\right).
\end{eqnarray*}
From Poincaré's duality we have an isomorphism
\[
\frac{H^{2}\left(\S,\Z\right)}{\mathrm{Tors\left(H^{2}\left(\S,\Z\right)\right)}}\iso\mathrm{Hom_{\Z}\left(H^{2}\left(\S,\Z\right),\Z\right)}
\]
given by the cup-product, i.e. there exists $h\in H^{2}\left(\S,\Z\right)$,
defined up to a torsion element, such that $f=dh$. In view of Corollary
\ref{cor:genusformula} and the surjectivity of (\ref{eq:mapsurjective})
we deduce that $a^{2}+a\cdot k$ is even for every $a\in H^{2}\left(\S,\Z\right)$.
Since $f^{2}=0$, we get that $h^{2}=0$ and finally $h\cdot k=-\frac{2}{d}$
must be even, this concludes the proof.
\end{proof}
\begin{defn}
Let $\mathcal{E}$ be a rank 2 vector bundle on a curve $\C$. We define $\mathbb{P}_{\C}\left(\mathcal{E}\right)$
as the $\Pu$-fibration over $\C$ with fibres defined as $\mathbb{P}_{\C}\left(\mathcal{E}\right)_{x}=\mathbb{P}\left(\mathcal{E}_{x}\right)$
for every $x\in\C$. We call $\mathbb{P}_{\C}\left(\mathcal{E}\right)$ the $\emph{projective bundle}$
associated to $\mathcal{E}$.
\end{defn}

\begin{rem}
\label{rem:projectivebundle}It is clear that, when $\pi:\mathcal{E}\to\C$
be a rank 2 vector bundle, $\pc\left(\mathcal{E}\right)$ is a surface locally
isomorphic (over $\C$) to $\C\times\Pu$ and there is a natural surjective
morphism $\tilde{\pi}:\pc\mbox{\ensuremath{\left(\mathcal{E}\right)\to\C}}$
associated to $\pi$. Moreover, transition functions for the fibration
$\pc\mbox{\ensuremath{\left(\mathcal{E}\right)}}$ are simply transition functions
for $\mathcal{E}$ up to a scalar multiple, hence, in view of the proof of
Proposition \ref{prop:Classificationcech}, we deduce that $\C$-isomorphism
classes of $\Pu$-fibrations are in natural one-to-one correspondence
with $\check{H}^{1}\left(\C,\pgl_{1}\left(\mathcal{O}_{\C}\right)\right)$.
\end{rem}

\begin{thm}
\label{thm:isomorphicPunobundles}Let $\mathcal{E},\mathcal{E}'$ be rank 2 vector bundles
over a curve $\C$. Then $\pc\left(\mathcal{E}\right)$ and $\pc\left(\mathcal{E}'\right)$
are $\C$-isomorphic if and only if there exists $\L\in\pic\left(\C\right)$
such that $\mathcal{E}\iso\mathcal{E}'\ot\L$.
\end{thm}

\begin{proof}
Let us consider the exact sequence
\[
1\to\mathcal{O}_{\C}^{\ast}\to\gl_{2}\left(\mathcal{O}_{\C}\right)\to\pgl_{1}\left(\mathcal{O}_{\C}\right)\to1,
\]
then, in view of \cite[Proposition 9.2.2]{Ke} and Remark \ref{rem:cohomologypic},
we have an exact sequence of cohomology sets
\[
1\to\pic\left(\C\right)\to\check{H}^{1}\left(\C,\gl_{2}\left(\mathcal{O}_{C}\right)\right)\to\check{H}^{1}\left(\C,\pgl_{2}\left(\oc\right)\right)\to1,
\]
which shows that
\[
\check{H}^{1}\left(\C,\pgl_{1}\left(\oc\right)\right)\iso\frac{\check{H}^{1}\left(\C,\gl_{2}\left(\oc\right)\right)}{\pic\left(\C\right)}
\]
and we conclude in view of Proposition \ref{prop:Classificationcech}.
\end{proof}
\begin{example}
\label{example: Hirzebruchsurfaces}Let $\mathcal{E}$ be a rank 2 vector bundle
over $\Pu$, then, in view of Theorem \ref{thm:fibratiPuno}, there
exist $a,b\in\Z$ such that $\mathcal{E}\iso\mathcal{O}_{\Pu}\left(a\right)\op\mathcal{O}_{\Pu}\left(b\right)$.
Let us suppose that $a\ge b$. In view of Theorem \ref{thm:isomorphicPunobundles}
we conclude that
\[
\mathbb{P}_{\Pu}\left(\mathcal{E}\right)\iso\mathbb{P}_{\Pu}\left(\mathcal{O}_{\Pu}\op\mathcal{O}_{\Pu}\left(b-a\right)\right).
\]
This shows that every $\Pu$-bundle over $\Pu$ is $\Pu$-isomorphic
to
\[
\mathbb{F}_{n}=\mathbb{P}_{\Pu}\left(\mathcal{O}_{\Pu}\op\mathcal{O}_{\Pu}\left(-n\right)\right)
\]
for some $n\ge0$.
\end{example}

\begin{construction}
Let $\pi:\pc \left( \mathcal{E} \right)=\S \to\C$ be a $\Pu$-bundle over a smooth curve $\C$. For every $s\in\S$, we can consider the corresponding line $\ell_{s}$ inside $\mathcal{E} _{\pi \left( s \right)}$. Let $\M$ be the sheaf on $\S$ defined by $\M_{s} =\ell _s$, then we have a monomorphism
\begin{equation}
0\to\M\to\pi^\ast \mathcal{E}.\label{eq:tautologicalS}
\end{equation}
\end{construction}
\begin{defn}
Let $\pi:\pc\left(\mathcal{E}\right)=\S\to\C$ be a $\Pu$-bundle. We define
$\os\left(1\right)$ by the exact sequence
\begin{equation}
0\to\M\to\pi^{\ast}\mathcal{E}\overset{u}{\longrightarrow}\os\left(1\right)\to0\label{eq:deftautological}
\end{equation}
arising from the sequence (\ref{eq:tautologicalS}). We call $\os\left(1\right)$
the $\emph{tautological sheaf of \ensuremath{\S}}$.
\end{defn}

\begin{thm}[Universal property of $\pc\left(\mathcal{E}\right)$]
\label{thm:universalpropertyproj}Let $X$ be a surface, $f:X\to\C$
a surjective morphism to a smooth curve and $\pi:\pc\left(\mathcal{E}\right)\to\C$
a $\Pu$-bundle over $\C$ with tautological sheaf $\mathcal{O}_{\pc\left(\mathcal{E}\right)}\left(1\right)$.
Then, to give a $\C$-morphism $g:X\to\pc\left(\mathcal{E}\right)$ is equivalent
to give an invertible sheaf $\L\in\pic\left(X\right)$ such that $\L\iso g^{\ast}\mathcal{O}_{\pc\left(\mathcal{E}\right)}\left(1\right)$
with an epimorphism $f^{\ast}\mathcal{E}\to\L$.
\end{thm}

\begin{proof}
One direction is clear: in view of the definition of $\mathcal{O}_{\pc\left(\mathcal{E}\right)}\left(1\right)$
we have an epimorphism
\[
\pi^{\ast}\mathcal{E}\to\mathcal{O}_{\pc\left(\mathcal{E}\right)}\left(1\right),
\]
moreover the commutativity of
\[
\xymatrix{X\ar[rr]^{g}\ar[dr]_{f} &  & \pc\left(\mathcal{E}\right)\ar[dl]^{\pi}\\
 & \C
}
\]
and the right exactness of the pullback give an epimorphism
\[
f^{\ast}\mathcal{E}=g^{\ast}\pi^{\ast}\mathcal{E}\to g^{\ast}\mathcal{O}_{\pc\left(\mathcal{E}\right)}\left(1\right),
\]
and hence we can take $\L=g^{\ast}\mathcal{O}_{\pc\left(\mathcal{E}\right)}\left(1\right)$.

Conversely, let $\L\in\pic\left(X\right)$ and let $\varphi:f^{\ast}\mathcal{E}\to\L$
be an epimorphism. For every $x\in X$ we set $g\left(x\right)$ as
the projectivization of the line $\mathrm{ker}\left(\varphi_{x}\right)\su\mathcal{E}_{f\left(x\right)}$.
It is clear that $g:X\to\pc\left(\mathcal{E}\right)$ is a $\C$-morphism.
Taking stalks over $g\left(x\right)$ in (\ref{eq:deftautological}),
we have a diagram
\[
\xymatrix{0\ar[r] & \mathrm{ker}\left(\varphi_{x}\right)\ar[r]^{j_{x}} & \mathcal{E}_{f\left(x\right)}\ar[r]^{\varphi_{x}}\ar@{=}[d] & \L_{x}\ar[r] & 0\\
0\ar[r] & \M_{g\left(x\right)}\ar[r]_{i_{x}} & \mathcal{E}_{f\left(x\right)}\ar[r]_{u_{g\left(x\right)}} & \mathcal{O}_{\pc\left(\mathcal{E}\right)}\left(1\right)_{g\left(x\right)}\ar[r] & 0
}
\]
with exact rows. It follows from the definition of $g$ that $\im\left(j_{x}\right)\su\im\left(i_{x}\right)$,
therefore, since they are 1-dimensional vector spaces, we have $\im\left(j_{x}\right)=\im\left(i_{x}\right)$
and there exists $\lambda\in\cc^{\ast}$ such that the map $\psi_{x}:\mathrm{ker}\left(\varphi_{x}\right)\to\M_{g\left(x\right)}$
given by $\psi_{x}\left(v\right)=\lambda v$ is an isomorphism with
$j_{x}=i_{x}\circ\psi_{x}$. In view of the 5-lemma, it follows that
the natural map $\xi_{x}:\L_{x}\to\mathcal{O}_{\pc\left(\mathcal{E}\right)}\left(1\right)_{g\left(x\right)}$
is an isomorphism such that $u_{g\left(x\right)}=\xi_{x}\circ\varphi_{x}$.
Thus we have an isomorphism $\xi:\L\to g^{\ast}\mathcal{O}_{\pc\left(\mathcal{E}\right)}\left(1\right)$
and the theorem is proved.
\end{proof}
\begin{lem}
\label{lem:flatness}Let $\S$ be a surface and $\pi:S\to\C$ a surjective
morphism onto a smooth curve, then $\pi$ is a flat morphism.
\end{lem}

\begin{proof}
Let $x\in\S$ be a point and let us consider
\[
\pi^{\ast}:\mathcal{O}_{\C,\pi\left(x\right)}\to\mathcal{O}_{\S,x}.
\]
Since $\S$ is irreducible, the local ring $\mathcal{O}_{\S,x}$ is an integral
domain an thus it is a torsion-free $\mathcal{O}_{\C,\pi\left(x\right)}$-module.
We conclude since $\pi\left(x\right)\in\C$ is a smooth point and
hence $\mathcal{O}_{\C,\pi\left(x\right)}$ is a Dedekind domain.
\end{proof}
\begin{lem}
\label{lem:2.1Hartshorne}Let $\pi:\pc\left(\mathcal{E}\right)=\S\to\C$ be
a $\Pu$-bundle over a smooth curve, let $D\in\D\left(\S\right)$
such that $D\cdot F=n\ge0$ for every fibre $F$ of $\pi$. Then $\pi_{\ast}\os\left(D\right)$
is a locally free sheaf of rank $n+1$ on $\C$. In particular $\pi_{\ast}\os=\oc$.
\end{lem}

\begin{proof}
In view of Corollary \ref{cor:fibersnumerically}, we see that $n$
does not depend on the choice of the fibre $F$. In view of Lemma
\ref{lem:flatness} the morphism $\pi$ is flat. Under our assumptions
we have
\[
h^{0}\left(F,\os\left(D\right)_{|F}\right)=h^{0}\left(\Pu,\mathcal{O}_{\Pu}\left(n\right)\right)=n+1.
\]
In view of \cite[Proposition 9.5.2]{Ke} we conclude that $\pi_{\ast}\os\left(D\right)$
is a rank $n+1$ locally free sheaf on $\C$. To prove the last statement,
let us suppose $\os\left(D\right)\iso\os$, then $\pi_{\ast}\os$
is a locally free sheaf of rank 1. It follows from \cite[Proposition 9.5.2]{Ke}
that for every fibre $F$ we have an isomorphism
\[
R^{0}\pi_{\ast}\os\ot\cc\iso H^{0}\left(F,\mathcal{O}_{F}\right)\iso\cc.
\]
Therefore the natural morphism $\oc\to\pi_{\ast}\os$ is an isomorphism.
\end{proof}
\begin{cor}
\label{cor:surfaceisoprojectivebundle}Let $\S$ be a surface and
$\pi:S\to\C$ a surjective morphism onto a smooth curve such that
$\pi^{-1}\left(x\right)\iso\Pu$ for every $x\in\C$. Then, for every
section $H$ of $\pi$, there exists a $\C$-isomorphism $\S\iso\pc\left(\pi_{\ast}\os\left(H\right)\right)$.
\end{cor}

\begin{proof}
Let us set $\mathcal{E}=\pi_{\ast}\os\left(H\right)$. By definition of $H$
we have $H\cdot\pi^{-1}\left(x\right)=1$ for every point $x\in\C$,
then, in view of Lemma \ref{lem:2.1Hartshorne} we conclude that $\mathcal{E}$
is a rank 2 locally free sheaf on $\C$. We see that the natural morphism
\[
\pi^{\ast}\mathcal{E}=\pi^{\ast}\pi_{\ast}\os\left(H\right)\to\os\left(H\right)
\]
is an epimorphism: in view of Nakayama's lemma, it is sufficent to
check it on fibres. Let us consider, for every $x\in\C$, the commuting
diagram
\[
\xymatrix{\mathcal{E}\ot\cc\left(x\right)\ar[rr]\ar[dr]_{\alpha} &  & \os\left(H\right)_{|\pi^{-1}\left(x\right)}\iso\mathcal{O}_{\Pu}\left(1\right)\\
 & H^{0}\left(\pi^{-1}\left(x\right),\os\left(H\right)_{|\pi^{-1}\left(x\right)}\right)\ar[ur]_{\beta}
}
\]
The surjectivity of $\alpha$ follows again from \cite[Proposition 9.5.2]{Ke},
while the surjectivity of $\beta$ follows from the fact that $\mathcal{O}_{\Pu}\left(1\right)$
is very ample, thus, in particular, it is generated by global sections.
In view of Theorem \ref{thm:universalpropertyproj} we conclude that
there exists a unique $\C$-morphism $f:\S\to\pc\left(\mathcal{E}\right)$
such that $f^{\ast}\mathcal{O}_{\pc\left(\mathcal{E}\right)}\left(1\right)\iso\os\left(H\right)$.
Since $\os\left(H\right)$ is very ample on each fibre, $f$ is an
isomorphism on each fibre. Taking affine open sets in $\C$ small
enough to trivialize $\pc\left(\mathcal{E}\right)$, we see that $f$ is an
isomorphism.
\end{proof}
\begin{rem}
\label{rem:tautologicalnumericallyequivalent}If $p:\pc\left(\mathcal{E}\right)=\S\to\C$
is a $\Pu$-bundle and $H$ is a section, taking the identity $\S\to\pc\left(\mathcal{E}\right)$,
in view of Corollary \ref{cor:surfaceisoprojectivebundle}, we have
an isomorphism $\mathcal{O}_{\pc\left(\mathcal{E}\right)}\left(1\right)\iso\os\left(H\right)$.
\end{rem}

\begin{prop}
\label{prop:coomologiacommutapushforward}Let $\pi:\S=\pc\left(\mathcal{E}\right)\to\C$
be a $\Pu$-bundle over a smooth curve $\C$, and let $D\in\D\left(\S\right)$
such that for a fibre $F$ one has $D\cdot F\ge0$. Then, for every
$i\ge0$
\[
H^{i}\left(\S,\os\left(D\right)\right)\iso H^{i}\left(\C,\pi_{\ast}\os\left(D\right)\right).
\]
In particular, for every $\tilde{D}\in\D\left(\C\right)$ and for
every $i\ge0$, we have
\[
h^{i}\left(\S,\pi^{\ast}\oc\left(\tilde{D}\right)\right)=h^{i}\left(\C,\oc\left(\tilde{D}\right)\right).
\]
\end{prop}

\begin{proof}
Let $n=D\cdot F\ge0$, then, for every $x\in\C$ and for every $i>0$,
we have
\[
h^{i}\left(\pi^{-1}\left(x\right),\os\left(D\right)_{|\pi^{-1}\left(x\right)}\right)=h^{i}\left(\Pu,\mathcal{O}_{\Pu}\left(n\right)\right)=0.
\]
In view of \cite[Proposition 9.5.2]{Ke} we see that $R^{i}\pi_{\ast}\os\left(D\right)=0$
for every $i>0$. Given an injective resolution
\[
0\to\os\left(D\right)\to\mathcal{I}^{\bullet},
\]
shifting syzygies we see that
\[
0\to\pi_{\ast}\os\left(D\right)\to\pi_{\ast}\mathcal{I}^{\bullet}
\]
is an exact sequence. Since, in view of Lemma \ref{lem:flatness},
$\pi_{\ast}$ is a right adjoint of a left exact functor, it preserves
injective objects and then $\pi_{\ast}\mathcal{I}^{\bullet}$ yields
an injective resolution of $\pi_{\ast}\os\left(D\right)$, in particular
\[
H^{i}\left(\S,\os\left(D\right)\right)\iso H^{i}\left(\C,\pi_{\ast}\os\left(D\right)\right).
\]
To prove the last statement we note that $\pi^{\ast}\tilde{D}\cdot F=0$
and therefore for every $i\ge0$
\[
h^{i}\left(\S,\pi^{\ast}\oc\left(\tilde{D}\right)\right)=h^{i}\left(\C,\pi_{\ast}\pi^{\ast}\oc\left(\tilde{D}\right)\right).
\]
In view of Lemma \ref{lem:2.1Hartshorne} and projection formula we
have
\begin{eqnarray*}
\pi_{\ast}\pi^{\ast}\oc\left(\tilde{D}\right) & = & \pi_{\ast}\left(\os\ot\pi^{\ast}\oc\left(\tilde{D}\right)\right)\\
 & = & \pi_{\ast}\os\ot\oc\left(\tilde{D}\right)=\oc\left(\tilde{D}\right).
\end{eqnarray*}
\end{proof}
\begin{defn}
Let $\mathcal{E}$ be a rank 2 locally free sheaf on a smooth curve $\C$,
we say that $\mathcal{E}$ is $\emph{normalized}$ if $h^{0}\left(\C,\mathcal{E}\right)>0$
and $h^{0}\left(\C,\mathcal{E}\ot\L\right)=0$ for every $\L\in\pic\left(\C\right)$
with $\deg\left(\L\right)<0$.
\end{defn}

\begin{rem}
\label{rem:normalization}Let $\C$ be a smooth curve and $\mathcal{E}$ be
a rank 2 locally free sheaf on $\C$. In view of Lemma \ref{lem:hasezioneglobale},
there exists $n\in\N$ such that, for every $\L\in\pic\left(\C\right)$
with $\deg\left(\L\right)=n$, we have $h^{0}\left(\C,\mathcal{E}\ot\L\right)>0$.
Moreover, taking $\deg\left(\L\right)\ll0$ we get $h^{0}\left(\C,\mathcal{E}\ot\L\right)=0$,
and hence there exists $\M\in\pic\left(\C\right)$ such that $h^{0}\left(\C,\mathcal{E}\ot\M\right)>0$
and $h^{0}\left(\C,\mathcal{E}\ot\M\ot\L\right)=0$ for every $\L\in\pic\left(\C\right)$
with $\deg\left(\L\right)<0$. Therefore, given a geometrically ruled
surface
\[
\pi:\S=\pc\left(\mathcal{E}\right)\to\C,
\]
in view of Theorem \ref{thm:isomorphicPunobundles} we can always
assume $\mathcal{E}$ to be normalized. In this case we set $e\left(\S\right)=-\deg\left(\det\mathcal{E}\right)$.
\end{rem}

\begin{example}
Let $n\in\N$ and let us consider the sheaf $\mathcal{E}=\mathcal{O}_{\Pu}\op\mathcal{O}_{\Pu}\left(-n\right)$
over $\Pu$. We have
\[
h^{0}\left(\Pu,\mathcal{E}\right)=h^{0}\left(\Pu,\mathcal{O}_{\Pu}\right)+h^{0}\left(\Pu,\mathcal{O}_{\Pu}\left(-n\right)\right)=1,
\]
moreover, if $k>0$,
\[
h^{0}\left(\Pu,\mathcal{E}\left(-k\right)\right)=h^{0}\left(\Pu,\mathcal{O}_{\Pu}\left(-k\right)\right)+h^{0}\left(\Pu,\mathcal{O}_{\Pu}\left(-n-k\right)\right)=0,
\]
therefore $\mathcal{E}$ is normalized. It follows from Theorem \ref{thm:isomorphicPunobundles}
that $e\left(\ff\right)=n$ (see Example \ref{example: Hirzebruchsurfaces}).
\end{example}

\begin{thm}
\label{thm:possiblevalueseS}Let $\C$ be a smooth curve and $\mathcal{E}$
be a normalized rank 2 locally free sheaf on $\C$, then

\begin{enumerate}
\item if $\mathcal{E}$ is decomposable, then there exists $\L\in\pic\left(\C\right)$
with $\deg\left(\L\right)\le0$ such that $\mathcal{E}\iso\oc\op\L$. In particular
$e\left(\pc\left(\mathcal{E}\right)\right)\ge0$ and every positive value of
$e\left(\pc\left(\mathcal{E}\right)\right)$ is possible;
\item if $\mathcal{E}$ is indecomposable, then $-2g\left(\C\right)\le e\left(\pc\left(\mathcal{E}\right)\right)\le2g\left(\C\right)-2$.
\end{enumerate}
\end{thm}

\begin{proof}
$\,$

\begin{enumerate}
\item Let $\mathcal{E}\iso\L_{1}\op\L_{2}$ and suppose $\deg\left(\L_{1}\right)>0$,
then
\[
h^{0}\left(\C,\mathcal{E}\ot\L_{1}^{-1}\right)=h^{0}\left(\C,\oc\right)+h^{0}\left(\C,\L_{2}\ot\L_{1}^{-1}\right)\ge1
\]
which is a contradiction and therefore $\deg\left(\L_{1}\right)$,
$\deg\left(\L_{2}\right)\le0$. Since $h^{0}\left(\C,\mathcal{E}\right)>0$,
then, up to reordering, we have $\L_{1}\iso\oc$ and hence
\[
\mathcal{E}\iso\oc\op\L\,\,\,\,\mbox{with}\,\,\,\,\deg\left(\L\right)<0.
\]
The second part of the statement follows trivially.
\item Since $h^{0}\left(\C,\mathcal{E}\right)>0$ there exists $\L\in\pic\left(\C\right)$
and an exact sequence
\[
0\to\oc\to\mathcal{E}\to\L\to0.
\]
In view of Proposition \ref{prop:splittingcriterio} we have $H^{1}\left(\C,\L^{-1}\right)\neq0$
and hence, in view of \cite[Corollary 8.4.3]{Ke}, we deduce
\[
\deg\left(\L^{-1}\right)\le2g\left(\C\right)-2.
\]
Viceversa, if $\M\in\pic\left(\C\right)$ has $\deg\left(\M\right)=-1$,
then the exact sequence
\[
0\to\M\to\mathcal{E}\ot\M\to\L\ot\M\to0
\]
yields the exact sequence
\[
0\to H^{0}\left(\C,\M\right)\to H^{0}\left(\C,\mathcal{E}\ot\M\right)\to H^{0}\left(\C,\L\ot\M\right)\to H^{1}\left(\C,\M\right).
\]
Since $h^{0}\left(\C,\mathcal{E}\ot\M\right)=0$, we deduce that $h^{0}\left(\C,\M\right)=0$
and that $h^{0}\left(\C,\L\ot\M\right)\le h^{1}\left(\C,\M\right)$.
In view of riemann-Roch for curves applied to $\M$ we conclude that
$h^{1}\left(\C,\M\right)=g\left(\C\right)$ and that
\[
h^{0}\left(\C,\L\ot\M\right)\ge\deg\left(\L\right)-g,
\]
which gives $\deg\left(\L\right)\le2g\left(\C\right)$ and finally
\[
e\left(\pc\left(\mathcal{E}\right)\right)=-\deg\left(\L\right)\ge-2g\left(\C\right).
\]
\end{enumerate}
\end{proof}

\section{The classification problem}

\subsection{Numerical invariants}

In this section we introduce the principal algebraic and topological
numerical invariants of a smooth projective surface over the complex
numbers.
\begin{defn}[Algebraic invariants]
For a smooth projective surface $\S$ we set

\begin{itemize}
\item $q\left(\S\right)=h^{1}\left(\S,\os\right)$;
\item $p_{g}\left(\S\right)=h^{0}\left(\S,\omega_{\S}\right)=h^{2}\left(\S,\os\right)$;
\item $P_{n}\left(\S\right)=h^{0}\left(\S,\omega_{\S}^{n}\right)$ for $n\ge1$.
\end{itemize}
They are called, respectively, \emph{the irregularity, the geometric
genus} and the \emph{$n$-th plurigenus} of $\S$.
\end{defn}

\begin{rem}
\label{rem:potenze}Let $\ensuremath{s\in H^{0}\left(\S,\omega_{\S}^{n}\right)}$
be a non-zero section for some $n\ge1$, then, for every $k\in\N$,
we have a non-zero section $s^{k}\in H^{0}\left(\S,\omega_{\S}^{nk}\right)$.
This shows that, if $P_{n}\left(\S\right)\neq0$, then $P_{m}\left(\S\right)\neq0$
for every multiple of $n$. In particular, if $p_{g}\left(\S\right)\neq0$,
then $P_{n}\left(\S\right)\neq0$ for every $n\ge1$.
\end{rem}

\begin{defn}[Topological invariants]
For a smooth projective surface $\S$ defined over $\cc$$ $, we
set

\begin{itemize}
\item $b_{i}\left(\S\right)=\dim_{\mathbb{R}}H^{i}\left(\S,\mathbb{R}\right)$
for every $i=0,\dots,4$;
\item $\xt\left(\S\right)=\sum_{i=0}^{4}\left(-1\right)^{i}b_{i}\left(\S\right)$.
\end{itemize}
\end{defn}

\begin{rem}
Since our varieties are irreducible, in particular they are connected
and, in view of Poincaré duality, we have
\[
b_{0}\left(\S\right)=b_{4}\left(\S\right)=1\,\,\,\,\mbox{and}\,\,\,\,b_{1}\left(\S\right)=b_{3}\left(\S\right),
\]
thus $\xt\left(\S\right)=2-2b_{1}\left(\S\right)+b_{2}\left(\S\right)$.
\end{rem}

\begin{fact}
\label{fact:hodgeshows}For a smooth projective surface $\S$ defined
over $\cc$, Hodge theory shows that
\begin{eqnarray*}
b_{1}\left(\S\right) & = & 2h^{0}\left(\S,\Omega_{\S}^{1}\right)=2q\left(\S\right)\\
b_{2}\left(\S\right) & = & 2p_{g}\left(\S\right)+h^{1}\left(\S,\Omega_{\S}^{1}\right).
\end{eqnarray*}
Moreover the formula
\[
12\x\left(\os\right)=\ks^{2}+\xt\left(\S\right),
\]
called $\emph{Noether's formula}$, holds. See \cite{GH}.
\end{fact}

\begin{rem}
It is clear that the integers $P_{n}\left(\S\right)$ are birational
invariants for $n\ge1$. It follows from Fact \ref{fact:hodgeshows}
that also $q\left(\S\right)$ is a birational invariant.
\end{rem}

\subsection{Relatively minimal models}

From now on, all surfaces are assumed to be projective, smooth and
irreducible.
\begin{defn}
Let $\S$ be a surface. We say that $\S$ is a \emph{relatively minimal
model} if every birational morphism $\varphi:\S\to\S'$ is an isomorphism.
\end{defn}

The interest in relatively minimal models relies in the following
Proposition:
\begin{prop}
\label{prop:everysurfacesrelativelyminimalmodel}Every surface dominates
a relatively minimal model.
\end{prop}

\begin{proof}
Let $\S$ be a surface. If $\S$ is minimal, then we are done. Otherwise
there exists a birational morphism $\varphi_{1}:\S\to\S_{1}$ which
is not an isomorphism. If $\S_{1}$ is minimal, then the proposition
is proved, otherwise there exists a birational morphism $\varphi_{2}:\S_{1}\to\S_{2}$
which is not an isomorphism and so on. This process must end since
every map $\varphi_{i}$ can be factored into a finite number of blow-ups
and thus we have
\[
\rg\ns\left(\S\right)\gvertneqq\rg\ns\left(\S_{1}\right)\gvertneqq\dots\gvertneqq\rg\ns\left(\S_{i}\right)\gvertneqq\dots
\]
We conclude as $\rg\ns\left(\S\right)<\infty$.
\end{proof}
We need to introduce a definition:
\begin{defn}
A curve $\C$ on a surface $\S$ is called a $\left(-1\right)\emph{-curve}$
if $\C\iso\Pu$ and $\C^{2}=-1$.
\end{defn}

First we need a Lemma from topology:
\begin{lem}
\label{lem:lemmafromtopology}Let $X,Y$ be Hausdorff topological
spaces, $K\su X$ a compact subset and let $f:X\to Y$ be a continuous
map such that $f_{|K}$ is a homeomorphism and suppose that for every
$k\in K$ there exists a open neighborhood of $k$ over which $f$
is a homeomorphism. Then there exist $U\su X$ and $V\su Y$ open
neighborhoods of $K$ and $f\left(K\right)$ such that $f_{|U}$ gives
a homeomorphism $U\iso V$.
\end{lem}

\begin{thm}[Castelnuovo's contractibility criterion]
\label{thm:castelnuovocontractibility}Let $\S$ be a surface, and
let $E\su\S$ be a $\left(-1\right)$-curve. Then there exist a smooth
surface $\S'$ and a point $p\in\S'$ such that $\S\iso\bl_{p}\left(\S'\right)$
with exceptional curve $E$ (i.e. $\S$ is a blown-up surface).
\end{thm}

\begin{proof}
Let $H\in H^{0}\left(\S,\os\left(1\right)\right)$, let us note that,
in view of Theorem \ref{prop:amplemultiple}, we can suppose $H^{1}\left(\S,\os\left(H\right)\right)=0$.
Set $k=H\cdot E$ and $H'=H+kE$. Since $E\iso\Pu$, the invertible
sheaves on $E$ are characterized by their degree, so
\[
\begin{array}{c}
\os\left(E\right)_{|E}\iso\mathcal{O}_{E}\left(-1\right)\\
\os\left(H\right)_{|E}\iso\mathcal{O}_{E}\left(k\right)\\
\os\left(H'\right)_{|E}\iso\mathcal{O}_{E}
\end{array}
\]
For every $i=1,\dots,k$, let us tensorize
\[
0\to\os\left(-E\right)\to\os\to\mathcal{O}_{E}\to0
\]
with $\os\left(H+iE\right)$ to get the exact sequence
\begin{equation}
0\to\os\left(H+\left(i-1\right)E\right)\to\os\left(H+iE\right)\to\mathcal{O}_{E}\left(k-i\right)\to0.\label{eq:castelnuovorationality}
\end{equation}
Under our assumptions, we have an exact sequence
\[
0\to H^{0}\left(\S,\os\left(H\right)\right)\to H^{0}\left(\S,\os\left(H+E\right)\right)\to H^{0}\left(E,\mathcal{O}_{E}\left(k-1\right)\right)\to0.
\]
Let us suppose $H^{1}\left(\S,\os\left(H+\left(i-1\right)E\right)\right)=0$
for $i\le k-1$. As $E\iso\Pu$, looking at the cohomology sequence
associated to (\ref{eq:castelnuovorationality}), we deduce that $H^{1}\left(\S,\os\left(H+iE\right)\right)$
vanishes and hence, by induction, the sequence
\[
0\to H^{0}\left(\S,\os\left(H+\left(i-1\right)E\right)\right)\to H^{0}\left(\S,\os\left(H+iE\right)\right)\to H^{0}\left(E,\mathcal{O}_{E}\left(k-i\right)\right)\to0
\]
is exact for every $i=1,\dots,k$. Moreover, this shows that there
is a decomposition
\begin{equation}
H^{0}\left(\S,\os\left(H'\right)\right)\iso H^{0}\left(\S,\os\left(H\right)\right)\op\bigoplus_{i=0}^{k-1}H^{0}\left(E,\mathcal{O}_{E}\left(i\right)\right).\label{eq:castelnuovocontrdecomposition}
\end{equation}
Let $t\in H^{0}\left(\S,\os\left(E\right)\right)$ be such that $\left(t\right)_{0}=E$,
pick a base $\left\{ s_{0},\dots,s_{n}\right\} $ of $H^{0}\left(\S,\os\left(H\right)\right)$
and, for every $i$ with $1\le i\le k$, let $\left\{ a_{i,0},\dots,a_{i,k-i}\right\} $
be elements of $H^{0}\left(\S,\os\left(H+iE\right)\right)$ which
restrict to a base of $H^{0}\left(E,\mathcal{O}_{E}\left(k-i\right)\right)$.
In view of (\ref{eq:castelnuovocontrdecomposition}), the set
\[
\left\{ t^{k}s_{0},\dots,t^{k}s_{n},t^{k-1}a_{1,0},\dots,t^{k-1}a_{1,k-1},\dots,ta_{k-1,1},a_{k,0}\right\}
\]
gives a base of $H^{0}\left(\S,\os\left(H'\right)\right)$.

Let $\varphi:\S\dashrightarrow\PN$ be the associated rational map.
Since $H$ was a very ample divisor, the restriction $\varphi_{|\S-E}$
is a closed embedding and the curve $E$ is contracted to the point
$p=\left[0:\dots:0:1\right]$. To conclude, it suffices to check that
the point $p$ is smooth on the surface $\S'=\varphi\left(\S\right)$.
Let $U\su\S$ be the open subset defined by $a_{k,0}\neq0$ and set
\[
x=\frac{a_{k-1,0}}{a_{k,0}}\,\,\,\,\,\,\,\,\,\,\,\,y=\frac{a_{k-1,1}}{a_{k,0}},
\]
then $x,y$ are sections of $\mathcal{O}_{U}\left(-E\right)$ which, taking
$U$ eventually smaller, at no point of $U$ both vanish. Moreover
they restrict to a base of $H^{0}\left(E,\mathcal{O}_{E}\left(1\right)\right)$.
We can define morphisms $f:U\to\Pu$ and $g:U\to\A^{2}$ as
\[
\begin{array}{c}
f\left(P\right)=\left[x\left(P\right):y\left(P\right)\right]\\
g\left(P\right)=\left(t\left(P\right)x\left(P\right),t\left(P\right)y\left(P\right)\right),
\end{array}
\]
then $\left(g,f\right):U\to\A^{2}\times\Pu$ factors through $\bl_{0}\left(\A^{2}\right)$
and we have a natural commutative square
\[
\xymatrix{U\ar[d]_{\varphi_{|U}}\ar[r]^{h} & \bl_{0}\left(\A^{2}\right)\ar[d]^{\eta}\\
\varphi\left(U\right)\ar[r]_{\overline{h}} & \A^{2}
}
\]
where $h=\left(g,f\right)^{|\bl_{0}\left(\A^{2}\right)}$ and $\overline{h}$
is defined by the condition $\overline{h}\left(p\right)=\left(0,0\right)$.
Clearly $h$ induces an isomorphism from $E$ to the exceptional curve
of $\bl_{0}\left(\A^{2}\right)$. Note also that, for every $q\in E$,
the inverse image under $h$ of a system of local coordinates at $h\left(q\right)$
is a system of local coordinates at $q$, in fact let $z_{1},z_{2}$
be coordinates on $\A^{2}$ and $Z_{1},Z_{2}$ be coordinates on $\Pu$.
Then $\bl_{0}\left(\A^{2}\right)\su\A^{2}\times\Pu$ is defined by
$z_{1}Z_{2}=z_{2}Z_{1}$. Let us suppose that $x\left(q\right)=0$
and $y\left(q\right)=1$, thus we can take $z_{2}$ and $\frac{Z_{1}}{Z_{2}}$
as local coordinates at $h\left(q\right)$. Pulling back we obtain
\[
h^{\ast}\left(z_{2}\right)=ty\,\,\,\,\,\,\,\,\,\,\,h^{\ast}\left(\frac{Z_{1}}{Z_{2}}\right)=\frac{x}{y}.
\]
From the choice of $t,x$ and $y$, it follows that $h^{\ast}\left(z_{2}\right)$
vanishes on $E$ with order 1, while $h^{\ast}\left(\frac{Z_{1}}{Z_{2}}\right)$
is a local coordinate on $q$ when restricted to $E$.

Let us move to the complex topology. We claim that, after replacing
$U$ with a smaller open set, the map $\overline{h}_{|U}$ gives an
isomorphism onto its image. We see that the claim implies the theorem
because the image of $\overline{h}_{|U}$ is an open set in $\A^{2}$
which contains $\left(0,0\right)$, thus $p$ turns out to be a smooth
point.

In view of Lemma \ref{lem:lemmafromtopology}, there exist an open
neighborhood $V$ of $E$ and an open neighborhood $W$ of $h\left(E\right)$
such that $h$ gives an homeomorphism $V\iso W$. Let us consider
$\varphi\circ h_{|W}^{-1}$: this function contracts the exceptional
curve of $\bl_{0}\left(\A^{2}\right)$ to $p$, therefore, in view
of the universal property of blowing-down, it factors as $\psi\circ\eta$,
where $\psi$ is a morphism which must be the inverse of $\overline{h}$
. So $\varphi\left(U\right)$ is isomorphic to the open neighbourhood
$\eta(W)$ of the origin in $\A^{2}$ .
\end{proof}
\begin{cor}
\label{cor:characterizationrelminimal}A surface $\S$ is a relatively
minimal model if and only if it does not contain $\left(-1\right)$-curves.
\end{cor}

\begin{proof}
If $E\su\S$ is a $\left(-1\right)$-curve, in view of Castelnuovo's
Theorem \ref{thm:castelnuovocontractibility}, there exists a birational
morphism $\S\to\S'$ which contracts $E$, thus it is not an isomorphism
and hence $\S$ is not a relatively minimal model.

If $\S$ is not a relatively minimal model, then there exist a surface
$\S'$ and a birational morphism $f:\S\to\S'$ which is not an isomorphism.
Then
\[
f=\epsilon_{1}\circ\dots\circ\epsilon_{n}
\]
 is a finite composition of blowing ups and the exceptional curve
of the last one gives a $\left(-1\right)$-curve on $\S$.
\end{proof}

\subsection{Ruled surfaces}
\begin{defn}
A surface $\S$ is said to be $\emph{ruled}$ if there exist a smooth
curve $\C$ such that $\S$ is birationally equivalent to $\C\times\Pu$.
\end{defn}

\begin{example}
It is clear that for every smooth curve $\C$, the product $\C\times\Pu$
is a ruled surface. Moreover rational surfaces are ruled, since $\Pd$
is birationally equivalent to $\Pu\times\Pu$.
\end{example}

$\,$
\begin{example}
For every rank 2 vector bundle over a smooth curve $\C$, the $\Pu$-bundle
$\pc\left(\mathcal{E}\right)$ is a ruled surface as it is birationally equivalent
to $\C\times\Pu$.
\end{example}

\begin{defn}
Let $\C$ be a smooth curve. We say that a surface $\S$ is $\emph{geometrically ruled over \ensuremath{\C}}$
if there exists a surjective morphism $\pi:\S\to\C$ whose fibres
are smooth rational curves.
\end{defn}

\begin{thm}[Noether-Enriques]
\label{thm:noetherenriques}Let $\S$ be a surface and $\pi:\S\to\C$
a surjective morphism onto a smooth curve. If there exists $x\in\C$
such that $\pi^{-1}\left(x\right)\iso\Pu$, then there exists an open
subset $U\su\C$ containing $x$ and a $U$-isomorphism $\pi^{-1}\left(U\right)\to U\times\Pu$.
\end{thm}

\begin{proof}
In view of Lemma \ref{lem:flatness} we deduce that $\pi$ is a flat
morphism, moreover, in view of \cite[Proposition 9.5.1]{Ke} and
the fact that $h^{0}\left(\pi^{-1}\left(x\right),\mathcal{O}_{\S|\pi^{-1}\left(x\right)}\right)=1$
we deduce that the general fibre is connected. Since we work over
the field $\cc$ of complex numbers, the general fibre is also smooth
and hence irreducible. Moreover we have $h^{1}\left(\pi^{-1}\left(x\right),\mathcal{O}_{\S|\pi^{-1}\left(x\right)}\right)=0$
and hence, in view again of \cite[Proposition 9.5.1]{Ke} and Remark
\ref{rem:geometricarithmeticgenus}, we deduce that there exists $U\su\C$
open set such that $\pi^{-1}\left(y\right)\iso\Pu$ for every $y\in U$.
We conclude in view of Proposition \ref{prop:existssection} and Corollary
\ref{cor:surfaceisoprojectivebundle}.
\end{proof}
\begin{rem}
It follows in particular from Theorem \ref{thm:noetherenriques} that
geometrically ruled surfaces are actually ruled surfaces.
\end{rem}

\begin{lem}
\label{lem:lemmaIIInove}Let $\S$ be a surface and $p:\S\to\C$ a
surjective morphism onto a smooth curve with connected fibres. If
$F=\sum n_{i}C_{i}$ is a reducible fibre, then $C_{i}^{2}<0$ for
every $i$.
\end{lem}

\begin{proof}
We have, in view of Corollary \ref{cor:fibersnumerically}, that
\[
n_{i}C_{i}^{2}=C_{i}\cdot\left(F-\sum_{j\neq i}n_{j}C_{j}\right)=-\sum_{j\neq i}n_{j}C_{i}\cdot C_{j}.
\]
Since $F$ is connected we deduce that $C_{i}\cdot C_{j}\ge0$ for
every $j\neq i$ and there exists, at least, one index $j$ such that
$C_{i}\cdot C_{j}>0$. It follows that $n_{i}C_{i}^{2}<0$ and thus
$C_{i}^{2}<0$.
\end{proof}
\begin{lem}
\label{lem:generalefibregeometricallyruled}Let $\S$ be relatively
minimal surface. If there exists a smooth curve $\C$ and a surjective
morphism $\pi:\S\to\C$ with general fibre isomorphic to $\Pu$, then
$\S$ is geometrically ruled by $\pi$.
\end{lem}

\begin{proof}
In view of Corollary \ref{cor:surfaceisoprojectivebundle}, it is
enough to show that $\emph{every}$ fibre of $\pi$ is a smooth rational
curve. First we show that every fibre is reduced and connected: in
view of Theorem \ref{thm:noetherenriques}, $p$ has a local section
$H$. Since $\C$ is a smooth curve, this section lifts to a global
section and hence $H\cdot F=1$ for every fibre.

In view of Corollary \ref{cor:fibersnumerically} and Example \ref{ex:fibresquaredzero}
we have $F^{2}=0$ and $F\cdot\ks=-2$ for every fibre $F$ of $\pi$.
In view of Corollary \ref{cor:genusformula} we conclude that, for
every irreducible fibre $F$, we have $p_{a}\left(F\right)=0$ and
thus, in view of Remark \ref{rem:geometricarithmeticgenus}, $F\iso\Pu$
for every irreducible fibre.

Let $F=\sum n_{i}C_{i}$ be a reducible fibre, in view of Corollary
\ref{cor:fibersnumerically} we deduce that $F\cdot\ks=-2$. From
Corollary \ref{cor:genusformula} and Lemma \ref{lem:lemmaIIInove},
we have that $C_{i}\cdot\ks\ge-1$ for every $i$, with equality if
and only if $p_{a}\left(C_{i}\right)=0$ and $C_{i}^{2}=-1$, i.e.
if $C_{i}$ is a $\left(-1\right)$-curve. Under our hypothesis, in
view of Corollary \ref{cor:characterizationrelminimal}, this is not
possible, thus $C_{i}\cdot\ks\ge0$, but this is a contradiction with
the fact that $F\cdot\ks=-2$ .
\end{proof}
\begin{thm}
\label{thm:minimalmodelirrational}Let $\C$ be a smooth $\emph{irrational}$
curve. Then the relatively minimal models for $\C\times\Pu$ are $\Pu$-bundles
$\pc\left(\mathcal{E}\right)$.
\end{thm}

\begin{proof}
Let $\S$ be a relatively minimal model for $\C\times\Pu$ and let
$\varphi:\S\ra\C\times\Pu$ be a birational map. We set $f=q\circ\varphi$,
where $q:\C\times\Pu\to\C$ is the first projection. We can solve
the singularities of $f$ and get a commutative diagram
\[
\xymatrix{ & \tilde{\S}\ar[dr]^{g}\ar[dl]_{\sigma}\\
\S\ar@{-->}[rr]_{f} &  & \C
}
\]
where $\sigma=\epsilon_{1}\circ\dots\circ\epsilon_{n}$ is the composition
of a finite number of blow ups, moreover we can assume that the number
$n$ is the minimal number of blow ups for a diagram of this type
to exist. Let us suppose that $n>0$, and let $E$ be the exceptional
curve of the last blow up. If $g\left(E\right)=\C$, then $\C$ is
dominated by a smooth rational curve, and hence $\C$ is rational,
which contradicts our assumptions. It follows that $g\left(E\right)$
is a point and thus we can factor $g=g'\circ\epsilon_{n}$, which
is impossible for the minimality of $n$. Thus $n=0$ and $f$ is
a morphism with general fibre isomorphic to $\Pu$. We conclude in
view of Lemma \ref{lem:generalefibregeometricallyruled}.
\end{proof}
We conclude this section computing some invariant for ruled and geometrically
ruled surfaces.
\begin{prop}
\label{prop:picgeometricallyruled}Let $p:\pc\left(\mathcal{E}\right)=\S\to\C$
be a geometrically ruled surface with $\mathcal{E}$ normalized, let $\C_{0}$
be a section and $F$ be a fibre. Then

\begin{itemize}
\item $\pic\left(\S\right)\iso p^{\ast}\pic\left(\C\right)\op\Z\cdot\left[\C_{0}\right]$;
\item $\num\left(\S\right)\iso\Z\cdot\left[F\right]\op\Z\cdot\left[\C_{0}\right]$
with the relations
\[
\C_{0}\cdot F=1,\,\,\,\,\C_{0}^{2}=\deg\left(\det\mathcal{E}\right)\,\,\,\,\mbox{and}\,\,\,\,F^{2}=0.
\]
\end{itemize}
\end{prop}

\begin{proof}
By definition of $H$ and in view of Example \ref{ex:fibresquaredzero}
we have
\[
\C_{0}\cdot F=1\,\,\,\,\mbox{and}\,\,\,\,F^{2}=0.
\]
In view of Lemma \ref{lem:rango2sequenzaesatta}, there exists an
exact sequence
\[
0\to\L\to\mathcal{E}\to\mathcal{N}\to0
\]
where $\L,\mathcal{N}\in\pic\left(\C\right)$ which gives, in view
of Lemma \ref{lem:flatness}, the exact sequence
\[
0\to p^{\ast}\L\to p^{\ast}\mathcal{E}\to p^{\ast}\mathcal{N}\to0.
\]
Therefore we have
\[
c_{2}\left(p^{\ast}\mathcal{E}\right)=p^{\ast}\L\cdot p^{\ast}\mathcal{N}=0
\]
and, in view of sequence (\ref{eq:deftautological}) and Remark \ref{rem:secondclass},
we conclude that $\C_{0}\cdot\left[\M\right]=0$. Moreover (\ref{eq:deftautological})
shows also that
\[
p^{\ast}\det\mathcal{E}\iso\M\ot\os\left(1\right).
\]
If $E=\left[\det\mathcal{E}\right]\in\pic\left(\C\right)$, then
\[
p^{\ast}E\li\left[\M\right]+H,
\]
which gives $\C_{0}^{2}=\C_{0}\cdot p^{\ast}E=\deg\left(\det\mathcal{E}\right)$.

Let $D\in\pic\left(\S\right)$ and let $n=D\cdot F$. We set $D'=D-n\C_{0}$,
thus $F\cdot D'=0$ and hence it is enough to show that $D'\in p^{\ast}\pic\left(\C\right)$.
Let $D_{n}=D'+nF$, then $F\cdot D_{n}=0$, $D_{n}^{2}=D^{2}$ and
\[
D_{n}\cdot\ks=D'\cdot\ks+nF\cdot\ks=D'\cdot\ks-2n.
\]
Moreover, for $n\gg0$, we have $h^{2}\left(\S,\os\left(D_{n}\right)\right)=h^{0}\left(\S,\os\left(\ks-D_{n}\right)\right)=0$
and hence, in view of Riemann-Roch and Lemma \ref{lem:genuszero},
we have
\begin{eqnarray*}
h^{0}\left(\S,\os\left(D_{n}\right)\right) & \ge & 1+\frac{1}{2}D_{n}^{2}-\frac{1}{2}D_{n}\cdot\ks\\
 & = & n+1+\frac{1}{2}D^{2}-\frac{1}{2}D\cdot\ks,
\end{eqnarray*}
this shows that, for $n\gg0$, there exists $E\in|D_{n}|$. Since
$F\cdot E=0$, it follows that $E$ is a fibre of $p$. We conclude
that also $D'$ is a fibre of $p$. The last statement follows from
Corollary \ref{cor:fibersnumerically}.
\end{proof}
\begin{prop}
\label{prop:calcoloconkunneth}Let $\S$ be surface birationally equivalent
to $\C_{1}\times\C_{2}$ where $\C_{i}$ are smooth curves, then

\begin{enumerate}
\item $q\left(\S\right)=g\left(\C_{1}\right)+g\left(\C_{2}\right);$
\item $P_{n}\left(\S\right)=P_{n}\left(\C_{1}\right)\cdot P_{n}\left(\C_{2}\right).$
\end{enumerate}
\end{prop}

\begin{proof}
Since these constants are invariant under birational equivalence,
we can assume $\S=\C_{1}\times\C_{2}$.

\begin{enumerate}
\item Let $p_{i}:\S\to\C_{i}$, $i=1,2$ be the projection. Then it is clear
that $\Omega_{\S}^{1}=p_{1}^{\ast}\Omega_{\C_{1}}^{1}\op p_{2}^{\ast}\Omega_{\C_{2}}^{1}$
and hence, in view of Künneth formula \cite[Proposition 9.2.4]{Ke}
\begin{eqnarray*}
h^{0}\left(\S,\Omega_{\S}^{1}\right) & = & h^{0}\left(\C_{1}\times\C_{2},p_{1}^{\ast}\Omega_{\C_{1}}^{1}\op p_{2}^{\ast}\Omega_{\C_{2}}^{1}\right)\\
 & = & h^{0}\left(\C_{1}\times\C_{2},p_{1}^{\ast}\Omega_{\C_{1}}^{1}\ot\os\right)+h^{0}\left(\C_{1}\times\C_{2},\os\ot p_{2}^{\ast}\Omega_{\C_{2}}^{1}\right)\\
 & = & h^{0}\left(\C_{1}\times\C_{2},p_{1}^{\ast}\Omega_{\C_{1}}^{1}\ot p_{2}^{\ast}\mathcal{O}_{\C_{2}}\right)+h^{0}\left(\C_{1}\times\C_{2},p_{1}^{\ast}\mathcal{O}_{\C_{1}}\ot p_{2}^{\ast}\Omega_{\C_{2}}^{1}\right)\\
 & = & h^{0}\left(\C_{1},\Omega_{\C_{1}}^{1}\right)+h^{0}\left(\C_{2},\Omega_{\C_{2}}^{1}\right).
\end{eqnarray*}
\item It follows again in view of of Künneth formula \cite[Proposition 9.2.4]{Ke}
that
\[
h^{0}\left(\S,\omega_{\S}^{n}\right)=h^{0}\left(\C_{1},\omega_{\C_{1}}^{n}\right)\cdot h^{0}\left(\C_{2},\omega_{\C_{2}}^{n}\right).
\]
\end{enumerate}
\end{proof}
\begin{cor}
\label{cor:invariantsruled}Let $\S$ be a ruled surface over $\C$.
Then $q\left(\S\right)=g\left(\C\right)$ and $P_{n}\left(\S\right)=0$
for every $n\ge1$.
\end{cor}

\begin{lem}
\label{lem:CUBICSCROLLinvariante}Let $\S=\pc\left(\mathcal{E}\right)\to\C$
be a geometrically ruled surface over an elliptic curve with $\mathcal{E}$
indecomposable and normalized. Then $e\left(\S\right)\neq-2$.
\end{lem}

\begin{proof}
Let us suppose $e\left(\S\right)=-2$, then there exist $P,Q\in\C$
and a non-splitting exact sequence
\[
0\to\oc\to\mathcal{E}\to\oc\left(P+Q\right)\to0,
\]
moreover $h^{0}\left(\C,\oc\left(P+Q\right)\right)\neq0$, so we can
find $R,S\in\C$ such that
\[
R+S\li P+Q.
\]
Since $H^{0}\left(\C,\mathcal{E}\left(-R\right)\right)=0$ and in view of Proposition
\ref{prop:splittingcriterio}, the connecting morphisms
\begin{eqnarray*}
H^{0}\left(\C,\oc\left(S\right)\right) & \overset{\gamma_{R,S}}{\longrightarrow} & H^{1}\left(\C,\oc\left(-R\right)\right)\\
H^{0}\left(\C,\oc\right)=\cc & \overset{\delta}{\longrightarrow} & H^{1}\left(\C,\oc\left(-P-Q\right)\right)
\end{eqnarray*}
are injective. Let $f_{S}:\oc\to\oc\left(S\right)$ be the morphism
corresponding to a global section, then the maps
\[
\alpha_{S}=H^{0}\left(\C,f\right)\,\,\,\,\mbox{and}\,\,\,\,\beta_{R}=H^{1}\left(\C,f\ot\oc\left(-P-Q\right)\right)
\]
fit into a commutative diagram
\[
\xymatrix{H^{0}\left(\C,\oc\right)=\cc\ar[r]^{\delta}\ar[d]_{\alpha_{S}} & H^{1}\left(\C,\oc\left(-P-Q\right)\right)\ar[d]^{\beta_{R}}\\
H^{0}\left(\C,\oc\left(S\right)\right)\ar[r]_{\gamma_{R,S}} & H^{1}\left(\C,\oc\left(-R\right)\right)
}
\]
In view of Serre duality, we have an induced map
\[
\tilde{\beta}_{R}:H^{0}\left(\C,\oc\left(R\right)\right)\to H^{0}\left(\C,\oc\left(P+Q\right)\right)
\]
whose image corresponds to the point $R+S\in|P+Q|$. In view of this
duality,
\[
\delta\left(1\right)\in H^{1}\left(\C,\oc\left(-P-Q\right)\right)
\]
corresponds to a divisor $D_{\delta}\in|P+Q|$, therefore $D_{\delta}=A+B$
for $A,B\in\C$. Then we have
\[
\gamma_{A,B}\circ\alpha_{B}=\beta_{A}\circ\delta=0,
\]
since $\alpha_{A}\neq0$, this contradicts the injectivity of $\gamma_{A,B}$
.
\end{proof}
\begin{thm}
\label{thm:CUBICSCROLLunichepossibilita}Let $\S=\pc\left(\mathcal{E}\right)\to\C$
be a geometrically ruled surface over an elliptic curve with $\mathcal{E}$
indecomposable and normalized. Then $e\left(\S\right)=0$ or $-1$
and there exists exactly one such surface for each of these values
of $e\left(\S\right)$.
\end{thm}

\begin{proof}
In view of Theorem \ref{thm:possiblevalueseS} and Lemma \ref{lem:CUBICSCROLLinvariante},
the only possibilities for $e\left(\S\right)$ are $-1$ and $0$.

\begin{description}
\item [{$e\left(\S\right)=0$}] There exists a non-splitting exact sequence
\[
0\to\oc\to\mathcal{E}\to\L\to0
\]
with $\L\in\pic\left(\C\right)$ and $\deg\left(\L\right)=0$. In
particular, in view of Proposition \ref{prop:splittingcriterio},
we have $H^{1}\left(\C,\L^{-1}\right)\neq0$. In view of Serre duality
and since $\C$ is an elliptic curve, we conclude that $h^{0}\left(\C,\L\right)\neq0$
and hence $\L\iso\oc$. Viceversa, since $h^{1}\left(\C,\oc\right)=1$
there exists a unique non trivial extension
\begin{equation}
0\to\oc\to\mathcal{E}\to\oc\to0,\label{eq:estensionenonbanalecaso1}
\end{equation}
we see that such an $\mathcal{E}$ is normalized. Clearly $h^{0}\left(\C,\mathcal{E}\right)\neq0$,
moreover, if $\M\in\pic\left(\C\right)$ is such that $\deg\left(\M\right)<0$,
then tensoring (\ref{eq:estensionenonbanalecaso1}) with $\M$ we
see that $h^{0}\left(\C,\mathcal{E}\ot\M\right)=0$. Let us suppose that $\mathcal{E}$
is decomposable: in view of Theorem \ref{thm:possiblevalueseS}, there
exists $\L\in\pic\left(\C\right)$ with $\deg\left(\L\right)<0$ and
$\mathcal{E}\iso\oc\op\L$. In view of (\ref{eq:estensionenonbanalecaso1})
we get
\[
\oc\iso\bigwedge^{2}\mathcal{E}\iso\L
\]
which is not possible .
\item [{$e\left(\S\right)=-1$}] Since every invertible sheaf of degree
1 on $\C$ is of the form $\oc\left(P\right)$ for some $P\in\C$,
we get a non-splitting exact sequence
\[
0\to\oc\to\mathcal{E}\to\oc\left(P\right)\to0.
\]
We see that, for each $P\in\C$, there exists such a normalized sheaf
and it is unique up to isomorphism. In view of Serre duality and since
$\C$ is an elliptic curve, we conclude that $h^{1}\left(\C,\oc\left(-P\right)\right)=h^{0}\left(\C,\oc\left(P\right)\right)=1$,
then there exists a unique non trivial extension
\begin{equation}
0\to\oc\to\mathcal{E}\to\oc\left(P\right)\to0.\label{eq:estensionenonbanalecaso2}
\end{equation}
As before we see that $h^{0}\left(\C,\mathcal{E}\right)\neq0$. Let $\M\in\pic\left(\C\right)$
such that $\deg\left(\M\right)<0$, then we have an exact sequence
\[
0\to\M\to\mathcal{E}\ot\M\to\M\left(P\right)\to0.
\]
If $\M\iso\oc\left(-P\right)$, then cohomology yields an exact sequence
\[
0\to H^{0}\left(\C,\oc\left(-P\right)\right)\to H^{0}\left(\C,\mathcal{E}\left(-P\right)\right)\to H^{0}\left(\C,\oc\right)\overset{\delta}{\longrightarrow}H^{1}\left(\C,\oc\left(-P\right)\right).
\]
By construction of $\mathcal{E}$ and in view of Proposition \ref{prop:splittingcriterio}
we deduce that $\delta\left(1\right)\neq0$ and hence
\[
h^{0}\left(\C,\mathcal{E}\left(-P\right)\right)=h^{0}\left(\C,\oc\left(-P\right)\right)=0.
\]
For $\M\not\iso\oc\left(-P\right)$, then
\[
h^{0}\left(\C,\M\right)=h^{0}\left(\C,\M\left(P\right)\right)=0,
\]
and hence also $h^{0}\left(\C,\mathcal{E}\ot\M\right)=0$. It is enough to
show that, repeating the above construction for a point $Q\neq P$,
then we get a sheaf $\mathcal{E}'$ with $\mathcal{E}\iso\mathcal{E}'\ot\M$ for some $\M\in\pic\left(\C\right)$.
Let us note that $h^{0}\left(\C,\oc\left(P+Q\right)\right)=2$ and,
in view of Riemann-Hurwitz formula, we get a degree 2 map $\C\to\Pu$
which ramifies in 4 points. If $R$ is such a point, then $2R\li P+Q$.
Let us consider the exact sequence
\[
0\to\oc\left(R-P\right)\to\mathcal{E}\left(R-P\right)\to\oc\left(R\right)\to0
\]
which gives that $h^{0}\left(\C,\mathcal{E}\left(R-P\right)\right)\neq0$ and
then there exists an exact sequence
\[
0\to\oc\to\mathcal{E}\left(R-P\right)\to\M\to0
\]
with $\M\in\pic\left(\C\right)$. Then
\[
\M\iso\left(\bigwedge^{2}\mathcal{E}\right)\ot\oc\left(2R-2P\right)\iso\oc\left(P\right)\ot\oc\left(2R-2P\right)=\oc\left(2R-P\right)
\]
and hence $\M\iso\oc\left(Q\right)$. In view of the unicity of the
non trivial extension, we conclude that $\mathcal{E}'\iso\mathcal{E}\left(R-P\right)$.
\end{description}
\end{proof}

\subsection{Rational surfaces}

Here we focus on a special class of ruled surfaces: those which are
ruled over $\Pu$. Since $\Pd$ and $\Pu\times\Pu$ are birationally
equivalent, it follows that these surfaces are actually the rational
ones.
\begin{rem}
\label{rem:FngeomrigatesuPuno}It follows from Theorem \ref{thm:noetherenriques}
and Example \ref{example: Hirzebruchsurfaces} that the only geometrically
ruled surfaces over $\Pu$ are the surfaces $\mathbb{F}_{n}$ for
$n\ge0$. Moreover, in view of Proposition \ref{prop:calcoloconkunneth},
a ruled surface $\S$ is rational if and only if $q\left(\S\right)=0$.
\end{rem}

\begin{prop}
\label{prop:proprietaFn}Let $n\ge0$ be an integer and let us consider
$\pi:\ff\to\Pu$. If $H$ is a section and $F$ a fibre of $\pi$,
then

\begin{enumerate}
\item $\pic\left(\ff\right)=\Z\cdot\left[H\right]\op\Z\cdot\left[F\right]$
with $H^{2}=n$, $F^{2}=0$ and $H\cdot F=1$;
\item if $n>0$, there exists a unique curve $\C_{0}\su\ff$ with negative
self-intersection, in particular $\C_{0}^{2}=-n$.
\end{enumerate}
\end{prop}

\begin{proof}
Since $\pic\left(\Pu\right)$ is cyclic generated by the class of
a point, its image under the group homomorphism $p^{\ast}$ is cyclic
generated by the class of a fibre, then $1)$ follows from Proposition
\ref{prop:picgeometricallyruled}. To see $2)$, in view of Theorem
\ref{thm:isomorphicPunobundles}, we can suppose $\ff=\mathbb{P}_{\Pu}\left(\mathcal{O}_{\Pu}\op\mathcal{O}_{\Pu}\left(n\right)\right)$.
Let us consider the first projection
\[
\mathcal{O}_{\Pu}\op\mathcal{O}_{\Pu}\left(n\right)\to\mathcal{O}_{\Pu}.
\]
In view of Theorem \ref{thm:universalpropertyproj}, there exists
a unique morphism $s:\Pu\to\ff$ such that $ $$\pi\circ s=\mathrm{id}_{\Pu}$
and $s^{\ast}\mathcal{O}_{\ff}\left(1\right)=\mathcal{O}_{\Pu}$. Let $\C_{0}=s\left(\Pu\right)$,
then $\C_{0}\li tH+kF$. Since
\[
1=\C_{0}\cdot F=\left(tH+kF\right)\cdot F=t
\]
and
\[
H\cdot\C_{0}=\deg\left(\mathcal{O}_{\Pu}\right)=0
\]
we conclude that $H^{2}+k=0$, which, in view of Proposition \ref{prop:picgeometricallyruled},
yields $k=-n$ and then $\C_{0}^{2}=-n$. Let now $\C\su\ff$ be an
irreducible curve such that $\C_{0}\cdot\C\ge0$, then $\C\li\alpha H+\beta F$.
Since $\C_{0}\li H-nF$, we conclude that $\beta\ge0$ and, since
$\C\cdot F\ge0$, we have $\beta\ge0$. It follows that $\C^{2}=\alpha^{2}n+2\alpha\beta\ge0$.
\end{proof}
\begin{rem}
\label{rem:altreproprietaFn}It follows from Proposition \ref{prop:proprietaFn}
that $\ff\iso\mathbb{F}_{m}$ if and only if $n=m$ and, in view of
Corollary \ref{cor:characterizationrelminimal}, that these surfaces
are relatively minimal except for the case when $n=1$.
\end{rem}

\begin{prop}
\label{prop:F1scoppiamento}$\mathbb{F}_{1}$ is isomorphic to $\bl_{p}\left(\Pd\right)$.
\end{prop}

\begin{proof}
Let $\S=\bl_{p}\left(\Pd\right)$ with exceptional divisor $E$ and
let us consider the projection $\pi_{p}:\Pd\ra\Pu$. Then $\pi_{p}$
induces a morphism $\S\to\Pu$ with fibres isomorphic to $\Pu$. It
follows from Corollary \ref{cor:surfaceisoprojectivebundle} that
$\S$ is a $\Pu$-bundle. Since $E^{2}=-1$, we conclude in view of
Remark \ref{rem:FngeomrigatesuPuno} and Proposition \ref{prop:proprietaFn}.
\end{proof}
\begin{lem}
\label{lem:RATIONALSCROLLShasezioniglobali}Let $n\ge0$ be an integer,
then $h^{0}\left(\ff,\mathcal{O}_{\ff}\left(\C_{0}+kF\right)\right)\neq0$
if and only if $k=0$ or $k\ge n$. In particular there exists $\C_{1}\li\C_{0}+nF$
with $\C_{1}\cap\C_{0}=\emptyset$.
\end{lem}

\begin{proof}
In view of \ref{thm:universalpropertyproj}, giving a section $D\li\C_{0}+kF$
is equivalent to giving an epimorphism
\[
\mathcal{O}_{\Pu}\op\mathcal{O}_{\Pu}\left(-n\right)\to\mathcal{O}_{\Pu}\left(k-n\right).
\]
If $k-n<0$, then $\mathcal{O}_{\Pu}\left(k-n\right)$ has no global section,
thus, in view of Nakayama's Lemma, an epimorphism $\mathcal{O}_{\Pu}\op\mathcal{O}_{\Pu}\left(-n\right)\to\mathcal{O}_{\Pu}\left(k-n\right)$
yields an isomorphism $\mathcal{O}_{\Pu}\left(-n\right)\iso\mathcal{O}_{\Pu}\left(k-n\right)$
and hence $k=0$. This case corresponds to the section $D\li\C_{0}$.
If $k-n\ge0$, taking disjoint effective divisors of degree $k-n$
and $k$ on $\Pu$ we get two non-zero morphisms
\[
\begin{array}{ccc}
\mathcal{O}_{\Pu} & \to & \mathcal{O}_{\Pu}\left(k-n\right)\\
\mathcal{O}_{\Pu}\left(-n\right) & \to & \mathcal{O}_{\Pu}\left(k-n\right)
\end{array}
\]
which induce an epimorphism $\mathcal{O}_{\Pu}\op\mathcal{O}_{\Pu}\left(-n\right)\to\mathcal{O}_{\Pu}\left(k-n\right).$
In view of Proposition\ref{prop:proprietaFn} we have $\left(\C_{0}+nF\right)^{2}=0$
and the second part of the statement follows.
\end{proof}
\begin{thm}
\label{thm:RATIONALSCROLLSveryample}Let $n\ge0$ be an integer and
let $\Lambda_{k}=\C_{0}+kF\in\D\left(\ff\right)$. Then the linear
system $|\Lambda_{k}|$ is very ample on $\ff$ if and only if $k>n$.
\end{thm}

\begin{proof}
If $\Lambda_{k}$ is very ample, then $\Lambda_{k}^{2}>0$ and, in
view of Proposition\ref{prop:proprietaFn}, this gives $k>n$. Viceversa,
let us suppose $k>n$. Let $P,Q\in\ff$ be two points and $t$ a tangent
direction at $P$ and let us suppose $P,Q$ or $P,t$ lie in the same
fibre $F$. Since $F\iso\Pu$ and $\Lambda_{k}\cdot F=1$, $\Lambda_{k|F}$
is very ample on $F$ and thus, to separate $P,Q$ or $P,t$ it is
enough to check that the restriction map
\[
H^{0}\left(\ff,\mathcal{O}_{\ff}\left(\Lambda_{k}\right)\right)\to H^{0}\left(F,\mathcal{O}_{F}\left(1\right)\right)
\]
is surjective. Let us consider the exact sequence
\[
0\to\mathcal{O}_{\ff}\left(-F\right)\to\mathcal{O}_{\ff}\to\mathcal{O}_{F}\to0,
\]
tensoring with $\mathcal{O}_{\ff}\left(\Lambda_{k}\right)$ and taking cohomology,
we see that it is enough to prove that $H^{1}\left(\ff,\mathcal{O}_{\ff}\left(\Lambda_{k}-F\right)\right)=0$.
Since $\left(\Lambda_{k}-F\right)\cdot F=1$, in view of Proposition
\ref{prop:coomologiacommutapushforward} we have
\[
h^{1}\left(\ff,\mathcal{O}_{\ff}\left(\Lambda_{k}-F\right)\right)=h^{1}\left(\Pu,\pi_{\ast}\mathcal{O}_{\ff}\left(\Lambda_{k}-F\right)\right).
\]
In view of projection formula we compute
\begin{eqnarray*}
\pi_{\ast}\mathcal{O}_{\ff}\left(\Lambda_{k}-F\right) & = & \pi_{\ast}\left(\mathcal{O}_{\ff}\left(\C_{0}\right)\ot\pi^{\ast}\mathcal{O}_{\Pu}\left(k-1\right)\right)\\
 & = & \pi_{\ast}\mathcal{O}_{\ff}\left(\C_{0}\right)\ot\mathcal{O}_{\Pu}\left(k-1\right)\\
 & = & \left(\mathcal{O}_{\Pu}\op\mathcal{O}_{\Pu}\left(-n\right)\right)\ot\mathcal{O}_{\Pu}\left(k-1\right)\\
 & = & \mathcal{O}_{\Pu}\left(k-1\right)\op\mathcal{O}_{\Pu}\left(k-n-1\right).
\end{eqnarray*}
Since $k>n\ge0$, we conclude that $h^{1}\left(\ff,\mathcal{O}_{\ff}\left(\Lambda_{k}-F\right)\right)=0$.

If $P\neq Q$ are not both in $\C_{0}$ and not both in the same fibre
$F$ then, since fibres are linearly equivalent, a divisor $\C_{0}+kF$
for a suitable $F$ will contain $P$ but not $Q$.

If $P$ and $t$ are not both in $\C_{0}$ and not both in same fibre
$F$ then, since fibres are linearly equivalent, a divisor $\C_{0}+F_{1}+\dots F_{k}$
will contain $P$ but not $t$.

We are left to the case when $P,Q$ or $P,t$ are both in $\C_{0}$.
In view of Lemma \ref{lem:RATIONALSCROLLShasezioniglobali}, we can
pick $\C_{1}+F_{1}+\dots+F_{k-n}$.
\end{proof}

\subsection{A numerical point of view}
\begin{defn}
Let $\S$ be a surface. We say that $\S$ is $\emph{minimal}$ if
$\K_{\S}$ is nef.

We say that $\S'$ is a $\emph{minimal model}$ of $\S$ if $\S'$
is minimal and birationally equivalent to $\S$.
\end{defn}

\begin{rem}
\label{rem:amplenef}We see that a minimal model is a relatively minimal
model: if $\ks$ is nef and $\C\su\S$ is a smooth rational curve,
then Corollary \ref{cor:genusformula} gives
\[
0=2+\C^{2}+\C\cdot\ks
\]
and hence $\C^{2}\le-2$. Then $\S$ does not contain $\left(-1\right)$-curves,
in view of Corollary \ref{cor:characterizationrelminimal} we conclude
that $\S$ is a relatively minimal model. The converse is not true:
in view of Corollary \ref{cor:characterizationrelminimal} and Examples
\ref{ex:inttheoryonplane}, \ref{ex:inttheoryonquadric} we conclude
that $\Pd$ and $\Pu\times\Pu$ are both relatively minimal models,
but they are not minimal models. In fact, since $\K_{\Pd}\li-3L$,
for a curve $\C\su\Pd$ we have
\[
\K_{\Pd}\cdot\C=-3\deg\left(\C\right)<0.
\]
Moreover we have that $\K_{\Pu\times\Pu}\li-2h_{1}-2h_{2}$ and hence
$\K_{\Pu\times\Pu}\cdot h_{1}=-2$. For instance, let us note that,
in view of Theorem \ref{thm:NakaiMoishezon}, $-\K_{\Pd}$ and $-\K_{\Pu\times\Pu}$
are ample.
\end{rem}

We shall see immediately the importance of this definition:
\begin{prop}
\label{prop:proprietaminimali}Let $f:\S\ra\S'$ be a birational map
between smooth surfaces. If $\S'$ is minimal, then $f$ is a morphism.
\end{prop}

\begin{proof}
We can solve the indeterminacy of $f$ by a finite number of blowing
ups and consider only the last one, thus we can suppose $f$ has an
indeterminacy only in $x\in\S$ and fits into a commutative diagram
\[
\xymatrix{ & \tilde{\S}\ar[dr]^{\mu}\ar[dl]_{\sigma}\\
\S\ar@{-->}[rr]_{f} &  & \S'
}
\]
where $\sigma$ is the blowing up of $\S$ in $x$ with exceptional
curve $E\su\tilde{\S}$ and $\mu$ is a finite composition of blowing
ups. Since $f$ is not defined in $x$, there exists a curve $\C\su\S'$
such that $\mu\left(E\right)=\C$. Under our hypothesis we must have
$\K_{\S'}\cdot C\ge0$, moreover there exists a curve $\tilde{E}\su\tilde{\S}$
such that $\K_{\tilde{S}}=\mu^{\ast}\K_{\S'}+\tilde{E}$. We conclude
that
\[
-1=\K_{\tilde{S}}\cdot E=\mu^{\ast}\K_{\S'}\cdot E+\tilde{E}\cdot E.
\]
Since $\tilde{E}\cdot E\ge0$, we have
\[
0>\mu^{\ast}\K_{\S'}\cdot E=\K_{\S'}\cdot\mu_{\ast}E
\]
which is a contradiction in view of the fact that $\mu_{\ast}E=d\C$
for some $d>0$ .
\end{proof}
\begin{cor}
\label{cor:uniquenessminimalmodel}Two minimal models are isomorphic
if and only if they are birationally equivalent.
\end{cor}

We need to assume
\begin{thm}[Stein factorization]
\label{thm:steinfactorization}Let $f:X\to Y$ be a morphism between
projective varieties. There there exists a commutative diagram
\[
\xymatrix{X\ar[rr]^{f}\ar[dr]_{\alpha} &  & Y\\
 & Z\ar[ur]_{\beta}
}
\]
where $Z$ is a normal variety, $\alpha$ is a morphism with connected
fibres and $\beta$ is a finite morphism.
\end{thm}

$\,$
\begin{thm}
\label{thm:existstype}Let $\S$ be a surface. Then there exists a
unique a constant $t\in\left\{ 0,1,2,3\right\} $ depending only on
$\S$ such that

\begin{itemize}
\item $\ks+tH$ is nef for every ample divisor $H$;
\item if $t>0$, then there exists an ample divisor $H_{0}$ such that $\ks+tH_{0}$
is not ample.
\end{itemize}
\end{thm}

\begin{proof}
Let us see first that, for every ample divisor $H$, the divisor $\ks+3H$
is nef: since $H$ is both big and nef, it follows from Theorem \ref{thm:RKW}
that, for $n\ge1$,
\[
\x\left(\os\left(\ks+nH\right)\right)=h^{0}\left(\ks+nH\right).
\]
In view of Lemma \ref{lem:hodgeindexpossibilita} and Proposition
\ref{prop:pernakaidue}, we conclude that $\x\left(\os\left(\ks+nH\right)\right)$
is a non-zero polynomial of degree at most 2, then there exists $i_{0}\in\left\{ 1,2,3\right\} $
such that $h^{0}\left(\ks+i_{0}H\right)>0$.

If there exists an irreducible curve $\C$ such that $\left(\ks+3H\right)\cdot\C<0$,
then also $\left(\ks+i_{0}H\right)\cdot\C<0$. Thus $\C^{2}<0$ and
$\ks\cdot\C<-3$, which is a contradiction from Corollary \ref{cor:genusformula}.
We conclude that $\ks+3H$ is nef.

Let us set
\[
t=\mathrm{min}\left\{ i\,\,|\,\,\ks+iH\,\,\mbox{is nef for every \ensuremath{H}ample}\right\} ,
\]
then it is clear that $t$ exists and it is unique, moreover it follows
from the previous discussion that $t\le3$.

Let us suppose that $\ks+tH$ is ample for every ample divisor $H$,
we consider 2 cases:

\begin{description}
\item [{$t=1$}] Let $H$ be an ample divisor, then we have a sequence
of ample divisors
\[
H_{0}=H,\,H_{1}=\ks+H,\,\dots,\,H_{n+1}=\ks+H_{n}=n\left(\ks+\frac{1}{n}H\right).
\]
If there exists a curve $\C$ such that $\ks\cdot\C<0$, for $n\gg0$
we have $\left(\ks+\frac{1}{n}H\right)\cdot\C<0$, which is a contradiction
.
\item [{$t\ge2$}] Let $H$ be an ample divisor, then we have a sequence
of ample divisors
\[
H_{0}=H,\,H_{1}=\ks+H,\,\dots,\,H_{n+1}=\ks+tH_{n}.
\]
Then
\begin{eqnarray*}
H_{n+1} & = & \left(t^{n}+\dots+t+1\right)\ks+t^{n+1}H_{0}\\
 & = & \alpha\left(\ks+\frac{t^{n+1}}{t^{n+1}-1}\left(t-1\right)H_{0}\right)
\end{eqnarray*}
for $\alpha>0$. As we did before, one shows that $\ks+\left(t-1\right)H_{0}$
is nef, which is again a contradiction in view of the minimality of
$t$ .
\end{description}
\end{proof}
\begin{defn}
In the same setting and notations as in Theorem \ref{thm:existstype},
we call the constant $t=t\left(\S\right)$ the $\emph{type}$ of $\S$,
and the divisor $H_{0}$ is called \emph{adapted polarization} for
$\S$.
\end{defn}

\begin{rem}
\label{rem:minimaltypezero}It is clear that a surface $\S$ is a
minimal model if and only if $t\left(\S\right)=0$.
\end{rem}

Our next goal is to give a classification of surfaces using the notion
of type.
\begin{prop}
\label{prop:classifthereexistsmorphism}Let $\S$ be a surface, $D$
a nef divisor such that $D^{2}=0$ and such that $h^{0}\left(\S,\os\left(mD\right)\right)>0$
for $m\gg0$. Then there exists a smooth curve $B$ and a morphism
$\varphi:\S\to B$ with connected fibres such that, for the general
fibre $F$, one has $F\cdot D=0$.
\end{prop}

\begin{proof}
Under our assumptions, for $m\gg0$ we can write $|mD|=C+|M|$ where
$C$ is the fixed part of the linear system. It follows from Lemma
\ref{lem:tricktochecknefeness} that $M$ is nef. Moreover $0=D^{2}=D\cdot C+D\cdot M$,
then it follows that $D\cdot C=D\cdot M=0$. Also $C\cdot M+M^{2}=0$
gives $C\cdot M=M^{2}=0$, then $\bs|M|=\emptyset$ and thus it defines
a morphism $\varphi:\S\to B$, where $B\su\Pn$ is a curve. In view
of Theorem \ref{thm:steinfactorization} there exists a smooth curve
$\tilde{B}$ together with a morphism $\psi:\S\to\tilde{B}$ with
connected fibres and a finite morphism $\xi:\tilde{B}\to B$ such
that $\varphi=\xi\circ\psi$. It follows that, for the general fibre
$F$ of $\psi$ has $F\cdot M=0$. Since $M\cdot C=0$, $\psi$ contracts
$C$, then $C\cdot F=0$ and hence $D\cdot F=0$.
\end{proof}
\begin{cor}
\label{cor:classifiDnefsquaredzero}Let $\S$ be a surface and $D$
be a divisor linearly equivalent to $\ks+H$ where $H$ is ample.
If $D$ is nef and $D^{2}=0$, then $D\li0$ or there exists a smooth
curve $B$ together with a morphism $\varphi:\S\to B$ with connected
fibres such that the general fibre $F$ is a smooth rational curve
with $F\cdot H=2$.
\end{cor}

\begin{proof}
In view of Theorem \ref{thm:RKW} we have $\x\left(\os\left(D\right)\right)=h^{0}\left(\S,\os\left(D\right)\right)$
and, since $D$ is nef, $D\cdot H\ge0$. We have two possibilities:

\begin{description}
\item [{$D\cdot H=0$}] It follows from Theorem \ref{thm:Hodge} that $D^{2}\le0$
then, in view of Prosposition \ref{prop:nefpositivesquare}, we deduce
that $D^{2}=0$ and hence, in view again of Theorem \ref{thm:Hodge},
$D\sim0$. In view of Corollary \ref{cor:amplenessnumericalproperty},
$-\ks$ is ample and then, from Theorem \ref{thm:RKW},
\[
h^{1}\left(\S,\os\right)=h^{2}\left(\S,\os\right)=0,
\]
in particular $\x\left(\os\right)=1$. In view of Riemann-Roch we
conclude that
\[
h^{0}\left(\S,\os\left(D\right)\right)=\x\left(\os\left(D\right)\right)=\x\left(\os\right)=1,
\]
then $D$ is effective and hence $D\li0$.
\item [{$D\cdot H>0$}] Since $D^{2}=0$, then $D\cdot\ks=-D\cdot H<0$.
In view of Serre duality, we have
\[
h^{2}\left(\S,\os\left(mD\right)\right)=h^{0}\left(\S,\os\left(\ks-mD\right)\right)=0
\]
since, for $m\gg0$, $\left(\ks-mD\right)\cdot H<0$. Then
\[
h^{0}\left(\S,\os\left(mD\right)\right)\ge\x\left(\os\left(mD\right)\right)=\x\left(\os\right)-\frac{m}{2}D\cdot\ks
\]
and we conclude that $h^{0}\left(\S,\os\left(mD\right)\right)>0$
for $m\gg0$. In view of Proposition \ref{prop:classifthereexistsmorphism},
there exists a smooth curve $B$ and a morphism $\varphi:\S\to B$
with connected fibres such that, for the general fibre $F$, one has
$F\cdot D=0$. Since $F\cdot H>0$, then $F\cdot\ks<0$ and, in view
of Example \ref{ex:fibresquaredzero}, $F^{2}=0$. Since we work over
the field $\cc$ of complex numbers and the general fibre is connected,
it is also smooth and hence irreducible. It follows from Corollary
\ref{cor:genusformula} that
\[
2p_{a}\left(F\right)-2=F^{2}+F\cdot\ks<0,
\]
and hence $F\iso\Pu$ with $F\cdot H=-F\cdot\ks=2$.
\end{description}
\end{proof}
\begin{prop}
\label{prop:classificationnumerical}Let $\S$ be a surface and $H$
be an ample divisor.

\begin{enumerate}
\item If $\ks+3H\sim0$, then $\S\iso\Pd$ with $H\in|\mathcal{O}_{\Pd}\left(1\right)|$;
\item if $\ks+2H\sim0$, then $\S\iso\Pu\times\Pu$ with $H\in|\mathcal{O}_{\Pu\times\Pu}\left(1,1\right)|$.
\end{enumerate}
\end{prop}

\begin{proof}
In view Corollary \ref{cor:amplenessnumericalproperty}, we conclude
that in both cases $-\ks$ is ample, therefore, in view of Serre duality
and Theorem \ref{thm:RKW}, $\x\left(\os\right)=1$.

\begin{enumerate}
\item Let us set $P\left(t\right)=\x\left(\os\left(tH\right)\right)$, in
view of Riemann-Roch we have
\[
P\left(t\right)=\frac{1}{2}\left(t^{2}H^{2}-tH\cdot\ks\right)+1.
\]
Then, again in view of Serre duality and Theorem \ref{thm:RKW},
\[
P\left(-1\right)=\x\left(\os\left(-H\right)\right)=\x\left(\os\left(\ks+H\right)\right)=h^{0}\left(\S,\os\left(\ks+H\right)\right)\ge0.
\]
Since $\ks\sim-3H$, then
\[
P\left(t\right)=\frac{tH^{2}}{2}\left(t+3\right)+1.
\]
Since $H$ is ample and $1-H^{2}=P\left(-1\right)\ge0$, we conclude
that $H^{2}=1$ and then $P\left(1\right)=3$. Let $|H|=F+|M|$, where
$F,M\ge0$ and $F$ is the fixed part of $|H|$. We have $1=H^{2}=H\cdot F+H\cdot M$,
from which we deduce that $F=0$. Therefore $|H|$ defines a birational
map $\varphi_{|H|}:\S\ra\Pd$ which is indeed everywhere defined in
view of Example \ref{ex:inttheoryonplane}, and then it is an isomorphism.
\item Since $2H$ is ample, in view of Theorem \ref{thm:RKW} we have
\[
h^{0}\left(\S,\os\left(\ks+2H\right)\right)=\x\left(\os\left(\ks+2H\right)\right)=\x\left(\os\right)=1,
\]
then we can apply Proposition \ref{prop:classifthereexistsmorphism}
to conclude that there exists a surjective morphism $\varphi:\S\to B$
onto a smooth curve with connected fibres. Let $C\su\S$ be an irreducible
curve with $C^{2}<0$, then, from Corollary \ref{cor:genusformula}
\[
2p_{a}\left(C\right)-2=C^{2}+C\cdot\ks\le-1-2C\cdot H\le-3
\]
which is not possible. We conclude that $\S$ has no curve with negative
self-intersections and thus, in view of Corollary \ref{cor:characterizationrelminimal},
$\S$ is a relatively minimal model. Moreover, in view of Lemma \ref{lem:lemmaIIInove},
we conclude that every fibre $F$ of $\varphi$ is irreducible and
then, again in view of Corollary \ref{cor:genusformula}
\[
2p_{a}\left(F\right)-2=F\cdot\ks=-2F\cdot\ks\le-2.
\]
Therefore $p_{a}\left(F\right)=0$ and, from Remark \ref{rem:geometricarithmeticgenus},
we conclude that $F\iso\Pu$ for every fibre $F$ and then, in view
of Corollary \ref{cor:surfaceisoprojectivebundle}, $\varphi:\S\to B$
is a $\Pu$-bundle over $B$. Since $q\left(\S\right)=0$, in view
of Corollary \ref{cor:invariantsruled} we conclude that $B\iso\Pu$,
but then $\S$ is a relatively minimal model with $C^{2}\ge0$ for
every curve $C\su\S$, so, in view of Remarks \ref{rem:altreproprietaFn}
and \ref{rem:FngeomrigatesuPuno}, we have that $\S\iso\mathbb{F}_{0}\iso\Pu\times\Pu$.
\end{enumerate}
\end{proof}
\begin{prop}
\label{thm:classificazionepreliminare}Let $\S$ be a surface of type
$t$ with adapted polarization $H_{0}$. Then

\begin{itemize}
\item $t=3$ if and only if $\S\iso\Pd$ with $H_{0}\in|\mathcal{O}_{\Pd}\left(1\right)|$;
\item if $t=2$, then $\S$ is a $\Pu$-bundle over a smooth curve and,
for each fibre $F$ we have $H_{0|F}\in\mathcal{O}_{\Pu}\left(1\right)$;
\item if $t=1$, then $\S$ is a $\Pu$-bundle over a smooth curve and,
for each fibre $F$ we have $H_{0|F}\in\mathcal{O}_{\Pu}\left(2\right)$ or
there exists an exceptional curve $E$ with $E\cdot H_{0}=1$;
\item $t=0$ if and only if $\S$ is a minimal model.
\end{itemize}
\end{prop}

\begin{proof}
We see that the statement for $t=0$ is simply Remark \ref{rem:minimaltypezero},
so we only need to consider the remaining cases.

\begin{description}
\item [{$t=3$}] Since $\ks+3H_{0}$ is nef, in view of Proposition \ref{prop:nefpositivesquare}
we have $\left(\ks+3H_{0}\right)^{2}\ge0$. If $\left(\ks+3H_{0}\right)^{2}>0$,
since it is not ample, in view of Theorem \ref{thm:NakaiMoishezon}
we conclude that there exists a curve $\C\su\S$ such that $\left(\ks+3H_{0}\right)\cdot\C=0$.
Therefore $\C^{2}<0$ with $\ks\cdot\C\le-3$, which is a contradiction
with Corollary \ref{cor:genusformula}, then we conclude that
\[
\left(\ks+3H_{0}\right)^{2}=0.
\]
In view of Corollary \ref{cor:classifiDnefsquaredzero}, either $\ks+3H_{0}\li0$
or there exists a smooth curve $B$ together with a surjective morphism
$\varphi:\S\to B$ with smooth rational fibres $F$ such that $3F\cdot H_{0}=2$.
Since the latter equality is not possibile, we get that $\ks+3H_{0}\li0$
and we conclude with Proposition \ref{prop:classificationnumerical}.
Viceversa, if $L\su\Pd$ is a line, then $L$ determines a very ample
divisor on $\Pd$ and $\mathcal{K}_{\Pd}\li-3L$. Then, for $0\le n\le2$,
we have
\[
\mathcal{K}_{\Pd}+nL\li\left(n-3\right)L
\]
which is not nef. We conclude that $t\left(\Pd\right)=3$.
\item [{$t=2$}] As before, in view of Corollary \ref{cor:genusformula},
we have $\left(\ks+2H_{0}\right)^{2}=0$. In view of Corollary \ref{cor:classifiDnefsquaredzero},
either $\ks+2H_{0}\li0$ or there exists a smooth curve $B$ together
with a surjective morphism $\varphi:\S\to B$ with smooth rational
fibres $F$ such that $\left(\ks+2H_{0}\right)\cdot F=0$ and $F\cdot H_{0}=1$.
In the first case we conclude in view of Proposition \ref{prop:classificationnumerical}
that $\S\iso\Pu\times\Pu$, in the second we conclude in view of Corollary
\ref{cor:surfaceisoprojectivebundle}.
\item [{$t=1$}] Let $\ks+H_{0}$ be nef but not ample. If $\left(\ks+H_{0}\right)^{2}>0$,
in view of Theorem \ref{thm:NakaiMoishezon} there exists an irreducible
curve $E$ such that $\left(\ks+H_{0}\right)\cdot E=0$. Since $\left(\ks+H_{0}\right)\cdot H_{0}>0$,
in view of Lemma \ref{lem:hodgeindexpossibilita} we conclude that
$E^{2}<0$. Moreover $\ks\cdot E=-H_{0}\cdot E<0$ and hence, in view
of Corollary \ref{cor:genusformula} we conclude that
\[
E^{2}=\ks\cdot E=-1\,\,\,\,\mbox{and}\,\,\,\,p_{a}\left(E\right)=0.
\]
In view of Remark \ref{rem:geometricarithmeticgenus}, $E$ is a $\left(-1\right)$-curve.
Let us consider the case when $\left(\ks+H_{0}\right)^{2}=0$: in
view of Corollary \ref{cor:classifiDnefsquaredzero}, either $\ks+H_{0}\li0$
or there exists a surjective morphism $\varphi:\S\to B$ where $B$
is a curve such that all fibres are connected and the general fibre
$F$ is smooth rational with $F\cdot H_{0}=2$. In the latter case,
let $\overline{F}$ be a degenerate fibre. If $\overline{F}=2F$,
then $F\cdot H_{0}=1$ and $F\cdot\ks=-1$, but we get a contradiction
with Corollary \ref{cor:genusformula}, since $F^{2}+F\cdot\ks$ must
be even. Therefore $\overline{F}=F_{1}+F_{2}$ and we see that
\[
0=\overline{F}\cdot F_{1}=F_{1}^{2}+F_{1}\cdot F_{2},
\]
since $\overline{F}$ is connected, we conclude that $F_{1}^{2}=F_{2}^{2}<0$.
In view of Corollary \ref{cor:fibersnumerically} we also have $\ks\cdot F_{1}+\ks\cdot F_{2}=-2$,
then
\[
\ks\cdot F_{i}=F_{i}^{2}=-1\,\,\,\,\mbox{for \ensuremath{i=1,2}}.
\]
Thus the degenerate fibres are union of $\left(-1\right)$-curves.
We are left to the case when $\ks+H_{0}\li0$. Let $H$ be an ample
divisor and let us suppose there exists no $n\in\N$ such that $n\ks+H\li0$.
As in the proof of Theorem \ref{thm:existstype}, we have a sequence
of nef divisors $\left(\ks+\frac{1}{n}H\right)_{n>0}$, but, since
$-\ks$ is ample, for every curve $C$ we have $\ks\cdot C=-\left(-\ks\right)\cdot C<0$,
and hence, for $n\gg0$
\[
\left(\ks+\frac{1}{n}H\right)\cdot C=\ks\cdot C+\frac{1}{n}\left(H\cdot C\right)<0,
\]
which is a contradiction. This shows that there exists $n>0$ such
that $n\ks+H\li0$ and, in view of Lemma \ref{lem:divisordifferencesmooth},
every divisor can be written as a difference of two ample divisors,
in particular we have shown that
\[
\pic\left(\S\right)\iso\Z\cdot\left[-\ks\right].
\]
Under our assumptions, $-\ks$ is ample, therefore, in view of Serre
duality and Theorem \ref{thm:RKW}
\[
h^{1}\left(\S,\os\right)=h^{2}\left(\S,\os\right)=0,
\]
in particular $\x\left(\os\right)=1$. Let us consider the exact sequence
\[
0\to\Z\to\os^{an}\to\left(\os^{*}\right)^{an}\to0.
\]
Looking at the cohomology sequence we deduce that $H^{1}\left(\S,\Z\right)=0$
and that $H^{2}\left(\S,\Z\right)\iso\pic\left(\S\right)$, therefore
we have $b_{1}\left(\S\right)=0$ and $b_{2}\left(\S\right)=1$. In
view of Noether's Formula, we have $12=\ks^{2}+\xt\left(\S\right)=\ks^{2}+3$,
thus $\ks^{2}=9$ which is not possibile since the intersection form
on $H^{2}\left(\S,\Z\right)$ is unimodular .
\end{description}
\end{proof}
\begin{thm}
\label{thm:classificazione}Let $\S$ be a surface of type $t\left(\S\right)$.
Then

\begin{itemize}
\item $t=3$ if and only if $\S\iso\Pd$;
\item $t=2$ if and only if $\S$ is a $\Pu$-bundle over a smooth curve;
\item $t=1$ if and only if $\S$ contains a $\left(-1\right)$-curve and
$\S\not\iso\bl_{p}\left(\Pd\right)$;
\item $t=0$ if and only if $\S$ is a minimal model.
\end{itemize}
\end{thm}

\begin{proof}
In view of Proposition \ref{thm:classificazionepreliminare} we need
only to consider the cases when $t\left(\S\right)=2$ and $t\left(\S\right)=1$.
If $\pi:\S\to B$ is a $\Pu$-bundle over a smooth curve $B$, in
view of Proposition \ref{prop:existssection}, there exists a section
$H$ of $\pi$. Let $F$ be a fibre of $\pi$ and let $n\in\N$ such
that $H^{2}+n>0$. Then
\begin{eqnarray*}
\left(H+nF\right)^{2} & = & H^{2}+2n>0\\
\left(H+nF\right)\cdot H & = & H^{2}+n>0\\
\left(H+nF\right)\cdot F & = & 1.
\end{eqnarray*}
It follows from Theorem \ref{thm:NakaiMoishezon} and Proposition
\ref{prop:picgeometricallyruled} that $\tilde{H}=H+nF$ is ample.
Since $F\iso\Pu$, then $\ks\cdot F=-2$ and thus
\[
\left(\ks+2\tilde{H}\right)\cdot F=0,
\]
showing that $\ks+2\tilde{H}$ is not ample and thus $t\left(\S\right)\ge2$.
In view of Example \ref{ex:inttheoryonplane}, we conclude that $\S\not\iso\Pd$,
and hence, in view of Theorem \ref{thm:classificazionepreliminare}
we conclude that $t\left(\S\right)=2$.

Let us suppose now that $\pi:\S\to B$ is a $\Pu$-bundle over a smooth
curve $B$ and let $E\su\S$ be a $\left(-1\right)$-curve. Clearly
$E$ is not contained in any fibre, than $E$ dominates $B$ via $\pi$
and then $B\iso\Pu$. In view of Proposition \ref{prop:proprietaFn}
we conclude that $\S\iso\mathbb{F}_{1}\iso\bl_{p}\left(\Pd\right)$
and hence we are in case $t\left(\S\right)=2$.
\end{proof}

\subsection{Consequences}
\begin{cor}
\label{cor:ruledess}Let $\S$ be a relatively minimal surface with
$\ks^{2}<0$. Then $t\left(\S\right)\ge2$, in particular $\S$ is
ruled.
\end{cor}

\begin{proof}
Since $\S$ is relatively minimal, in view of Theorem \ref{thm:classificazionepreliminare},
we have $t\left(\S\right)\neq1$. Since $\ks^{2}<0$, it follows from
Proposition \ref{prop:nefpositivesquare} that $\ks$ is not nef.
Then $t\left(\S\right)\ge2$.
\end{proof}
\begin{prop}
\label{prop:minimalmodelnonruled}Let $\S$ be a non ruled surface.
Then $\S$ dominates a unique minimal model.
\end{prop}

\begin{proof}
Since $\S$ is non ruled, $t\left(\S\right)\le1$. If $t\left(\S\right)=0$,
then we are done. If $t\left(\S\right)=1$, in view of \ref{prop:everysurfacesrelativelyminimalmodel},
there exists a relatively minimal model $\S_{0}$ and a birational
morphism $\sigma:\S\to\S_{0}$. Moreover $t\left(\S_{0}\right)=0$
and we conclude in view of Theorem \ref{thm:classificazione} and
Corollary \ref{cor:uniquenessminimalmodel}.
\end{proof}
\begin{defn}
We say that a surface $\S$ is \emph{uniruled} if there exists a rational
dominant map $\S'\ra\S$ with $\S'$ ruled, while we say that $\S$
is \emph{covered by rational curves }if, given $s\in\S$ a general
point, there exists a rational curve on $\S$ passing through $s$.
\end{defn}

\begin{lem}
\label{lem:pefbastavederenef}Let $\S$ be a surface, $P$ a divisor.
Then $P$ is pef if and only if $P\cdot N\ge0$ for every $N$ nef.
\end{lem}

\begin{proof}
Since ample divisors are nef, if $P\cdot N\ge0$ for every nef divisor,
then $P$ is pef. To prove the converse, let us suppose that $P$
is pef and $P\cdot N<0$ for a nef divisor $N$. Let $H$ be an ample
divisor, then, in view of Remark \ref{rem:sommanefampio}, $N+\frac{1}{n}H$
is ample for every $n>0$ and hence, for $n\gg0$
\[
0\le P\cdot\left(N+\frac{1}{n}H\right)=P\cdot N+\frac{1}{n}P\cdot H<0,
\]
which is a contradiction .
\end{proof}
\begin{prop}
\label{prop:ruledKSnonpef}Let $\S$ be a ruled surface, then $\ks$
is not pef.
\end{prop}

\begin{proof}
Let $\C$ be a curve such that $\S$ is birationally equivalent to
$\C\times\Pu$, then we have a commutative diagram\[
\xymatrix{
\tilde{\S}\ar[dr]^{\varphi}\ar[d]_{\mu}\ar@/_2pc/[dd]_{\psi}\\
\C\times\Pu\ar@{-->}[r]\ar[d]^{\pi_{1}} & \S\\
\C
}
\]where $\mu$ is given by a finite number of blow ups. In view of Riemann-Hurwitz
formula, $\mathcal{K}_{\tilde{\S}}\li\varphi^{\ast}\ks+R_{\varphi}$
with $R_{\varphi}\ge0$. Let us suppose $\ks$ is pef, then $\mathcal{K}_{\tilde{\S}}$
is pef, moreover, since the general fibre $F$ of $\psi$ is smooth
and rational, in view of Corollary \ref{cor:genusformula} we have
$F\cdot\mathcal{K}_{\tilde{\S}}<0$, which is not possibile in view
of Lemma \ref{lem:pefbastavederenef} .
\end{proof}
\begin{prop}
\label{prop:relativelyminimalruled}The relatively minimal ruled surfaces
are $\Pd$ or $\Pu$-bundles $\pc\left(\mathcal{E}\right)\to\C$.
\end{prop}

\begin{proof}
Let $\S$ be a ruled surface, in view of Proposition \ref{prop:ruledKSnonpef},
we conclude that $\S$ is not minimal. In view of Proposition \ref{prop:everysurfacesrelativelyminimalmodel},
there exists a birational morphism $\sigma:\S\to\S_{0}$ with $\S_{0}$
relatively minimal. Since $\S_{0}$ is ruled, in view of Proposition
\ref{prop:ruledKSnonpef}, we see that $\S_{0}$ is not minimal. We
conclude in view of Corollary \ref{cor:characterizationrelminimal}
and Theorem \ref{thm:classificazione}.
\end{proof}
\begin{thm}
\label{thm:equivalence1}Let $\S$ be a surface, then the following
are equivalent:

\begin{enumerate}
\item there exists a curve $\C\su\S$ with $\C\iso\Pu$ and $\C^{2}\ge0$;
\item $\ks$ is not pef;
\item $\S$ is ruled;
\item $\S$ is uniruled;
\item $\S$ is covered by rational curves.
\end{enumerate}
\end{thm}

\begin{proof}
$\,$

\begin{description}
\item [{$1)\Longrightarrow2)$}] Let $\C\su\S$ be a smooth rational curve
with $\C^{2}\ge0$. Then $\C$ defines a nef divisor and, in view
of Corollary \ref{cor:genusformula}, we have $\ks\cdot\C<0$, therefore
we conclude in view of Lemma \ref{lem:pefbastavederenef}.
\item [{$2)\Longrightarrow3)$}] In view of Proposition \ref{prop:everysurfacesrelativelyminimalmodel},
let $\sigma:\S\to\S_{0}$ be a birational morphism with $\S_{0}$
relatively minimal. In view of Riemann-Hurwitz formula $\mathcal{K}_{\S_{0}}$
is not pef and therefore $t\left(\S_{0}\right)\ge1$. In view of Corollary
\ref{cor:characterizationrelminimal} and Theorem \ref{thm:classificazione}
we conclude that $t\left(\S_{0}\right)\ge2$ and thus $\S$ is ruled.
\item [{$3)\Longrightarrow1)$}] In view of Proposition \ref{prop:relativelyminimalruled},
let $\sigma:\S\to\S_{0}$ be a birational morphism with $\S_{0}=\Pd$
or $\S_{0}=\mathbb{P}_{B}\left(\mathcal{E}\right)$ is a $\Pu$-bundle over a smooth
curve $B$. If $\S_{0}=\Pd$, then the general line $L$ does not
contain the base points of $\sigma^{-1}$, and therefore we take $\C=\sigma^{-1}\left(L\right)\su\S$.
If $\pi:\S_{0}\to B$ is a $\Pu$-bundle, then we can take $\C$ as
the general fibre of $\pi\circ\sigma$.
\item [{$3)\Longrightarrow4)$}] It is clear.
\item [{$4)\Longrightarrow5)$}] It is clear.
\item [{$5)\Longrightarrow2)$}] By definition of covered by rational curves,
there exists a 1-dimensional algebraic family of curves on $\S$ such
that the general member is an irreducible rational curve, then we
see that there exist a ruled surface $\tilde{\S}$ and a surjective
morphism $\varphi:\tilde{\S}\to\S$. In particular $\mathcal{K}_{\tilde{\S}}$
is not pef, thus, if $\ks$ were pef, then we would get a contradiction
in view of Riemann-Hurwitz formula .
\end{description}
\end{proof}
\begin{defn}
A surface $\S$ is said to be \emph{rationally connected }if, for
every $s_{1},s_{2}\in\S$ general points, there exists a rational
curve $\C\su\S$ such that $s_{1},s_{2}\in\C$.
\end{defn}

\begin{thm}
\label{thm:equivalence2}Let $\S$ be a surface, then the following
are equivalent:

\begin{enumerate}
\item there exists a curve $\C\su\S$ with $\C\iso\Pu$ and $\C^{2}>0$;
\item $q\left(\S\right)=P_{2}\left(\S\right)=0$;
\item $\S$ is rational;
\item $\S$ is unirational;
\item $\S$ is rationally connected.
\end{enumerate}
\end{thm}

\begin{proof}
$\,$

\begin{description}
\item [{$1)\Longrightarrow2)$}] Let $\C\su\S$ be a smooth rational curve
with $\C^{2}>0$. Then $\C$ defines a nef divisor and, in view of
Corollary \ref{cor:genusformula}, we have $\ks\cdot\C<0$, therefore,
in view of Lemma \ref{lem:pefbastavederenef}, we conclude that $\ks$
is not pef. It follows from the proof of Lemma \ref{lem:hodgeindexpossibilita}
that we have $P_{2}\left(\S\right)=0$. Let us consider now the exact
sequence
\[
0\to\os\left(-\C\right)\to\os\to\mathcal{O}_{\C}\to0.
\]
 In view of Lemma \ref{lem:sufficienttobebig} we see that $\C$ defines
a big and nef divisor, in view of Theorem \ref{thm:RKW} we have
\[
h^{1}\left(\S,\os\left(-\C\right)\right)=h^{2}\left(\S,\os\left(-\C\right)\right)=0
\]
 and then
\[
h^{1}\left(\S,\os\right)=h^{1}\left(\C,\oc\right)=h^{1}\left(\Pu,\mathcal{O}_{\Pu}\right)=0.
\]
\item [{$2)\Longrightarrow3)$}] Since $q\left(\S\right)$ and $P_{2}\left(\S\right)$
are birational invariants, we can suppose, in view of Proposition
\ref{prop:everysurfacesrelativelyminimalmodel}, that $\S$ is relatively
minimal, thus we assume that $t\left(\S\right)\neq1$. If $t\left(\S\right)=3$,
in view of Theorem \ref{thm:classificazione}, we are done. If $t\left(\S\right)=2$,
in view of Theorem \ref{thm:classificazione} and Remark \ref{rem:FngeomrigatesuPuno},
we conclude that $\S$ is rational. Thus, let us suppose $t\left(\S\right)=0$.
Since $P_{2}\left(\S\right)=0$, then $p_{g}\left(\S\right)=0$ by
Remark \ref{rem:potenze} and hence
\[
\x\left(\os\right)=1-q\left(\S\right)+p_{g}\left(\S\right)=1.
\]
Since we are in case $t\left(\S\right)=0$, in view of Theorem \ref{thm:classificazione}
we have that $\ks$ is nef and thus, in view of Riemann-Roch, Theorem
\ref{thm:RKW} and Proposition \ref{prop:nefpositivesquare}, we conclude
that
\[
h^{0}\left(\S,\os\left(-\ks\right)\right)+h^{2}\left(\S,\os\left(-\ks\right)\right)\ge\x\left(\os\left(-\ks\right)\right)=\ks^{2}+1>0.
\]
In view of Serre duality,
\[
h^{2}\left(\S,\os\left(-\ks\right)\right)=h^{0}\left(\S,\os\left(2\ks\right)\right)=0
\]
and then $h^{0}\left(\S,\os\left(-\ks\right)\right)>0$. We conclude
since then, for every ample divisor $H$, we have
\[
\ks\cdot H\ge0\,\,\,\,\mbox{and}\,\,\,\,\ks\cdot H\le0.
\]
Therefore $\ks\cdot H=0$ and then, in view of Theorem \ref{thm:Hodge}
and Proposition \ref{prop:nefpositivesquare} we conclude that $-\ks\sim0$,
but we have seen that $-\ks$ is effective, thus $\ks\li0$ and $p_{g}\left(\S\right)=1$
which is a contradiction .
\item [{$3)\Longrightarrow1)$}] In view of Proposition \ref{prop:relativelyminimalruled}
and Remarks \ref{rem:FngeomrigatesuPuno} and \ref{rem:altreproprietaFn},
there exists a birational morphism $\sigma:\S\to\S_{0}$ with $\S_{0}=\Pd$
or $\S_{0}=\ff$ with $n\neq1$. If $\S_{0}=\Pd$, then the general
line $L$ does not contain the base points of $\sigma^{-1}$, and
therefore we take $\C=\sigma^{-1}\left(L\right)\su\S$. If $\S_{0}=\ff$,
in view of Proposistion \ref{prop:proprietaFn}, let $C_{0}\su\ff$
be the curve such that $C_{0}^{2}=-n$. In view of Theorem \ref{thm:RATIONALSCROLLSveryample},
for $m\gg0$ the divisor $H=C_{0}+mF$ is very ample. Moreover, in
view of \cite[Theorem 6.6.1]{Ke} and Corollary \ref{cor:BertiniII},
the general element $D\in|H|$ is a smooth irreducible curve. Since
$H\cdot F=1$, then $D$ dominates $\Pu$ and thus $H\iso\Pu$. We
can choose $D$ to avoid the base points of $\sigma^{-1}$, and therefore
we take $\C=\sigma^{-1}\left(D\right)\su\S$.
\item [{$3)\Longrightarrow4)$}] It is clear.
\item [{$4)\Longrightarrow5)$}] It is clear.
\item [{$5)\Longrightarrow2)$}] By definition of rationally connected,
there exists a 1-dimensional algebraic family of curves on $\S$ such
that the general member is an irreducible rational curve and, moreover,
the for every pair of general points, there exists a member of this
family containing them. Then we see that there exist a ruled surface
$\pi:\tilde{\S}\to\C$ and a surjective morphism $\varphi:\tilde{\S}\to\S$.
Since we can pull-back a general curve on $\S$, it follows that $\C\iso\Pu$
and thus, in view of Corollary \ref{cor:invariantsruled},
\[
q\left(\S\right)=q\left(\tilde{\S}\right)=0.
\]
In view of Riemann-Hurwitz, we have
\[
2\mathcal{K}_{\tilde{\S}}=\varphi^{\ast}\left(2\ks\right)+2R_{\varphi}.
\]
Since $2R_{\varphi}\ge0$, if $\varphi^{\ast}\left(2\ks\right)\ge0$,
then $2\mathcal{K}_{\tilde{\S}}\ge0$, which is not possibile since
$\mathcal{K}_{\tilde{\S}}\cdot F<0$ for every fibre $F$ of $\pi$
in view of Corollary \ref{cor:genusformula}. Therefore $|2\ks|=\emptyset$.
\end{description}
\end{proof}
In view of Theorem \ref{thm:equivalence2} we deduce immediately
\begin{thm}[Castelnuovo's criterion for rationality]
Let $\S$ be a surface, then $\S$ is rational if and only if $q\left(\S\right)=P_{2}\left(\S\right)=0$.
\end{thm}

$\,$
\begin{thm}[Lüroth's problem for complex surfaces]
Let $\S$ be a surface. Then $\S$ is rational if and only if it
is unirational.
\end{thm}

\begin{cor}
\label{cor:rationalnefandbig}Let $\S$ be a surface such that $-\ks$
is nef and big. Then $\S$ is rational.
\end{cor}

\begin{proof}
In view of Theorem \ref{thm:RKW} and Serre duality, we have $q\left(\S\right)=0$
and, as in the proof of Lemma \ref{lem:hodgeindexpossibilita}, since
$-\ks$ is big, we conclude that $P_{2}\left(\S\right)=0$. The results
follows from Theorem \ref{thm:equivalence2}.
\end{proof}
\begin{prop}
Let $\S$ be a surface with $p_{g}\left(\S\right)=0$ and $q\left(\S\right)\ge2$.
Then $\S$ is ruled.
\end{prop}

\begin{proof}
Since both $p_{g}\left(\S\right)$ and $q\left(\S\right)$ are birational
invariants, we can suppose, in view of Proposition \ref{prop:everysurfacesrelativelyminimalmodel},
$\S$ to be relatively minimal. In view of Noether's formula \ref{fact:hodgeshows},
we have
\begin{eqnarray*}
\ks^{2} & = & 12\x\left(\os\right)-\xt\left(\S\right)\\
 & = & 12\left(1-q\left(\S\right)\right)-2+4q\left(\S\right)-b_{2}\left(\S\right)<0
\end{eqnarray*}
and we conclude in view of Corollary \ref{cor:ruledess}.
\end{proof}
\begin{cor}
Let $\S$ be a surface with $\x\left(\os\right)<0$. Then $\S$ is
ruled.
\end{cor}

\begin{proof}
Since $\x\left(\os\right)$ is a birational invariant, then, in view
of Theorem \ref{thm:classificazione}, we can suppose $t\left(\S\right)\neq1$.
Moreover, in view of \ref{fact:hodgeshows} we have
\[
\ks^{2}+\xt\left(\S\right)<0.
\]
If $\ks^{2}<0$, then we conclude in view of Corollary \ref{cor:ruledess},
if $\xt\left(\S\right)<0$, we conclude in view of \cite[Theorem X.4]{Beau}.
\end{proof}

\section{Embedded surfaces}

\subsection{Del Pezzo surfaces}
\begin{defn}
Let $\S$ be a surface, we say that $\S$ is a \emph{del Pezzo surface}
(also called \emph{Fano surface}) if $-\ks$ is ample. If this is
the case, we denote $d\left(\S\right)=\ks^{2}$, called the \emph{degree
of $\S$}.
\end{defn}

\begin{rem}
\label{rem:DelPezzoarerational}It follows from Corollary \ref{cor:rationalnefandbig}
that a del Pezzo surface is rational.
\end{rem}

\begin{example}
\label{ex:DelPezzoexamples}In view of Remark \ref{rem:amplenef},
we see that $\Pd$ and $\Pu\times\Pu$ are examples of del Pezzo surfaces.
\end{example}

\begin{lem}
\label{lem:DelPezzogradomassimo}Let $\S$ be a del Pezzo surface,
then $d\left(\S\right)\le9$.
\end{lem}

\begin{proof}
In view of Proposition \ref{prop:everysurfacesrelativelyminimalmodel},
let $\sigma:\S\to\S_{0}$ be a composition of a finite number of blow
ups with $\S_{0}$ relatively minimal. In view of Remark \ref{rem:DelPezzoarerational}
we see that $\S$ is rational and then, in view of Proposition \ref{prop:relativelyminimalruled}
and Remarks \ref{rem:FngeomrigatesuPuno}, either $\S_{0}=\Pd$ or
$\S_{0}=\ff$ for $n\neq1$. In the first case we have
\[
d\left(\S\right)=\ks^{2}\le\K_{\Pd}^{2}=9,
\]
in the second
\[
d\left(\S\right)\le\K_{\ff}^{2}=\left(-2\C_{0}-\left(2+n\right)F\right)^{2}=8.
\]
\end{proof}
\begin{prop}
\label{prop:DelPezzosolocurveeccezionali}Let $\S$ be a del Pezzo
surface. Then every irreducible curve with negative self-intersection
is a $\left(-1\right)$-curve.
\end{prop}

\begin{proof}
Let $\C\su\S$ be an irreducible curve with $\C^{2}<0$. Since $-\ks$
is ample, then $\C\cdot\ks<0$ and hence, in view of Corollary \ref{cor:genusformula},
the only possibility is
\[
\C^{2}=\ks\cdot\C=-1
\]
and $p_{a}\left(\C\right)=0$. In view of Remark \ref{rem:geometricarithmeticgenus}
we conclude that $\C$ is a $\left(-1\right)$-curve.
\end{proof}
\begin{defn}
We say that a finite collection of points on $\Pd$ is \emph{in general
position} if no 3 points lie on a line and no 6 points lie on a conic.
\end{defn}

\begin{thm}
\label{thm:DelPezzoclassification}Let $\S$ be a del Pezzo surface
of degree $d$, then

\begin{itemize}
\item if $d=9$, then $\S\iso\Pd$;
\item if $d=8$, then $\S\iso\Pu\times\Pu$ or $\S\iso\bl_{p}\left(\Pd\right)$;
\item if $1\le d\le7$, then $\S$ is isomorphic to the blow-up of $\Pd$
in $9-d$ general points.
\end{itemize}
\end{thm}

Viceversa $\mathbb{P}^{2}$, $\Pu\times\Pu$ and every blow-up of $\Pd$ in
$r\le8$ general points are del Pezzo surfaces.
\begin{proof}
As pointed out in Example \ref{ex:DelPezzoexamples}, both $\Pd$
and $\Pu\times\Pu$ are del Pezzo surfaces. Let $1\le r\le8$ be an
integer and let
\[
\mu:\S_{r}\to\Pd
\]
be the blow up of $\Pd$ in $\left\{ p_{1},\dots,p_{r}\right\} $
general points. Then
\[
\K_{\S_{r}}^{2}=\K_{\Pd}^{2}-r=9-r>0.
\]
Let $L\su\Pd$ be a line and
\[
\C=a\mu^{\ast}L+a_{1}E_{1}+\dots+a_{r}E_{r}
\]
be the class of a curve on $\S_{r}$, where $E_{i}$ denotes the exceptional
curve of the $i$-th blow up. Therefore, since points are in general
position, we have
\begin{eqnarray*}
\K_{\S_{r}}\cdot\C & = & \left(-3\mu^{\ast}L+E_{1}+\dots+E_{r}\right)\cdot\left(a\mu^{\ast}L+a_{1}E_{1}+\dots+a_{r}E_{r}\right)\\
 & = & -3a-a_{1}-\dots-a_{r}.
\end{eqnarray*}
In view of Theorem \ref{thm:NakaiMoishezon}, we need to show that
\[
-3a-a_{1}-\dots-a_{r}<0.
\]
Let us assume that $-3a-a_{1}-\dots-a_{r}\ge0$ and let $\C_{0}=\mu_{\ast}\C\su\Pd$
with $a=\deg\left(\C_{0}\right)$. We have
\begin{eqnarray*}
-3a & = & \K_{\Pd}\cdot\mu_{\ast}\C=\mu^{\ast}\K_{\Pd}\cdot\C\\
 & = & \K_{\S_{r}}\cdot\C-E_{1}\cdot\C-\dots-E_{r}\cdot\C,
\end{eqnarray*}
then
\[
\sum_{i=1}^{r}\mu\left(p_{i},\C_{0}\right)=\C\cdot E_{1}+\dots+\C\cdot E_{r}\ge3a.
\]
If $a\le2$, then $\C_{0}$ should be either a line through 3 of the
blown-up points or a conic through 5 of the blown-up points .

Let us suppose $a\ge3$, let
\[
\left\{ p_{r+1},\dots,p_{9}\right\} \su\C_{0}-\left\{ p_{1},\dots,p_{r}\right\}
\]
be distinct points and consider the cubic $\C_{1}$ passing through
$p_{1}\dots,p_{9}$. Then, in view of Example \ref{ex:inttheoryonplane}
and Lemma \ref{lem:disuguaglianzamolteplicita} we have
\[
\C_{0}\cdot\C_{1}=3a\ge\sum_{i=1}^{9}\mu\left(p_{i},\C_{0}\right)\mu\left(p_{i},\C_{1}\right),
\]
therefore, under our assumptions, we get
\[
\sum_{i=1}^{9}\mu\left(p_{i},\C_{0}\right)>\sum_{i=1}^{r}\mu\left(p_{i},\C_{0}\right)\ge3a\ge\sum_{i=1}^{9}\mu\left(p_{i},\C_{0}\right)\mu\left(p_{i},\C_{1}\right)\ge\sum_{i=1}^{9}\mu\left(p_{i},\C_{0}\right),
\]
which is not possible .

Let now $\S$ be a del Pezzo surface, in view of Proposition \ref{prop:everysurfacesrelativelyminimalmodel},
there exists a relatively minimal surface $\S_{0}$ and a finite composition
of blow ups
\[
\sigma:\S=\S_{r}\to\dots\to\S_{1}\to\S_{0},
\]
moreover, in view of Remark \ref{rem:DelPezzoarerational}, Proposition
\ref{prop:relativelyminimalruled} and Remark \ref{rem:FngeomrigatesuPuno}
we can assume $\S_{0}=\Pd$ or $\S_{0}=\ff$ for $n\neq1$.

If $\S_{0}=\Pd$, then, in view of Remark \ref{rem:DelPezzoarerational},
we see that no blown-up point may lie on the exceptional divisor of
a previous blow-up, then $\S$ is the blow-up of $\Pd$ in $r$ distinct
points, so that $d=\ks^{2}=9-r$. In view of Proposition \ref{lem:DelPezzogradomassimo}
we conclude that $r\le8$. If 3 points lie on a line $L$ and if $\tilde{L}$
denotes the strict transform of $L$ via $\sigma$, then $\tilde{L}^{2}=-2$,
which is not possibile in view of Remark \ref{rem:DelPezzoarerational}
. In the same way one sees that no 6 points
lie on a conic.

If $\S=\S_{0}=\Pu\times\Pu$, then $\S$ is a del Pezzo surface with
$d=8$. If $r\ge1$, then $\S_{1}$ contains two non-intersecting
$\left(-1\right)$-curves $C_{1}$ and $C_{2}$ and, contracting them,
we get a birational morphism $\S_{1}\to\Pd$ and we are in the previous
case.

If $\S_{0}=\ff$ with $n\ge2$, in view of \ref{prop:proprietaFn},
let $\C_{0}\su\ff$ be a curve such that $\C_{0}^{2}=-n$. If $\tilde{\C_{0}}$
denotes the strict transform of $\C_{0}$ via $\sigma$, then $\tilde{\C_{0}}^{2}<-1$
and it is not possible in view of Proposition \ref{prop:DelPezzosolocurveeccezionali}.
\end{proof}
\begin{cor}
Let $\S\su\mathbb{P}^{3}$ be a cubic surface, then $\S$ is isomorphic to
$\Pd$ blown up in 6 general points.
\end{cor}

\begin{proof}
In view of adjunction formula \cite[Corollary 6.3.16]{Ke}, if
$H\su\mathbb{P}^{3}$ is an hyperplane we have

\[
\ks=\left(\K_{\mathbb{P}^{3}}+\S\right)_{|\S}=\left(-4H+3H\right)_{|\S}=-H_{|\S},
\]
that is $-\ks$ is very ample. Therefore $\S$ is a del Pezzo surface
with $d\left(\S\right)=3$ and we conclude in view of Theorem \ref{thm:DelPezzoclassification}.
\end{proof}

\subsection{Scrolls}
\begin{defn}
An embedded geometrically ruled surface $\S=\pc\left(\mathcal{E}\right)\su\PN$
is called a \emph{scroll }if every fibre $F$ has degree 1.
\end{defn}

\begin{thm}
\label{thm:RATIONALSCROOLembedding}For every $k>n\ge0$ there exists
an embedding of $\ff$ as a rational scroll of degree $2k-n$ in $\mathbb{P}^{2k-n+1}$.
\end{thm}

\begin{proof}
It follows from Theorem \ref{thm:RATIONALSCROLLSveryample} that $|\Lambda_{k}|$
is a very ample linear system on $\ff$, moreover
\[
\Lambda_{k}^{2}=\C_{0}^{2}+2k\C_{0}\cdot F=2k-n.
\]
Since $\Lambda_{k}\cdot F=1$, in view of Proposition \ref{prop:coomologiacommutapushforward}
we have
\[
h^{0}\left(\ff,\mathcal{O}_{\ff}\left(\Lambda_{k}\right)\right)=h^{0}\left(\Pu,\pi_{\ast}\mathcal{O}_{\ff}\left(\Lambda_{k}\right)\right)
\]
and hence, in view of projection formula, we compute
\begin{eqnarray*}
h^{0}\left(\Pu,\pi_{\ast}\mathcal{O}_{\ff}\left(\Lambda_{k}\right)\right) & = & h^{0}\left(\Pu,\pi_{\ast}\left(\mathcal{O}_{\ff}\left(\C_{0}\right)\ot\pi^{\ast}\mathcal{O}_{\Pu}\left(k\right)\right)\right)\\
 & = & h^{0}\left(\Pu,\pi_{\ast}\mathcal{O}_{\ff}\left(\C_{0}\right)\ot\mathcal{O}_{\Pu}\left(k\right)\right)\\
 & = & h^{0}\left(\Pu,\mathcal{O}_{\Pu}\left(k\right)\right)+h^{0}\left(\Pu,\mathcal{O}_{\Pu}\left(k-n\right)\right)\\
 & = & k+1+k-n+1=2k-n+2.
\end{eqnarray*}
\end{proof}
\begin{example}
For $n=0$ and $k=1$ we recover the isomorphism
\[
\mathbb{F}_{0}\iso Q\su\mathbb{P}^{3}
\]
where $Q$ denotes the smooth quadric surface. For $n=1$ and $k=2$,
in view of  Proposition \ref{prop:F1scoppiamento}, we have an embedding
$\bl_{p}\left(\Pd\right)\su\mathbb{P}^{4}$ as a rational scroll of degree
3.
\end{example}

\begin{notation}
Let $\pi :\pc \left(\mathcal{E}\right)\to\C$ be a geometrically ruled surface. For any point $P\in\C$ we denote the fibre $\pi ^{-1} \left( P\right)$ with $F_P$.
\end{notation}
\begin{lem}
\label{lem:ELLIPTICSCROLLbasepointfree}Let $\pi:\S=\pc\left(\mathcal{E}\right)\to\C$
be a geometrically ruled surface over a smooth curve with $\mathcal{E}$ normalized.
Let $D\in\D\left(\C\right)$ and $H\li\C_{0}+\pi^{\ast}D$ where $\C_{0}$
is a section with $\C_{0}^{2}=e\left(\S\right)$. Then $|H|$ is base-point-free
on $F_{P}$ if and only if
\[
h^{0}\left(\S,\os\left(H\right)\right)=h^{0}\left(\S,\os\left(H-F_{p}\right)\right)+2.
\]
In particular in this case the restriction map
\[
H^{0}\left(\S,\os\left(H\right)\right)\overset{r}{\longrightarrow}H^{0}\left(F_{P},\mathcal{O}_{F_{P}}\left(H\right)\right)
\]
is surjective.
\end{lem}

\begin{proof}
Let us consider the exact sequence
\[
0\to\os\left(H-F_{P}\right)\to\os\left(H\right)\to\mathcal{O}_{F_{P}}\left(H\right)\to0
\]
and the exact sequence in cohomology
\[
0\to H^{0}\left(\S,\os\left(H-F_{P}\right)\right)\to H^{0}\left(\S,\os\left(H\right)\right)\overset{r}{\longrightarrow}H^{0}\left(F_{P},O_{F_{P}}\left(H\right)\right).
\]
Since $H\cdot F_{P}=1$ and $F_{P}\iso\Pu$, then $\mathcal{O}_{F_{P}}\left(H\right)\iso\mathcal{O}_{\Pu}\left(1\right)$
and
\[
h^{0}\left(F_{P},\mathcal{O}_{F_{P}}\left(H\right)\right)=2.
\]
Therefore $\rg\left(r\right)=0$ if and only if $F_{P}$ lie in the
fixed part of $|H|$, and $\rg\left(r\right)=1$ if and only if $|H|$
has a base point on $F_{P}$.
\end{proof}
\begin{prop}
\label{prop:ELLIPTICSCROLLisomorfismotranneK}Let $\pi:\S=\pc\left(\mathcal{E}\right)\to\C$
be a geometrically ruled surface over a smooth curve with $\mathcal{E}$ normalized.
Let $D\in\D\left(\C\right)$ and $H\li\C_{0}+\pi^{\ast}D$ where $\C_{0}$
is a section with $\C_{0}^{2}=e\left(\S\right)$. If $|H|$ is base-point-free,
then $\varphi_{|H|}$ is an isomorphism on $\S-K$ where
\[
K=\left\{ x\in\S\,\,|\,\,\mbox{there exists \ensuremath{P\in\C}such that \ensuremath{x}is a base point of \ensuremath{|H-F_{p}|}}\right\} .
\]
\end{prop}

\begin{proof}
We see that $|H|_{|\S-K}$ separates points and tangent directions.
Let $x,y\in\S-K$ be two distinct points and let $t$ be a tangent
vector on $x$. If $x,y$ lie both on $F_{P}$ for $P\in\C$, in view
of Lemma \ref{lem:ELLIPTICSCROLLbasepointfree}, the restriction map
\[
H^{0}\left(\S,\os\left(H\right)\right)\to H^{0}\left(F_{P},\mathcal{O}_{F_{P}}\left(H\right)\right)
\]
is surjective. Then $|H|$ separates $x$ and $y$ since $\mathcal{O}_{F_{P}}\left(H\right)\iso\mathcal{O}_{\Pu}\left(1\right)$
is very ample on $F_{P}$. Let us suppose $x\in F_{P}$ and $y\in F_{Q}$
for $P\neq Q$. Since $x\not\in K$, then it is not a base point for
$|H-F_{Q}|$ and we conclude since elements of $|H-F_{Q}|$ are divisors
which contain $F_{Q}$, there exists $D\in|H|$ which contains $F_{Q}$
(and hence $y$) but not $x$.

Suppose now both $x$ and $t$ lie on $F_{P}$, as before, in view
of Lemma \ref{lem:ELLIPTICSCROLLbasepointfree}, the restriction map
\[
H^{0}\left(\S,\os\left(H\right)\right)\to H^{0}\left(F_{P},\mathcal{O}_{F_{P}}\left(H\right)\right)
\]
is surjective and $\mathcal{O}_{F_{P}}\left(H\right)\iso\mathcal{O}_{\Pu}\left(1\right)$
is very ample on $F_{P}$. Therefore there exists $D\in|H|$ with
$x\in\supp\left(D\right)$ intersecting $F_{P}$ transversally. Let
us suppose that $t$ does not lie on $F_{P}$. Since $x\not\in K$,
then it is not a base point for $|H-F_{P}|$. Let $D'\in|H-F_{P}|$
with $x\not\in\supp\left(D'\right)$, then $x\in\supp\left(D'+F_{P}\right)$
and $T_{x}\left(D\right)=T_{x}\left(F_{P}\right)$, therefore $t$
does not lie on $D$.
\end{proof}
\begin{thm}
\label{thm:ELLIPTICSCROLLisveryample}Let $\pi:\S=\pc\left(\mathcal{E}\right)\to\C$
be a geometrically ruled surface over a smooth curve with $\mathcal{E}$ normalized.
Let $D\in\D\left(\C\right)$ and $H\li\C_{0}+\pi^{\ast}D$ where $\C_{0}$
is a section with $\C_{0}^{2}=e\left(\S\right)$. If
\[
h^{0}\left(\S,\os\left(H-F_{P}-F_{Q}\right)\right)=h^{0}\left(\S,\os\left(H\right)\right)-4
\]
for any $P,Q\in\C$, then $|H|$ is very ample.
\end{thm}

\begin{proof}
Let us suppose $|H|$ is not base-point-free, in view of Lemma \ref{lem:ELLIPTICSCROLLbasepointfree},
there exists $P\in\C$ such that
\[
h^{0}\left(\S,\os\left(H-F_{P}\right)\right)\ge h^{0}\left(\S,\os\left(H\right)\right)-1
\]
and therefore, again in view of Lemma \ref{lem:ELLIPTICSCROLLbasepointfree},
for any point $Q\in\C$ we have
\[
h^{0}\left(\S,\os\left(H-F_{P}-F_{Q}\right)\right)\ge h^{0}\left(\S,\os\left(H\right)\right)-3,
\]
which is a contradiction. Therefore $\varphi_{|H|}$ is a morphism
$\S\to\PN$. Let us suppose it is not an isomorphism in $x\in\S$.
In view of Proposition \ref{prop:ELLIPTICSCROLLisomorfismotranneK},
there exists $P\in\C$ such that $x$ is a base-point of $|H-F_{P}|$.
In view of Lemma \ref{lem:ELLIPTICSCROLLbasepointfree} we would have
\[
h^{0}\left(\S,\os\left(H-F_{P}-F_{Q}\right)\right)\ge h^{0}\left(\S,\os\left(H-F_{P}\right)\right)-1\ge h^{0}\left(\S,\os\left(H\right)\right)-3,
\]
which is again a contradiction.
\end{proof}
\begin{lem}
\label{lem:ELLIPTICSCROLLcoomologiasomma}Let $\pi:\S=\pc\left(\mathcal{E}\right)\to\C$
be a geometrically ruled surface over a smooth curve with $\mathcal{E}$ normalized.
Let $D\in\D\left(\C\right)$ with $H^{1}\left(\C,\oc\left(D\right)\right)=0$
and let $E\in\D\left(C\right)$ such that $\oc\left(E\right)\iso\det\mathcal{E}$.
Then we have
\[
h^{0}\left(\S,\os\left(\C_{0}+\pi^{\ast}D\right)\right)=h^{0}\left(\C,\oc\left(D\right)\right)+h^{0}\left(\C,\oc\left(D+E\right)\right).
\]
\end{lem}

\begin{proof}
Let us consider the exact sequence
\[
0\to\os\left(\pi^{\ast}D\right)\to\os\left(\C_{0}+\pi^{\ast}D\right)\to\mathcal{O}_{\C_{0}}\left(\C_{0}+\pi^{\ast}D\right)\to0.
\]
Looking at the associated cohomology sequence, since
\[
H^{1}\left(\C,\oc\left(D\right)\right)=0
\]
and in view of Proposition \ref{prop:coomologiacommutapushforward},
we conclude that
\[
h^{0}\left(\S,\os\left(\C_{0}+\pi^{\ast}D\right)\right)=h^{0}\left(\C,\oc\left(D\right)\right)+h^{0}\left(\C_{0},\mathcal{O}_{\C_{0}}\left(\C_{0}+\pi^{\ast}D\right)\right).
\]
Moreover we can identify $\C_{0}$ with $\C$ via $\pi$ and, since
$\mathcal{E}$ is normalized, we have an exact sequence
\[
0\to\oc\to\mathcal{E}\to\L\to0
\]
which yields a commutative diagram
\[
\xymatrix{\C_{0}=\mathbb{P}_{\C}\left(\L\right)\ar[rr]^{i}\ar[dr]_{\pi_{|\C_{0}}} &  & \pc\left(\mathcal{E}\right)=\S\ar[dl]^{\pi}\\
 & \C
}
\]
with $i$ embedding such that $i^{\ast}\os\left(1\right)\iso\pi^{\ast}\L$.
Therefore, using our identification and Remark \ref{rem:tautologicalnumericallyequivalent},
we deduce
\[
\L\iso\os\left(1\right)_{|C_{0}}\iso\os\left(\C_{0}\right)_{|\C_{0}}=\mathcal{O}_{\C_{0}}\left(\C_{0}\right).
\]
We conclude since $\L\iso\det\mathcal{E}$.
\end{proof}
\begin{lem}
Let $\pi:\S=\pc\left(\mathcal{E}\right)\to\C$ be a geometrically ruled surface
over a smooth curve with $\mathcal{E}$ normalized. Let $D\in\D\left(\C\right)$
with $H^{1}\left(\C,\oc\left(D\right)\right)=0$ and let $E\in\D\left(C\right)$
such that $\oc\left(E\right)\iso\det\mathcal{E}$. If $|D|$ and $|D+E|$ are
very ample on $\C$, then $|\C_{0}+\pi^{\ast}D|$ is very ample on
$\S$.
\end{lem}

\begin{proof}
Let $P,Q\in\C$ be two non necessarily distinct points. Since $|D|$
is very ample we have
\[
h^{0}\left(\C,\oc\left(D-P-Q\right)\right)=h^{0}\left(\C,\oc\left(D\right)\right)-2
\]
and hence, in view of \cite[Theorem 8.4.1]{Ke}, we conclude that
\[
h^{1}\left(\C,\oc\left(D-P-Q\right)\right)=0.
\]
Let us set $H=\C_{0}+\pi^{\ast}D$, in view of Lemma \ref{lem:ELLIPTICSCROLLcoomologiasomma}
and since $|D|$ and $|D+E|$ are very ample on $\C$ we deduce that
\begin{eqnarray*}
h^{0}\left(\S,\os\left(H-F_{P}-F_{Q}\right)\right) & = & h^{0}\left(\C,\oc\left(D\right)\right)+h^{0}\left(\C,\oc\left(D+E\right)\right)-4\\
 & = & h^{0}\left(\S,\os\left(H\right)\right)-4,
\end{eqnarray*}
therefore we conclude with Theorem \ref{thm:ELLIPTICSCROLLisveryample}.
\end{proof}
\begin{lem}
\label{lem:ELLIPTICSCROLLperdimostrarebasepointfree}Let $\C$ be
an elliptic curve and $\pi:\S=\pc\left(\mathcal{E}\right)\to\C$ be a $\mathbb{P}^{1}$-bundle
with $e\left(\S\right)=-1$. If $F$ is a fibre of $\pi$, then the
linear system $|\C_{0}+F|$ is base-point-free.
\end{lem}

\begin{proof}
Let us consider the exact sequence
\[
0\to\os\left(F\right)\to\os\left(\C_{0}+F\right)\to\mathcal{O}_{\C_{0}}\left(\C_{0}+F\right)\to0.
\]
In view of Proposition \ref{prop:coomologiacommutapushforward} we
deduce that the restriction map
\[
H^{0}\left(\S,\os\left(\C_{0}+F\right)\right)\to H^{0}\left(\C_{0},\mathcal{O}_{\C_{0}}\left(\C_{0}+F\right)\right)
\]
is surjective, therefore $|\C_{0}+F|$ has no base-point on $\C_{0}$.
To see it is base-point-free on $F$, we consider the exact sequence
\[
0\to\os\left(\C_{0}\right)\to\os\left(\C_{0}+F\right)\to\mathcal{O}_{F}\left(\C_{0}+F\right)\to0.
\]
Since $\mathcal{O}_{F}\left(\C_{0}+F\right)\iso\mathcal{O}_{\Pu}\left(1\right)$, it
is enough to see that $H^{1}\left(\S,\os\left(\C_{0}\right)\right)=0$.
In view of Corollary \ref{cor:genusformula} we have
\[
\ks\cdot\C_{0}=2g\left(\C_{0}\right)-2-\C_{0}^{2}=-1,
\]
therefore in view of Riemann-Roch
\[
\x\left(\os\left(\C_{0}\right)\right)=1+\x\left(\os\right).
\]
In view of Corollary \ref{cor:invariantsruled} we conclude that $\x\left(\os\right)=0$,
moreover, in view of Remark \ref{Rem:usefulremhodgeindex} and Serre
duality, if $h^{2}\left(\S,\os\left(\C_{0}\right)\right)>0$, we conclude
that $\dim|\C_{0}|\le\dim|\ks|$ which is not possibile in view of
Corollary \ref{cor:invariantsruled}, therefore
\[
h^{0}\left(\S,\os\left(\C_{0}\right)\right)-h^{1}\left(\S,\os\left(\C_{0}\right)\right)=1.
\]
In view of Remark \ref{rem:tautologicalnumericallyequivalent} we
conclude that $\os\left(\C_{0}\right)\iso\os\left(1\right)$, therefore,
since $\C_{0}\cdot F=1$, in view of Proposition \ref{prop:coomologiacommutapushforward}
we have
\[
h^{0}\left(\S,\os\left(\C_{0}\right)\right)=h^{0}\left(\C,\pi_{\ast}\os\left(1\right)\right)=h^{0}\left(\C,\mathcal{E}\right)>0.
\]
We have the exact sequences
\[
0\to\oc\to\mathcal{E}\to\oc\left(P\right)\to0,
\]
and
\[
0\to\oc\left(-P\right)\to\mathcal{E}^{\ast}\to\oc\to0,
\]
then, in view of \cite[Theorems 8.4.1 and 8.5.4]{Ke}, we deduce
that $\x\left(\mathcal{E}\right)=1$ and
\begin{gather*}
h^{0}\left(\C,\mathcal{E}\right)-h^{0}\left(\C,\mathcal{E}^{\ast}\right)=1,
\end{gather*}
in particular $h^{0}\left(\C,\mathcal{E}\right)=h^{1}\left(\C,\mathcal{E}^{\ast}\right)$.
Let us suppose that $h^{0}\left(\C,\mathcal{E}\right)\ge2$, then $h^{0}\left(\C,\mathcal{E}^{\ast}\right)\ge1$,
therefore there exists a section
\[
0\to\oc\to\mathcal{E}^{\ast},
\]
which is not possible in view of Theorem \ref{thm:possiblevalueseS}
.
\end{proof}
\begin{prop}
\label{prop:ELLIPTICSCROLLmoltoampio-1}Let $\C$ be an elliptic curve
and $\pi:\S=\pc\left(\mathcal{E}\right)\to\C$ be a $\mathbb{P}^{1}$-bundle with $e\left(\S\right)=-1$.
Then the linear system $|\C_{0}+2F|$ is very ample on $\S$.
\end{prop}

\begin{proof}
We have an exact sequence
\[
0\to\os\left(2F\right)\to\os\left(H\right)\to\mathcal{O}_{\C_{0}}\left(H\right)\to0,
\]
in view of Proposition \ref{prop:coomologiacommutapushforward} we
deduce that $H^{1}\left(\S,\os\left(2F\right)\right)=0$ and then
the restriction map
\[
H^{0}\left(\S,\os\left(H\right)\right)\to H^{0}\left(\C_{0},\mathcal{O}_{\C_{0}}\left(H\right)\right)
\]
is surjective, therefore $|H|$ separates points and tangent vectors
on $\C_{0}$ since $|H|_{|\C_{0}}$ is very ample.

Let us consider now the exact sequence
\[
0\to\os\left(\C_{0}+F\right)\to\os\left(H\right)\to\mathcal{O}_{F}\left(H\right)\to0,
\]
in view of Proposition \ref{prop:coomologiacommutapushforward} we
deduce that $H^{1}\left(\S,\os\left(\C_{0}+F\right)\right)=0$ and
then the restriction map
\[
H^{0}\left(\S,\os\left(H\right)\right)\to H^{0}\left(F,\mathcal{O}_{F}\left(H\right)\right)
\]
is surjective, therefore $|H|$ separates points and tangent vectors
on $F$ since $\mathcal{O}_{F}\left(H\right)\iso\mathcal{O}_{\Pu}\left(1\right)$ is
very ample.

Let $P\in\C_{0}$ and $Q\not\in\C_{0}$, we can suppose that $P,Q$
do not lie on the same fibre. If $Q\in F$, then we can take fibres
$F'$ and $F''$ such that
\[
P\in\supp\left(\C_{0}+F'+F''\right)
\]
but
\[
Q\not\in\supp\left(\C_{0}+F'+F''\right).
\]

We are left to the case when $P,Q\not\in\C_{0}$ and they do not lie
on the same fibre: let us suppose $P\in F$. In view of Lemma \ref{lem:ELLIPTICSCROLLperdimostrarebasepointfree},
there exists $D\in|\C_{0}+F|$ such that $P\not\in D$, therefore
taking a fibre $F'$ with $Q\in F'$ we have $D+F'\in|H|$ with
\[
P\not\in\supp\left(D+F'\right)
\]
and
\[
Q\in\supp\left(D+F'\right).
\]
To see that $|\C_{0}+2F|$ separates tangent vectors, we only need
to check the case when $P\in F$ and $t$ is a tangent direction not
lying on $F$, but, in view of Lemma \ref{lem:ELLIPTICSCROLLperdimostrarebasepointfree},
it follows as in the proof of Proposition \ref{prop:ELLIPTICSCROLLisomorfismotranneK}.
\end{proof}
\begin{thm}[{\cite[Exercise V, 2.13]{HA}}]
\label{thm:embeddingelliptic}For every $e\ge-1$ and $n\ge e+3$,
there exists an elliptic scroll of degree $d=2n-e$ in $\mathbb{P}^{d-1}$.
In particular, there exists an elliptic scroll of degree 5 in $\mathbb{P}^{4}$.
\end{thm}

\begin{proof}
Let $\C$ be an elliptic curve and $\pi:\S=\pc\left(\mathcal{E}\right)\to\C$
be a $\mathbb{P}^{1}$-bundle with $\mathcal{E}$ normalized. In view of Theorems \ref{thm:possiblevalueseS}
and \ref{thm:CUBICSCROLLunichepossibilita} we know that all values
$e\left(\S\right)\ge-1$ occur. Let $E\in\D\left(C\right)$ such that
$\oc\left(E\right)\iso\det\mathcal{E}$, we distinguish two cases

\begin{description}
\item [{$e=-1$}] If $e\left(\S\right)=-1$ and $n=2$, let us consider
$H=\C_{0}+2F\in\D\left(\S\right)$, then $H^{2}=5$, moreover, in
view of Lemma \ref{lem:ELLIPTICSCROLLcoomologiasomma}, since $\deg\left(\det\mathcal{E}\right)=1$
we have
\[
h^{0}\left(\S,\os\left(H\right)\right)=5.
\]
We conclude since $H\cdot F=1$ and $|H|$ is very ample in view of
Proposition \ref{prop:ELLIPTICSCROLLmoltoampio-1}. If $n\ge3$, let
$D$ be a divisor on $\C$ with
\[
\deg\left(D\right)=n.
\]
Then $\deg\left(D+E\right)=n-e\ge4$, therefore both $|D|$ and $|D+E|$
are very ample on $\C$, and hence, in view of Lemma \ref{lem:ELLIPTICSCROLLcoomologiasomma},
we have
\begin{eqnarray*}
h^{0}\left(\S,\os\left(\C_{0}+\pi^{\ast}D\right)\right) & = & h^{0}\left(\C,\oc\left(D\right)\right)+h^{0}\left(\C,\oc\left(D+E\right)\right)\\
 & = & 2n+1.
\end{eqnarray*}
We also compute
\[
\left(\C_{0}+\pi^{\ast}D\right)^{2}=\C_{0}^{2}+2\C_{0}\cdot\pi^{\ast}D=2n+1
\]
and $\left(\C_{0}+\pi^{\ast}D\right)\cdot F=1$, therefore $|\C_{0}+\pi^{\ast}D|$
embeds $\S$ in $\mathbb{P}^{2n}$ as an elliptic scroll of degree $2n+1$.
\item [{$e\ge0$}] Let $e\left(\S\right)=e$, $n\ge e+3\ge0$, and $D$
be a divisor on $\C$ with $\deg\left(D\right)=n.$ Then $\deg\left(D+E\right)=n-e\ge3$,
therefore both $|D|$ and $|D+E|$ are very ample on $\C$, and hence,
in view of Lemma \ref{lem:ELLIPTICSCROLLcoomologiasomma}, we have
\begin{eqnarray*}
h^{0}\left(\S,\os\left(\C_{0}+\pi^{\ast}D\right)\right) & = & h^{0}\left(\C,\oc\left(D\right)\right)+h^{0}\left(\C,\oc\left(D+E\right)\right)\\
 & = & 2n-e.
\end{eqnarray*}
We also compute
\[
\left(\C_{0}+\pi^{\ast}D\right)^{2}=\C_{0}^{2}+2\C_{0}\cdot\pi^{\ast}D=2n-e
\]
and $\left(\C_{0}+\pi^{\ast}D\right)\cdot F=1$, therefore $|\C_{0}+\pi^{\ast}D|$
embeds $\S$ in $\mathbb{P}^{2n-e-1}$ as an elliptic scroll of degree $2n-e$.
\end{description}
\end{proof}

\subsection{Surfaces of minimal and almost minimal degree}
\begin{lem}
\label{lem:fondamentalegrado}Let $X\su\mathbb{P}^{d}$ be an irreducible
and non-degenerate variety of dimension $r$. Then
\[
\deg\left(X\right)\ge d-r+1.
\]
\end{lem}

\begin{proof}
We proceed by induction on $d$. We see that the smallest possible
value for $d$ is $r$, but then $\deg\left(X\right)=1$ and the inequality
trivially holds. Let us suppose $d>r$ and let $p\in X$ be a smooth
point. Let $\pi_{p}:X\ra\mathbb{P}^{d-1}$ be the projection from the point
$p$ and let $Y\su\mathbb{P}^{d-1}$ be the closure of its image. If $Y$
were contained in a hypersurface $H$ of $\mathbb{P}^{d-1}$, then the linear
space $\left\langle H,p\right\rangle $ would contain $X$ and this
is not possible in view of our assumptions. Therefore $Y$ is non-degenerate
and, in view of the induction hypothesis, we have
\[
\deg\left(X\right)=\deg\left(\pi_{p}\right)\deg\left(Y\right)+1\ge\deg\left(\pi_{p}\right)\left(d-r\right)+1\ge d-r+1.
\]
\end{proof}
\begin{rem}
\label{rem:Cgradominimo}In view of Lemma \ref{lem:fondamentalegrado},
we see that, if $\C\su\mathbb{P}^{d}$ is a non-degenerate and irreducible
curve, then $\deg\left(\C\right)\ge d.$ If $\deg\left(\C\right)=d$,
as in the proof of Lemma \ref{lem:fondamentalegrado}, we see that
successive projections from smooth points give a birational map onto
an irreducible plane conic, therefore $\C\iso\Pu$.
\end{rem}

\begin{lem}
\label{lem:divisorecurvaellittica}Let $\C$ be a smooth irreducible
curve and $D\in\D\left(\C\right)$ be a divisor such that $\deg\left(D\right)=d\ge3$
and $h^{0}\left(\C,\oc\left(D\right)\right)=d$. Then $g\left(\C\right)=1$.
\end{lem}

\begin{proof}
In view of \cite[Theorem 8.4.1]{Ke} we have
\[
d-h^{1}\left(\C,\oc\left(D\right)\right)=d+1-g\left(\C\right),
\]
therefore $h^{1}\left(\C,\oc\left(D\right)\right)=g-1$ and hence
$g\ge1$. If $g>1$, then $D$ would be a special divisor and hence,
in view of Clifford's inequality,
\[
d-1=h^{0}\left(\C,\oc\left(D\right)\right)-1\le\frac{1}{2}d
\]
which gives $d\le2$, contradiction .
\end{proof}
\begin{lem}
\label{lem:Seesistemoltoampiotaleche}Let $\S$ be a surface and $H$
a very ample divisor such that $H^{2}=d>1$ and $h^{0}\left(\S,\os\left(H\right)\right)=d+2$.
Then $\S$ is rational and
\[
\ks\cdot H=-d-2,
\]
in particular $\ks+H$ is not nef.
\end{lem}

\begin{proof}
In view of \cite[Theorem 6.6.1]{Ke} and Corollary \ref{cor:BertiniII}
there exists a smooth irreducible curve $\C\in|H|$ with $\C^{2}=d>1$.
Let us consider the exact sequence
\[
0\to\os\to\os\left(H\right)\to\oc\left(H\right)\to0,
\]
in view of Theorem \ref{thm:RKW} and Serre duality we have an exact
sequence
\[
0\to\cc\to H^{0}\left(\S,\os\left(H\right)\right)\iso\cc^{d+2}\to H^{0}\left(\C,\oc\left(H\right)\right)
\]
which, under our assumptions, gives $h^{0}\left(\C,\oc\left(H\right)\right)\ge d+1$.
Moreover $\varphi_{|H|}\left(\C\right)\su\mathbb{P}^{d+1}$ is a smooth irreducible
curve of degree $d$, thus, in view of Lemma \ref{lem:fondamentalegrado},
we conclude that $h^{0}\left(\C,\oc\left(H\right)\right)\le d+1$
and hence $h^{0}\left(\C,\oc\left(H\right)\right)=d+1$. Since $\varphi_{|H|}\left(\C\right)$
has degree $d$ in $\mathbb{P}^{d}$, in view of Remark \ref{rem:Cgradominimo},
we conclude that $\C\iso\Pu$ and hence $\S$ is rational in view
of Theorem \ref{thm:equivalence2}. It follows from Corollary \ref{cor:genusformula}
that
\[
-2=d+\ks\cdot H
\]
and then
\[
\left(\ks+H\right)\cdot\C=-2.
\]
\end{proof}
\begin{thm}
\label{thm:minimaldegree}Let $\S\su\mathbb{P}^{d+1}$ be a non-degenerate
surface of degree $d$. Then

\begin{itemize}
\item $d=1$ and $\S=\Pd$;
\item $d=4$ and $\S=\nu_{2}\left(\Pd\right)$;
\item $\S\iso\ff$ is a rational scroll with the embedding given by the
linear system $|\C_{0}+kF|$ where $d=2k-n$.
\end{itemize}
\end{thm}

\begin{proof}
In view of Lemma \ref{lem:Seesistemoltoampiotaleche} we have $t\left(\S\right)\ge2$.

\begin{description}
\item [{$t\left(\S\right)=3$}] In view of Theorem \ref{thm:classificazione}
we have $\S\iso\Pd$ and thus a very ample divisor $H$ on $\Pd$
such that
\[
h^{0}\left(\Pd,\mathcal{O}_{\Pd}\left(H\right)\right)=H^{2}+2.
\]
Let $L\su\Pd$ be a line, then $H\li kL$ for $k>0$ and the above
condition reads
\[
\frac{1}{2}\left(k+1\right)\left(k+2\right)=\binom{k+2}{2}=h^{0}\left(\Pd,\mathcal{O}_{\Pd}\left(k\right)\right)=k^{2}+2.
\]
By elementary algebra we conclude that the only possibilities are
$k=1$ and $k=2$, which correspond to the cases $\S=\Pd$ and $\S=\nu_{2}\left(\Pd\right)$.
\item [{$t\left(\S\right)=2$}] In view of Theorem \ref{thm:classificazione}
and Lemma \ref{lem:Seesistemoltoampiotaleche} we see that $\S$ is
a $\Pu$-bundle over $\Pu$ and then, in view of Remark \ref{rem:FngeomrigatesuPuno},
we have $\S\iso\ff$ for some $n\ge0$. Then $H\li a\C_{0}+bF$ and,
since $H$ is very ample, we have $a>0$ and $b>-an$. In view of
Theorems \ref{thm:RATIONALSCROLLSveryample} and \ref{thm:RATIONALSCROOLembedding}
it is enough to prove that $a=1$. Let us suppose $a\ge2$ then, in
view of Lemma \ref{lem:Seesistemoltoampiotaleche} we have
\begin{eqnarray*}
-an+2ab & = & H^{2}=d>1\\
2an-2b-2-n & = & \ks\cdot H=-2-d.
\end{eqnarray*}
We deduce that
\[
2b+n-2an=2ab-an
\]
and hence
\[
2b-2ab=an-n.
\]
Since $a\ge2$, it follows from elementary algebra that $2b+n=0$,
therefore $b\le0$. We conclude since then
\[
d+an=2ab\le0
\]
and thus $d\le0$ . The second part of the
statement follows from Theorem \ref{thm:RATIONALSCROOLembedding}
and Theorem \ref{thm:RATIONALSCROLLSveryample}.
\end{description}
\end{proof}
\begin{prop}
\label{prop:irregularitygenusellipticscroll}Let $\S\su\mathbb{P}^{d}$ be
a non-degenerate surface and let $\C\in|H|$ be a smooth and irreducible
hyperplane section. If
\[
q\left(\S\right)=g\left(\C\right)=1,
\]
then $S$ is an elliptic scroll.
\end{prop}

\begin{proof}
Since $q\left(\S\right)=1$, in view of Theorem \ref{thm:equivalence2},
we see that $\S$ is not rational and hence, in view of Theorem \ref{thm:classificazione}
we deduce $t\left(\S\right)\neq3$.

In view of Corollary \ref{cor:genusformula} we deduce
\[
\left(\ks+H\right)\cdot H=0,
\]
therefore we have
\[
\ks\cdot H=-H^{2}<0,
\]
then $\ks$ is not nef and, in view of Theorem \ref{thm:equivalence2},
we conclude that $t\left(\S\right)\neq0$.

Let us suppose that $t\left(\S\right)=1$, then $\ks+H$ is nef and,
by Proposition \ref{prop:nefpositivesquare}, we get $\left(\ks+H\right)^{2}\ge0$
and, since $\left(\ks+H\right)\cdot H=0$, in view of Theorem \ref{thm:Hodge}
we conclude that $-\ks\sim H$. In view of Corollary \ref{cor:amplenessnumericalproperty}
$-\ks$ is ample, but this is a contradiction with Corollary \ref{cor:rationalnefandbig}
since we have already observed that $\S$ cannot be rational. Then
$t\left(\S\right)=2$ and, in view of Theorem \ref{thm:classificazione}
and Corollary \ref{cor:invariantsruled}, there exists a surjective
morphism $\pi:S\to B$ where $B$ is an elliptic curve and $\pi$
gives a geometrically ruled surface structure on $\S$.

Let $H\li a\C_{0}+bF$ where $\C_{0}^{2}=-e\left(\S\right)$ and $F$
is an irreducible fibre of $\pi$. Since $H$ is very ample, we have
$a\ge1$ and to conclude it is enough to show that $a=1$. Let us
suppose that $a\ge2$, then
\begin{eqnarray*}
0 & = & \left(\ks+H\right)\cdot H\\
 & = & -ae\left(a-2\right)+b\left(a-2\right)+a\left(b-e\right)\\
 & = & -ae\left(a-1\right)+2b\left(a-1\right),
\end{eqnarray*}
therefore we get
\[
2b=ae.
\]
Since $H\cdot\C_{0}>0$, we deduce that $b>ae$ and hence $ae<0$,
in view of Lemma \ref{lem:CUBICSCROLLinvariante} and Theorem \ref{thm:possiblevalueseS}
we get that $e=-1$. In particular we have
\[
H^{2}=\left(-2b\C_{0}+bF\right)^{2}=4b^{2}-4b^{2}=0,
\]
which is clearly not possibile.
\end{proof}
\begin{prop}
\label{prop:geometricallyruledbound}Let $B$ be an elliptic curve
and $\pi:\S=\mathbb{P}_{B}\left(\mathcal{E}\right)\to B$ be a $\Pu$-bundle. If $H=\C_{0}+bF$
is very ample, then $h^{0}\left(\S,\os\left(H\right)\right)\le H^{2}.$
\end{prop}

\begin{proof}
Let $d=H^{2}$ and $e=e\left(\S\right)$. In view of Corollary \ref{cor:invariantsruled}
and Theorem \ref{thm:equivalence2} we see that $\S$ is not rational,
therefore in view of Lemma \ref{lem:Seesistemoltoampiotaleche} it
is enough to show that $h^{0}\left(\S,\os\left(H\right)\right)\neq d+1$.
Since $H\li\C_{0}+bF$ we deduce that
\[
2b=e+d,
\]
moreover $H\cdot\C_{0}=-e+b=d-b>0$, in particular, in view of Lemma
\ref{lem:CUBICSCROLLinvariante} and Theorem \ref{thm:possiblevalueseS},
\[
d>b>e\ge-1.
\]
We see that $b>0$, otherwise $d=-e=1$, which is not possibile. Let
us consider the exact sequence
\[
0\to\os\left(bF\right)\to\os\left(H\right)\to\mathcal{O}_{\C_{0}}\left(H\right)\to0,
\]
in view of Proposition \ref{prop:coomologiacommutapushforward} we
deduce that
\[
h^{0}\left(\S,\os\left(bF\right)\right)=b\,\,\,\,\mbox{and}\,\,\,\,h^{1}\left(\S,\os\left(bF\right)\right)=0.
\]
Let us suppose that $h^{0}\left(\S,\os\left(H\right)\right)=d+1,$
then the exact cohomology sequence
\[
0\to H^{0}\left(\S,\os\left(bF\right)\right)\to H^{0}\left(\S,\os\left(H\right)\right)\to H^{0}\left(\C_{0},\mathcal{O}_{\C_{0}}\left(H\right)\right)\to0
\]
gives $h^{0}\left(\C_{0},\mathcal{O}_{\C_{0}}\left(H\right)\right)=d+1-b$.
We conclude since $H\cdot\C_{0}=d-b$, therefore, in view of Remark
\ref{rem:Cgradominimo}, we deduce $\C_{0}\iso\mathbb{P}^{1}$ which is a
contradiction .
\end{proof}
\begin{prop}
\label{prop:hyperplaneelliptic}Let $\S$ be a surface and $H\in\D\left(\S\right)$
be a very ample divisor with $H^{2}=d$ and $h^{0}\left(\S,\os\left(H\right)\right)=d+1$.
Then the general element $\C\in|H|$ is an elliptic curve.
\end{prop}

\begin{proof}
In view of \cite[Theorem 6.6.1]{Ke} and Corollary \ref{cor:BertiniII},
the general element $\C\in|H|$ is a smooth irreducible curve. Let
us consider the exact sequence
\[
0\to\os\to\os\left(H\right)\to\oc\left(H\right)\to0,
\]
which gives the exact sequence in cohomology
\[
0\to\cc\to H^{0}\left(\S,\os\left(H\right)\right)\to H^{0}\left(\C,\oc\left(H\right)\right)\to H^{1}\left(\S,\os\right)\to\dots
\]
Since $h^{0}\left(\S,\os\left(H\right)\right)=d+1$, we deduce that
$h^{0}\left(\C,\oc\left(H\right)\right)\ge d$. Moreover $H\cdot\C=H^{2}=d$,
therefore, in view of Lemma \ref{lem:fondamentalegrado}
\[
d+1=H\cdot\C+1\ge h^{0}\left(\C,\oc\left(H\right)\right)\ge d.
\]
Let us suppose $h^{0}\left(\C,\oc\left(H\right)\right)=d+1$, then
$|H|_{|\C}$ embeds $\C$ as a non-degenerate curve of degree $d$
in $\mathbb{P}^{d}$, therefore, in view of Remark \ref{rem:Cgradominimo},
we deduce $\C\iso\Pu$. Since $C^{2}=d>0$, in view of Theorem \ref{thm:equivalence2}
we deduce that $h^{1}\left(\S,\os\right)=0$, then we have an exact
sequence
\[
0\to\cc\to\cc^{d+1}\to\cc^{d+1}\to0,
\]
which is clearly impossible. Hence we have
\[
h^{0}\left(\C,\oc\left(H\right)\right)=d
\]
and, in view of Lemma \ref{lem:divisorecurvaellittica}, we conclude
that $\C$ is an elliptic curve.
\end{proof}
\begin{thm}
\label{thm:almostminimaldegree}Let $\S$ be a surface and suppose
there exists $H$ very ample with $H^{2}=d$ and $h^{0}\left(\S,\os\left(H\right)\right)=d+1$.
Then $\S$ is a del Pezzo surface with $-\ks\li H$, where $H$ is
an hyperplane section. In particular $d\le9$.
\end{thm}

\begin{proof}
Let $H$ be an hyperplane section, in view of Proposition \ref{prop:hyperplaneelliptic}
and Corollary \ref{cor:genusformula} we deduce that
\[
\left(\ks+H\right)\cdot H=0.
\]
Let $\C\in|H|$ be a smooth irreducible curve. In view of Proposition
\ref{prop:hyperplaneelliptic}, the curve $\C$ is elliptic, therefore,
in view of adjunction formula \cite[Corollary 6.3.16]{Ke} we have
\[
\oc\left(\ks+H\right)\iso\oc,
\]
then we have the exact sequence
\[
0\to\os\left(\ks\right)\to\os\left(\ks+H\right)\to\oc\to0.
\]
In view of Theorem \ref{thm:RKW}, taking cohomology we have an exact
sequence

\begin{tikzpicture}[descr/.style={fill=white,inner sep=1.5pt}]
\matrix (m) [
matrix of math nodes,
row sep=1em,
column sep=2.5em,
text height=1.5ex, text depth=0.25ex
]
{
0 & H^0(\S ,\os\left(\ks\right)) & H^0(\S ,\os\left(\ks +H\right)) & H^0(\C , \oc) \\
& H^1(\S ,\os\left(\ks\right)) & 0, &  \\
};

\path[overlay,->, font=\scriptsize,>=latex]
(m-1-1) edge (m-1-2)
(m-1-2) edge (m-1-3)
(m-1-3) edge (m-1-4)
(m-1-4) edge[out=355,in=175] (m-2-2)
(m-2-2) edge (m-2-3) ;
\end{tikzpicture}

from which, using Serre duality, we have that $q\left(\S\right)\le1$.
It follows from Propositions \ref{prop:irregularitygenusellipticscroll}
and \ref{prop:geometricallyruledbound} that the case $q\left(\S\right)=1$
cannot occur, therefore we conclude that $h^{1}\left(\S,\os\right)=0$
and that
\[
h^{0}\left(\S,\os\left(\ks+H\right)\right)=1,
\]
since $\left(\ks+H\right)\cdot H=0$, this gives $-\ks\li H$, therefore
$\S$ is a del Pezzo surface embedded by the anticanonical system
which is very ample. The last statement follows from Lemma \ref{lem:DelPezzogradomassimo}.
\end{proof}

\newpage
\nocite{*}
\bibliographystyle{alpha}
\bibliography{referencesMinimalModel}

@book{Beau,
  author    = {Arnaud Beauville},
  title     = {Complex Algebraic Surfaces},
  publisher = {Cambridge University Press},
  year      = {1996}
}

@article{FP,
  author    = {Luis Fuentes-Garcia and Manuel Pedreira},
  title     = {The Projective Theory of Ruled Surfaces},
  journal   = {Note di Matematica},
  volume    = {24},
  number    = {1},
  pages     = {25--63},
  year      = {2005}
}

@book{GH,
  author    = {Phillip Griffiths and Joseph Harris},
  title     = {Principles of Algebraic Geometry},
  publisher = {Wiley-Interscience},
  year      = {1994}
}

@book{Ha,
  author    = {Robin Hartshorne},
  title     = {Algebraic Geometry},
  publisher = {Springer},
  year      = {1977}
}

@book{La,
  author    = {Robert K. Lazarsfeld},
  title     = {Positivity in Algebraic Geometry II: Positivity for Vector Bundles, and Multiplier Ideals},
  publisher = {Springer},
  year      = {2004}
}

@book{Ke,
  author    = {George R. Kempf},
  title     = {Algebraic Varieties},
  series    = {London Mathematical Society Lecture Note Series},
  volume    = {172},
  publisher = {Cambridge University Press},
  year      = {1993}
}

\end{document}